%% file: SansGroupFall2016J7arXiv.tex
\documentclass[reqno]{amsart}

\newcommand{\cA}{\mathcal{A}}\newcommand{\cB}{\mathcal{B}}
\newcommand{\cC}{\mathcal{C}}\newcommand{\cD}{\mathcal{D}}
\newcommand{\cE}{\mathcal{E}}\newcommand{\cF}{\mathcal{F}}
\newcommand{\cG}{\mathcal{G}}\newcommand{\cH}{\mathcal{H}}

\newcommand{\cL}{\mathcal{L}}

\newcommand{\cO}{\mathcal{O}}

\newcommand{\cS}{\mathcal{S}}\newcommand{\cT}{\mathcal{T}}

\newcommand{\cX}{\mathcal{X}}

\newcommand{\C}{\mathbb{C}}

\newcommand{\bC}{\mathbb{C}}
\newcommand{\bE}{\mathbb{E}}\newcommand{\bF}{\mathbb{F}}

\newcommand{\bN}{\mathbb{N}}

\newcommand{\bQ}{\mathbb{Q}}\newcommand{\bR}{\mathbb{R}}

\newcommand{\bZ}{\mathbb{Z}}
\newcommand{\RS}{\Sigma^n_+}

\newcommand{\Div}{\mathbb{\mr{Div}}}
\newcommand{\Sym}{\mathbb{\mr{Sym}}}
\newcommand{\bk}{\mathbf{k}}
\newcommand{\be}{\mathbf{e}}
\newcommand{\br}{\mathbf{r}}
\newcommand{\bj}{\mathbf{j}}
\newcommand{\bc}{\mathbf{c}}

\newcommand{\N}{\mathbb{N}}
\newcommand{\R}{\mathbb{R}}

\newcommand{\Z}{\mathbb{Z}}

\newcommand{\m}{\to}

\usepackage[dvipsnames,svgnames,x11names,hyperref]{xcolor}
\usepackage{import,mathtools,url,graphicx,verbatim,amssymb,enumerate,stmaryrd,amsthm,nicefrac}
\usepackage[pagebackref,colorlinks,citecolor=Mahogany,linkcolor=Mahogany,urlcolor=Mahogany,filecolor=Mahogany]{hyperref}
\usepackage{microtype}
\usepackage[all,cmtip]{xy}
\usepackage[margin=1.5in]{geometry}
\usepackage{setspace}
\newtheorem{theorem}{Theorem}[section]
\newtheorem{lemma}[theorem]{Lemma}
\newtheorem{proposition}[theorem]{Proposition}
\newtheorem{corollary}[theorem]{Corollary}
\newtheorem*{corollary*}{Corollary}
\newtheorem*{theorem*}{Theorem}

\theoremstyle{definition}
\newtheorem{definition}[theorem]{Definition}
\newtheorem{example}[theorem]{Example}
\newtheorem{remark}[theorem]{Remark}

\newcommand{\cat}[1]{\mathsf{#1}}
\newcommand{\mr}[1]{{\rm #1}}
\newcommand{\fS}{\mathfrak{S}}
\newcommand{\fD}{\mathbf{D}_\theta}

\newcommand{\cdt}{\cE^\theta_n}
\newcommand{\cdtf}{\cE^\mr{SO}_n}
\newcommand{\cdtt}{\cE^{\mr{pt}}_n}

\newcommand{\bdt}{\mathbf{E}^\theta_n}
\newcommand{\bdtc}{\mathbf{E}^{\theta,C}_n}
\newcommand{\bdtt}{\mathbf{E}^\mr{pt}_n}

\title{$E_n$-cell attachments and a local-to-global principle for homological stability}
\author{Alexander Kupers}

\thanks{Alexander Kupers is supported by a William R. Hewlett Stanford Graduate Fellowship, Department of Mathematics, Stanford University, and was partially supported by NSF grant DMS-1105058.}
\address{Institut for Matematiske Fag, K\o benhavns Universitet \\ Universitetsparken 5 \\ 2100 K\o benhavn \O, Denmark}
\email{kupers@math.ku.dk}
\author{Jeremy Miller} \address{Department of Mathematics, Purdue University  \\ 150 N. University Street \\ 47907-2067, West Lafayette, IN, USA}\email{jeremykmiller@purdue.edu}

\date{\today}

\begin{document}

\begin{abstract}We define  bounded generation for $E_n$-algebras in chain complexes and prove that for $n \geq 2$ this property is equivalent to homological stability. Using this we prove a local-to-global principle for homological stability, which says that if an $E_n$-algebra $A$ has homological stability (or equivalently the topological chiral homology $\int_{\bR^n} A$ has homology stability), then so has the topological chiral homology $\int_M A$ of any connected non-compact manifold $M$. Using scanning, we reformulate the local-to-global homological stability principle so that it applies to compact manifolds. We also give several applications of our results.\end{abstract}

\maketitle

\tableofcontents

\section{Introduction} 

This paper establishes a local-to-global principle for homological stability. Since all sufficiently natural and local constructions on the category of $n$-manifolds are equivalent to topological chiral homology with coefficients in some $E_n$-algebra, our local-to-global homological stability principle is formulated using topological chiral homology. To prove our result, we introduce a condition on $E_n$-algebras which we call \emph{bounded generation} and prove that this condition is equivalent to homological stability. This allows us to leverage homological stability for configuration spaces to prove that if an $E_n$-algebra has homological stability, then so does its topological chiral homology. We discuss several applications of this local-to-global homological stability principle as well as applications of the equivalence between bounded generation and homological stability.

\subsection{Homological stability for framed $E_n$-algebras}

When working in categories with some notion of homotopy theory, e.g. chain complexes or topological spaces, one can talk about algebras with commutativity conditions interpolating between associative and commutative; these are $E_n$-algebras. The typical example of an $E_n$-algebra is an $n$-fold based loop space $\Omega^n X$. A framed $E_n$-algebra is an $E_n$-algebra with a compatible action of the special orthogonal group $SO(n)$. See Section \ref{suboperads} for precise definitions.

Many framed $E_n$-algebras of interest in the category of spaces have $\pi_0$ isomorphic to $\bN_0$, the non-negative integers, and one can compare the homology of different components. If the $i$th homology group of the $k$th component eventually becomes independent of $k$, then the framed $E_n$-algebra is said to have \textit{homological stability}, see Definition \ref{defhomstab}. This is the property of framed $E_n$-algebras we are interested in.

The basic example of a framed $E_n$-algebra that has homological stability is the configuration space of unordered particles in an $n$-dimensional Euclidean space, $X = \bigsqcup_{k \geq 0} C_k(\bR^n)$ \cite{Ar} \cite[Appendix A]{Se}. Here $C_k(\bR^n)$ is defined as $\{(x_1,\ldots,x_k) \in (\R^n)^k \,|\,x_i \neq x_j \text{ if $i \neq j$}\}/\fS_k$ with $\fS_k$ the symmetric group on $k$ letters. Other examples of framed $E_n$-algebras with homological stability include symmetric powers of $\bR^n$ \cite{Srod}, bounded symmetric powers of $\bR^n$ \cite{Y} \cite{kupersmillertran}, various decorated configuration spaces \cite{RW}, completions of certain partial $E_n$-algebras \cite{kupersmillercompletions}, some spaces of branched covers \cite{EVW}, classifying spaces of groups of diffeomorphisms fixing a disk \cite{harerstab} \cite{wahlmcg} \cite{sorenoscarstability}, moduli spaces of manifolds embedded in $\bR^n$ \cite{palmerthesis}, moduli spaces of instantons \cite{instantonmoduli} and spaces of rational or holomorphic functions \cite{Se} \cite{gravesen} \cite{BMHM} \cite{Gu}. In all of these examples, the map eventually inducing isomorphisms on homology is constructed using the framed $E_n$-algebra structure.

We believe the result of this paper can be put into an $\infty$-categorical context. However, we chose not to do this because it is not strictly necessary and would make the paper less accessible. As a convention, all our manifolds are smooth (but see Remark \ref{remgener}).

\subsection{Homological stability for topological chiral homology}

One can use a framed $E_n$-algebra as coefficients for a homology theory on oriented $n$-dimensional manifolds, which is called \textit{topological chiral homology}. The input is a framed $E_n$-algebra $A$ in a symmetric monoidal $(\infty,1)$-category $\cat{C}$ (for us spaces or chain complexes) and an oriented $n$-dimensional manifold $M$, and the output is an object $\int_M A$ of $\cat{C}$. It is also known as \emph{factorization homology}, \emph{higher Hochschild homology} or \emph{configuration spaces with summable labels}. References for topological chiral homology include \cite{An}, \cite{Fr2}, \cite{ginottradlerzeinalian}, \cite{lurieha} and \cite{Sa}.

Many of the examples of framed $E_n$-algebras with homological stability mentioned in the previous subsection are obtained by applying a geometric construction to $\R^n$ and have natural analogues replacing $\R^n$ by an arbitrary $n$-dimensional manifold $M$. In many cases, the result of this geometric construction is weakly equivalent to the topological chiral homology of the manifold with coefficients in the framed $E_n$-algebra. For example, we have $\bigsqcup_{k \geq 0} C_k(M) \simeq \int_M \bigsqcup_{k \geq 0} C_k(\bR^n)$. See Example \ref{exampsym} and Sections \ref{secapplications} for other examples.

If $M$ is a connected manifold and $X$ is a framed $E_n$-algebra in spaces, then there is an isomorphism $\pi_0(X) \cong \pi_0(\int_M X)$. This leads us to the following question, originally posed by Ralph Cohen: \begin{quote}If the connected components of a framed $E_n$-algebra in spaces have homological stability, is the same true for its topological chiral homology on an oriented manifold?\end{quote} The answer to the question turns out to be: Yes, as long as we restrict to non-compact connected manifolds (see Corollary \ref{cortopstab}). The assumption that the manifold is non-compact is used to construct maps $t$ between components of the topological chiral homology by ``bringing particles in from infinity.'' However, see Section \ref{scanningintro} for a reformulation that applies to compact manifolds. 


\subsection{Chain complexes and charge}

The group of connected components of an $E_n$-algebra in topological spaces inherits a natural monoid structure. In this paper, we will often work in the category of non-negatively graded chain complexes because we are interested in statements about homology. However, in the category of chain complexes one cannot define connected components and this makes it harder to formulate homological stability. Our solution is to work with chain complexes with an extra grading which keeps track of the ``connected components.'' We will call this extra grading \textit{charge} to differentiate it from the homological grading. In examples closely related to configuration spaces, it should be thought of as the number of particles counted with multiplicity. To make this precise, we define a \textit{charged space} to be a space $X$ together with a decomposition $X = \bigsqcup_{c \geq 0}X(c)$, and a \textit{charged chain complex} to be a chain complex $A$ with a decomposition $A = \bigoplus_{c \geq 0} A(c)$. If $X$ is a charged space with each $X(c)$ connected, then a \textit{charged algebra} structure on $X$ is a framed $E_n$-algebra structure on $X$ which respects the decomposition. That is, we require the algebra structure maps to restrict to maps: \[E_n(m) \times X(c_1) \times \ldots X(c_m) \m X(c_1+\ldots+c_m)\] One can similarly define \textit{charged algebras} in the category of chain complexes. For precise definitions see Subsection \ref{subsubcharged}.  


\subsection{$E_n$-cell decompositions}

Before we state our main results, we describe framed $E_n$-cell attachments which are the primary technical tool of this paper. The class of \textit{cellular} framed $E_n$-algebras consists of those framed $E_n$-algebras that are built by successive framed $E_n$-cell attachments. Every charged algebra is weakly equivalent to a cellular algebra. This is similar to the fact that every topological space is weakly homotopy equivalent to a CW-complex. Studying framed $E_n$-cell decompositions allows us to characterize which charged algebras have homological stability and bound the homological stability range. 

What are framed $E_n$-cell attachments? While we will primarily be interested in framed $E_n$-cell attachments in the category of chain complexes, the construction can also be defined for spaces. In that case the construction is more familiar: one can attach a cell (in the sense of CW-complexes) to a framed $E_n$-algebra in spaces, but the resulting space does not naturally carry the structure of framed $E_n$-algebra. For example, we do not know how to multiply elements in the interior of the cell. However, it is naturally a partial framed $E_n$-algebra and by freely adding those operations that are not defined yet, every partial framed $E_n$-algebra can be completed to a framed $E_n$-algebra. This is a framed $E_n$-cell attachment in spaces and a similar construction works in the category of chain complexes, see Section \ref{secencell}.

\subsection{Main result}

We can now state the main result of this paper. A charged algebra is called \emph{bounded generated} if it is weakly equivalent to a cellular framed $E_n$-algebra obtained by $E_n$-cell attachments where all $E_n$-cells of a fixed dimension are attached in only finitely many charges.

\begin{theorem}\label{thmmain} Let $n \geq 2$ and $A$ be a charged algebra in chain complexes, then the following are equivalent:
\begin{enumerate}[(i)]
\item The charged algebra $A$ has homological stability as in Definition \ref{defhomstab}.
\item For all oriented connected non-compact $n$-dimensional manifolds $M$, $\int_M A$ has homological stability as in Definition \ref{defhomstabtch}.
\item The charged algebra $A$ is  bounded generated as in Definition \ref{defboundedgeneration}.
\end{enumerate}\end{theorem}

As $\int_{\R^n} A \simeq A$, the equivalence of (i) and (ii) can be captured by the slogan: ``topological chiral homology has homological stability globally if and only if it has homological stability locally.'' Also note that this equivalence between (i) and (ii) holds trivially for $n=1$, since each  non-compact connected $1$-dimensional manifold is diffeomorphic to $\bR$ and $\int_\bR A \simeq A$. 

Our proof of Theorem \ref{thmmain} is a simultaneous induction involving (i), (ii) and (iii). This shows that bounded generation is a useful concept even if one is only interested in homological stability and reinforces the idea that topological chiral homology is a useful construction even if one is only interested in the properties of $E_n$-algebras. We also remark that the use of condition (iii) to prove (ii) can be summarized by saying that we resolve topological chiral homology in the algebra variable, not in the manifold variable as is traditionally done in homological stability arguments. We note that  our argument does use homological stability for configuration spaces as input (which can be proven using more traditional techniques). See Theorem \ref{thmmainrange} for a version of Theorem \ref{thmmain} involving homological stability with an explicit range. 

\begin{remark}\label{remgener}Using the same method one can give various generalizations of the result. Two of these will be treated in this paper because of their relevance to applications:
\begin{itemize}
	\item We can replace the monoid $\N_0$ used to define charged algebras with any partial abelian monoid of charges $C$ (see Definition \ref{defcharge}). As we do not know of interesting applications of this level of generality, in this paper we only prove our theorems in the case where the monoid is $\N_0^d$ with addition. 
 	\item One can replace orientations with other tangential structures $\theta$, e.g. consider framed $n$-manifolds and $E_n$-algebras. We prove a version of Theorem \ref{thmmain} for any tangential structure with the property that the corresponding space of $\theta$-framed embeddings of $\bR^n$ into itself is connected (this includes framed manifolds, but not unoriented manifolds). It would be interesting to know if this connectivity assumption can be removed.



\end{itemize}

The following is a non-exhaustive list of generalizations that we will not discuss in this paper.
\begin{itemize}
	\item One can generalize to other types of manifolds, e.g. topological manifolds and the corresponding of version of framed $E_n$-algebras. This requires no change except replacing $O(n)$ with $\mr{Top}(n)$ and the tangent bundle with the tangent microbundle, as well as using ideas from \cite{kupersmillerimprov} to construct stabilization maps.
 	\item One can replace the target category with chain complexes over any ring or more generally with the positive part $\cC_{\geq 0}$ of a stable symmetric-monoidal $(\infty,1)$-category $\cC$ with a compatible $t$-structure. Examples include spectra and module spectra over an $E_\infty$-ring spectrum.
\end{itemize}\end{remark}

Theorem \ref{thmmain} has the local-to-global principle for homological stability as an easy corollary:

\begin{corollary}\label{cortopstab} Suppose that $X$ is a charged algebra in spaces, then $X \simeq \int_{\R^n} X$ has homological stability if and only if $\int_M X$ has homological stability for all oriented connected non-compact $n$-dimensional manifolds $M$.
\end{corollary}


Note that to prove this theorem, our methods require that we work in a stable category, like chain complexes. Thus, even though we are primarily interested in $E_n$-algebras in spaces, it is helpful to also consider $E_n$-algebras in chain complexes.


The above results give us two new techniques for proving homological stability theorems. The first involves proving homological stability locally and then applying the local-to-global homological stability principle. In Subsection \ref{subseclocalglobal} we will give several examples of this. Here, we illustrate it by giving a new proof of Steenrod's result that symmetric powers exhibit homological stability \cite{Srod}.

\begin{example}\label{exampsym} Let $\mr{Sym}_k(M)$ denote the quotient space $M^k/\fS_k$ with the symmetric group $\fS_k$ acting by permuting the terms. We have that 
\[\bigsqcup_{k \geq 0} \mr{Sym}_k(M) \simeq \int_M \left(\bigsqcup_{k \geq 0} \mr{Sym}_{k}(\R^n)\right)\]
and since $\mr{Sym}_k(\R^n)$ is contractible for all $k$ and $n$, trivially the framed $E_n$-algebra $\bigsqcup_{k \geq 0} \mr{Sym}_{k}(\R^n)$  has homological stability. Corollary \ref{cortopstab} implies that the spaces $\mr{Sym}_k(M)$ have homological stability whenever $M$ is an oriented connected non-compact manifold of dimension $n \geq 2$. As symmetric powers are homotopy invariant, this actually proves homological stability for symmetric powers of any finite connected CW-complex by embedding it in some Euclidean space $\bR^n$ and taking a regular open neighborhood. This argument extends to ENR's by definition of an ENR and to arbitrary CW-complexes by exhaustion using finite subcomplexes.
\end{example}

The second new technique involves deducing homological stability from the existence of bounded $E_n$-cell decompositions. In Subsection \ref{subseccells}, we give several examples of this, including a large improvement over previously known homological stability ranges for bounded symmetric powers.

\subsection{Stable homology and compact manifolds} \label{scanningintro} After discussing the main theorem, two questions remain. What is the stable homology and what happens for compact manifolds? For $E_n$-algebras in spaces both questions are answered by the  \emph{scanning map}
\[s \colon \int^k_M X \to \Gamma^c_k(M,B^{TM} X)\]
Here $\int^k_M X$ denotes the charge $k$ component of $\int_M X$, $B^{TM} X$ is a bundle over $M$ with fiber given by the $n$-fold delooping $B^n X$ of $X$, and $\Gamma^c_k(-)$ denotes the space of compactly supported sections of degree $k$. In \cite{Mi2}, the second author proved that if $X$ has homological stability and $M$ is connected and non-compact, then the scanning map is a homology equivalence in a range increasing with $k$. When $M$ is non-compact, all the components of $\Gamma^c(M,B^{TM} X)$ are homotopy equivalent and thus the stable homology of $\int^k_M X$ is equal to $H_*(\Gamma^c_0(M,B^{TM} X))$  (see Theorem \ref{propstablehomology}).

In Theorem \ref{thmcompact}, we show that when $X$ has homological stability, the scanning map is a homology equivalence in a range even for compact manifolds (for technical reasons, we only give the proof in the case that the manifold is framed). However, for compact $M$ it is no longer the case that the components of $\Gamma^c(M,B^{TM} X)$ are homotopy equivalent and thus there is no such thing as the stable homology. This reflects the fact that $H_*(\int^k_M X)$ need not stabilize when $M$ is compact, even if the components of $X$ have homological stability (see Page 467 of \cite{Ch}).

These results show the importance of $B^n X$. We show in Proposition \ref{cellofBn} that attaching an $E_n$-cell of dimension $N$ to an $E_n$-algebra $X$ amounts to attaching an ordinary cell of dimension $n+N$ to $B^n X$. Similarly, when one attaches a framed $E_n$-cell of dimension $N$ to a framed $E_n$-algebra, it has the effect of attaching a free $SO(n)$-cell of dimension $n+N$ to $B^n X$.


\subsection{Acknowledgements} The authors would like to thank Ricardo Andrade, Kerstin Baer, Ralph Cohen, S\o ren Galatius, Martin Palmer and the anonymous referee for helpful conversations and comments.

\section{Topological chiral homology and completions of partial algebras} \label{sectch} 
In this section we define an operad $\cdt$ and topological chiral homology for partial $\cdt$-algebras. We also establish the existence of a useful spectral sequence and discuss the notion of homological stability. In contrast to the introduction, we will work with general tangential structures and partial monoids of charges.


\subsection{Operads, algebras and simplicial objects} We start with recalling some general definitions and results.

\subsubsection{Operads, monads and algebras} Operads are a general framework to encode algebraic structures and it hence is no surprise that we will use an operad to define the algebras of interest. An operad is a sequence of objects encoding for all integers $k \geq 0$ the $k$-ary operations in the algebraic structure of interest. Relevant references are \cite{M,fressebook}.

Let $(\cat{C},\otimes,1)$ be a symmetric monoidal category with all small colimits, small limits and such that $\otimes$ preserves colimits. A (reduced) operad in $\cat{C}$ is a sequence of objects $\{\cO(k)\}_{k \geq 0}$ such that $\cO(0) = 1$ with the following additional data: (i) an $\fS_k$-action on $\cO(k)$, where $\fS_k$ denotes the symmetric group on $k$ letters, (ii) operad composition morphisms $\cO(n) \otimes \cO(k_1) \otimes \ldots \otimes \cO(k_n) \to \cO(k_1+\ldots+k_n)$, and (iii) a unit map $1 \to \cO(1)$. These should satisfy appropriate unit, associativity and equivariance axioms.

The definition of an operad can be rephrased in terms of \emph{symmetric sequences}. A symmetric sequence $\cF$ is a sequence $\{\cF(k)\}_{k \geq 0}$ of objects $\cF(k)$ with $\fS_k$-action. For $X$ an object with a right $G$-action for some discrete group $G$ and $Y$ an object with a left $G$ action, let $X\otimes_{G}Y$ denote the quotient of $X \otimes Y$ by the diagonal action of $G$. Then there is a monoidal structure on the category of symmetric sequences with $\otimes$ given by
	\[(\cE \otimes \cF)(k) = \bigsqcup_{n} E(n) \otimes_{\fS_n} \left(\bigsqcup_{k_1+\ldots+k_n = k} \cF(k_1) \otimes \cdots \otimes \cF(k_n) \right)\]
An operad is the same as a unital monoid in symmetric sequences with respect to this monoidal structure which satisfies $\cF(0) = 1$.

For any symmetric sequence $\cF$, one can construct a functor $\mathbf{F} \colon \cat{C} \to \cat{C}$ as follows.

\begin{definition}Let $\cF$ be an symmetric sequence in $\cat{C}$. Then the functor $\mathbf{F} \colon \cat{C} \to \cat{C}$ is given by
	\[C \mapsto \bigsqcup_{n \geq 0} \cF(k) \otimes_{\fS_k} C^{\otimes k}\]
\end{definition}

An object $X$ of $\cat{C}$ can be viewed as a symmetric sequence $\cX$ with $\cX(0)=X$ and $\cX(i)$ an initial object of $\cat{C}$ for $i>0$. Given a symmetric sequence $\cF$, we have that $\cF \otimes \cX(0)=\mathbf F X$, and $\cF \otimes \cX(i)$ is an initial object for $i>0$. When $\cF$ is a unital monoid in symmetric sequences, $\mathbf{F}$ will be a unital monoid in functors:

\begin{definition}
A \emph{monad} in $\cat{C}$ is a unital monoid in the category of functors $\cat{C} \to \cat{C}$. That is, it is a functor $\mathbf{O}$ with a unit natural transformation $1\colon \mr{id} \to \mathbf{O}$ and a composition natural transformation $c\colon \mathbf{O}^2 \to \mathbf{O}$. These are required to satisfy associativity and unit axioms in the sense that the following diagrams commute:
\[\xymatrix{\mathbf{OOO} \ar[r]^{\mathbf{O}(c)} \ar[d]_c & \mathbf{OO} \ar[d]^c \\
	\mathbf{OO} \ar[r]_c & \mathbf{O}} \qquad \qquad \xymatrix{\mathbf{O} \ar@{=}[rd] \ar[r]^{\mathbf{O}(1)} & \mathbf{OO} \ar[d]^c \\
	& \mathbf{O}} \qquad \qquad \xymatrix{\mathbf{O} \ar@{=}[rd] \ar[r]^{1} & \mathbf{OO} \ar[d]^c \\
	& \mathbf{O}}\]
\end{definition}

\begin{proposition} Let $\cO$ be an operad and let $\mathbf{O}$ be the functor associated to $\cO$, viewed as a symmetric sequence. The operad structure on $\cO$ endows $\mathbf{O}$ with the structure of a monad. 
\end{proposition}


One can then define left modules and right modules over an operad $\cO$ as left modules and right modules in the category of symmetric sequences over the symmetric sequence $\{\cO(k)\}$. 

An algebra structure on an object $A$ is a left $\cO$-module structure on the associated symmetric sequence $\cA$. One can also define the notion of an algebra over an operad using the associated monad. We will take this point of view in this paper.

\begin{definition}An \emph{$\cO$-algebra structure} on an object $A$ of $\cat{C}$ is a morphism $a\colon \mathbf{O}(A) \to A$ satisfying associativity and unit axioms in the sense that the following two diagrams commute:
	\[\xymatrix{\mathbf{OO}(A) \ar[r]^{\mathbf{O}(a)} \ar[d]_c & \mathbf{O}(A) \ar[d]^a \\
		\mathbf{O}(A) \ar[r]_a & A} \qquad \qquad \xymatrix{A \ar@{=}[rd] \ar[r]^1 & \mathbf{O}(A) \ar[d]^a \\
		& A}\]
	An \emph{$\cO$-algebra} in $\cat{C}$ is an object with a fixed structure of an $\cO$-algebra.\end{definition}

Note that for any object $X$, $\mathbf{O} X$ is an $\cO$-algebra. Moreover, it is free, as $\mathbf{O}$ is the left adjoint to the forgetful functor from the category of $\cO$-algebras in $\cat{C}$ to $\cat{C}$. 

A right $\cO$-module structure on a symmetric sequence $\cF$ gives rise to a right functor structure on $\mathbf{F}$ over the monad $\mathbf{O}$, a notion which is defined as follows:

\begin{definition}A \emph{right functor} $\mathbf{F}$ over a monad $\mathbf{O}$ is a functor $\mathbf{F}\colon \cat{C} \to \cat{C}$ with a natural transformation $\colon \mathbf{F} \mathbf{O} \to \mathbf{F}$ satisfying associativity and unit axioms in the sense that the following two diagrams commute:
	\[\xymatrix{\mathbf{FOO} \ar[r]^{a} \ar[d]_{\mathbf{F}(c)} & \mathbf{FO} \ar[d]^a \\
		\mathbf{FO} \ar[r]_a & \mathbf{F}} \qquad \qquad \xymatrix{\mathbf{F} \ar@{=}[rd] \ar[r]^{\mathbf{F}(1)} & \mathbf{FO} \ar[d]^a \\
		& \mathbf{F}}\]\end{definition}

One reason to spell out this definition is that we will occasionally be interested in some right functors which do not come from right modules, such as the $n$-fold suspension functor viewed as a right functor over the little $n$-disks operad.

\subsubsection{Categories copowered over $\cat{Top}$} We next describe the contexts in which we will work the remainder of this paper. Let $\cat{C}$ be a symmetric monoidal category as before, i.e. with all colimits and small limits, and $\otimes$ preserving colimits. Further suppose $\cat{C}$ has a lax-monoidal copowering (also known as tensoring) over $\cat{Top}$, i.e. there is a functor $\odot\colon \cat{Top} \times \cat{C} \to \cat{C}$ and a natural transformation $\alpha\colon X \odot (Y \odot C) \to (X \times Y) \odot C$ satisfying associativity and unit axioms. We will only consider the following examples:
\begin{enumerate}[(i)]
	\item Topological spaces: $\cat{C} = \cat{Top}$, $X \odot Y = X \times Y$ and $X \odot (Y \odot C) \to (X \times Y) \odot C$ the identity.
	\item Simplicial sets: $\cat{C} = \cat{sSet}$, $X \odot Y = \mr{Sing}(X) \times Y$ and $ X \odot (Y \odot C) \to (X \times Y) \odot C$ induced by the natural isomorphism $\mr{Sing}(X) \times \mr{Sing}(Y) \cong \mr{Sing}(X \times Y)$.
	\item Chain complexes (over $\bZ$): $\cat{C} = \cat{Ch}$, $X \odot A = C_*(X) \otimes A$ and $X \odot (Y \odot C) \to (X \times Y) \odot C$ induced by the Eilenberg-Zilber map $\mr{EZ}\colon C_*(X) \otimes C_*(Y) \to C_*(X \times Y)$.
\end{enumerate} 

These examples have well-known model structures: the Quillen model structure on $\cat{Top}$, the Quillen model structure on simplicial sets, and the projective model structure on $\cat{Ch}$. This allows us to talks about weak equivalences and cofibrations.

\subsubsection{Bar constructions and homotopy-theoretic techniques} In this paper we will use several constructions (e.g. bar constructions or functors associated to symmetric sequences), and we will recall sufficient conditions for these constructions to preserve weak equivalences. These results are often well-known and we will apply them throughout the remainder of this paper.

We first discuss the homotopy invariance properties of functors associated to symmetric sequences. A symmetric sequence $\cF = \{F(k)\}_{k \geq 0}$ is said to be $\Sigma$-cofibrant if each $F(k)$ is cofibrant as an object with $\fS_k$-action (in the projective model structure). This is the case if $F(k)$ is cofibrant as an object of $\cat{C}$ and the action is free.

\begin{lemma}\label{lemweakequivfunctor} Suppose $\cF$ is $\Sigma$-cofibrant and $X$ is cofibrant, then $\mathbf{F}(X)$ is cofibrant. If $X \to Y$ is a weak equivalence between cofibrant objects of $\cat{C}$, then $\mathbf{F}(X) \to \mathbf{F}(Y)$ is a weak equivalence. Similarly, if $X$ is cofibrant and $\cF \to \cG$ is a weak equivalence between $\Sigma$-cofibrant symmetric sequence, then $\mathbf{F}(X) \to \mathbf{G}(X)$ is a weak equivalence.

Furthermore, when $\cat{C} = \cat{Top}$ and the action of $\fS_k$ on $F(k)$ and $G(k)$ are free, we can drop the cofibrancy assumptions on $X$, $Y$, and $F(k)$ and $G(k)$ for $k \geq 0$.\end{lemma}

\begin{proof}The first part follows from Lemma 11.5.2 and Proposition 11.5.3 of \cite{fressebook}.
	
For the second part, we note that for a free action of a finite group $G$ on space $Z$, the quotient map $Z \to Z/G$ is a covering map and hence a Serre fibration, see e.g. Lemma 79.1 and Theorem 81.5 of \cite{munkres} and note that the proof of 81.5 does not use the assumption of local path-connectedness in the proof that the quotient map is a covering map. Thus when the $\fS_k$-action on $F(k)$ is free, there is a long exact sequence of homotopy groups for the fiber sequence $\fS_k \to F(k) \times X^k \to F(k) \times_{\fS_k} X^k$. This can be combined with the five lemma to the prove the second part.
\end{proof}

To avoid repeating arguments, one may use the following lemma.

\begin{lemma}If for all $k \geq 0$ the action of $\mathfrak{S}_k$ on $\cO(k)$ is free, then for $X \in \cat{Top}$ or $\cat{sSet}$ there is a natural weak equivalence $\mathbf{F}(C_*(X)) \to C_*(\mathbf{F}(X))$.\end{lemma}

\begin{proof}The Eilenberg-Zilber map 
	\[C_*(\cF(k)) \otimes C_*(X)^{\otimes k} \to C_*(\cF(k) \times X^k)\]
 is a weak equivalence by the K\"unneth theorem. This map is $\fS_k$-equivariant and since the action of $\fS_k$ on $F(k)$ is free, both are levelwise free chain complexes of $\bZ[\fS_k]$-modules. Thus applying $-\otimes_{\bZ[\fS_k]} \bZ$ preserves the weak equivalence, e.g. by a K\"unneth spectral sequence argument. Now note that we have
	\begin{align*}(C_*(\cF(k)) \otimes C_*(X)^{\otimes k}) \otimes_{\bZ[\fS_k]} \bZ &\cong C_*(\cF(k)) \otimes_{\bZ[\fS_k]} C_*(X)^{\otimes k} \\
	C_*(\cF(k) \times X^k) \otimes_{\bZ[\fS_k]} \bZ &\cong C_*(\cF(k) \times_{\fS_k} X^k)\end{align*}
\end{proof}

We will next describe a simplicial object known as the \emph{double bar construction}.  For $\cat{C} = \cat{Top}$, $\cat{sSet}$, or $\cat{Ch}$, geometric realization of a simplicial object $C_\bullet$ in $\cat{C}$ is defined to be the coend over $\Delta$ of $[k] \mapsto \Delta^k$ and $[k] \mapsto C_k$. For $\cat{C} = \cat{sSet}$, this is naturally weakly equivalent to the diagonal of the bisimplicial set. For $\cat{C} = \cat{Top}$, this is the ordinary (thin) geometric realization. For $\cat{C} = \cat{Ch}$ this is naturally weakly equivalent to taking the total chain complex of the associated normalized double chain complex.


\begin{definition}\label{defbarconstruction} Let $\cO$ be an operad, $A$ an $\cO$-algebra and $ \mathbf{F}$ a right $\cO$-functor. Let $B_\bullet( \mathbf{F}, \mathbf{O},A)$ be the simplicial object in $\cat{C}$ given by
	\[[p] \mapsto  \mathbf{F}(\mathbf{O}^p(A))\]
	with face maps $B_p( \mathbf{F},\mathbf{O},A) \to B_{p-1}(\mathbf{F},\mathbf{O},A)$ induced by (i) for $d_0$ the natural transformation $ \mathbf{F} \mathbf{O} \to \mathbf{F}$, (ii) for $d_i$ with $0 < i < p$ the composition natural transformation $\mathbf{O}^2 \to \mathbf{O}$ applied to the $i$th and $(i+1)$st copy of $\mathbf{O}$, and (iii) the action map $\mathbf{O}A \to A$ for $d_p$. Similarly, the degeneracies are induced by the unit natural transformation $\mr{id} \to \mathbf{O}$. 
	
For $\cat{C} = \cat{Top}$, $\cat{sSet}$, or $\cat{Ch}$, we define 
\[B( \mathbf{F}, \mathbf{O},A) \coloneqq \left\vert B_\bullet( \mathbf{F}, \mathbf{O},A) \right\vert\]
with $|-|$ the (thin) geometric realization discussed above.
\end{definition}

We will now discuss the homotopy invariance properties of this construction, and a spectral sequence computing its homology.

For a simplicial object $X_\bullet$ in $\cat{C}$, the $p$th \emph{latching object} is defined to be $L_p X \coloneqq \mr{colim}_{[r] \hookrightarrow [p]} X_r$ where the colimit is over all injective maps $[r] \hookrightarrow [p]$ in $\Delta^\mr{op}$ not equal to $\mr{id}_{[p]}$. A simplicial object $X_\bullet$ in $\cat{C}$ is said to be \emph{Reedy cofibrant} if each map  $L_p X \to X_{p+1}$ is a cofibration. This implies each $X_p$ is cofibrant and is implied by all degeneracy maps being cofibrations. Also note all bisimplicial sets are Reedy cofibrant. If $\cat{C} = \cat{Top}$, a related definition is that of a \emph{proper} simplicial space: this is Reedy cofibrancy with respect to the Str\o m model category, i.e. $L_p X \to X_{p+1}$ is required to be a Hurewicz cofibration. This is implied by each $s_i \colon X_p \to X_{p+1}$ being a Hurewicz cofibration. Note there is no condition on the individual $X_p$, as all spaces are cofibrant in the Str\o m model structure.

\begin{lemma}\label{lemgeomrel} If $X_\bullet \to Y_\bullet$ is a levelwise weak equivalence between Reedy cofibrant simplicial objects in $\cat{C}$, then $|X_\bullet| \to |Y_\bullet|$ is a weak equivalence. When $\cat{C} = \cat{Top}$ and $X_\bullet \to Y_\bullet$ is a levelwise weak equivalence between proper simplicial objects in $\cat{C}$, then $|X_\bullet| \to |Y_\bullet|$ is also a weak equivalence.\end{lemma}

\begin{proof}This is discussed in Section A.2.9 of \cite{luriehtt}. For proper simplicial objects in $\cat{C} = \cat{Top}$, this is Proposition A.1 of \cite{segalcats}.\end{proof}

\begin{lemma}\label{lemgeomrealss} Associated to each Reedy cofibrant simplicial object or proper simplicial space $X_\bullet$ there is a \emph{geometric realization spectral sequence}
	\[E^1_{pq} = H_q(X_p) \Rightarrow H_*(|X_\bullet|)\]
There is also a relative version.
\end{lemma}

\begin{proof}This is the spectral sequence associated to the skeletal filtration, where the Reedy cofibrancy condition is used to identify the $E^1$-page. For proper simplicial spaces, Section 5 of \cite{Se2} describes the spectral sequence for the so-called fat geometric realization, and the Reedy cofibrancy condition implies the fat realization is weakly equivalent to the geometric realization by Proposition A.1 of \cite{segalcats}.
\end{proof}

\begin{lemma}\label{lemchainssimplicialobjects} If $X_\bullet$ is a Reedy cofibrant simplicial object in $\cat{C} = \cat{Top}$ or $\cat{sSet}$, or a proper simplicial object in $\cat{C} = \cat{Top}$, then $C_*(X_\bullet)$ is a Reedy cofibrant simplicial chain complex and there is a natural weak equivalence $|C_*(X_\bullet)| \to C_*(|X_\bullet|)$.\end{lemma}

\begin{proof}The map is induced by Eilenberg-Zilber maps, and that it is a weak equivalence is a consequence of spectral sequence comparison applied to the geometric realization spectral sequence. Also see Theorem 5.5.2.17 of \cite{lurieha}.
\end{proof}

Suppose that $\mathbf{F}$ comes from a symmetric sequence $\cF$ and $\cO$ is an operad. Unwinding the definitions, one deduces that $|B_\bullet(\mathbf{F},\mathbf{O},-)|$ preserves a weak equivalence $A \to B$ of $\cO$-algebras when $\cF(k)$ and $\cO(k)$ have free $\fS_k$-actions if $\cat{C}=\cat{sSet}$. If $\cat{C} = \cat{Ch}$, one additionally needs to demand that the underlying chain complexes of $A$, $B$, $\cO(k)$ and $\cF(k)$ are cofibrant. If $\cat{C} = \cat{Top}$ one additionally needs to demand that the inclusion $\{\mr{id}\} \hookrightarrow \cO(1)$ is a Hurewicz cofibration.

\subsection{Little disks operads, $\cdt$-algebras and charged algebras} \label{suboperads}  In this subsection we define the algebras that are the subject of this paper. As discussed in the previous two subsections, we shall always assume that $\cat{C} = \cat{Top}$, $\cat{sSet}$, or $\cat{Ch}$.

\subsubsection{Tangential structures} \label{subsubsectangent}  A \emph{tangential structure} is a map $\theta\colon W \to BO(n)$. As a convention, we will demand that $W$ is $1$-connected and $\theta\colon W \to BO(n)$ is a fibration. Recall that the classifying space of the $n$th orthogonal group $BO(n)$ has a \emph{universal vector bundle} $\gamma$ over it. A \emph{$\theta$-framing} of an $n$-dimensional manifold $M$ is a bundle map $\phi_M\colon TM \to \theta^* \gamma$, where for us by definition a bundle map is a fiberwise linear isomorphism. We fix once and for all a $\theta$-framing on $\bR^n$, which exists since $W$ is non-empty, and is essentially unique by Lemma \ref{lemthetaconn}, which uses that $W$ is 1-connected. 

Let $\mr{Bun}(TM,TN)$ denote the space of all bundle maps $TM \to TN$ with the compact-open topology, and $\mr{Map}([0,\infty),\mr{Bun}(TM,\theta^*\gamma))$ the space of paths with the compact-open topology.

\begin{definition}If $M$ and $N$ are $\theta$-framed manifolds, then $\mr{Bun}^{\theta}(TM,TN)$ is the space of triples 
\[\Phi = (\phi,\alpha,\varphi) \subset \mr{Bun}(TM,TN) \times [0,\infty) \times \mr{Map}([0,\infty),\mr{Bun}(TM,\theta^*\gamma))\]
of a bundle map $\phi\colon TM \to TN$, a locally constant function $\alpha \colon M \to [0,\infty)$ and continuous map $\varphi \colon [0,\infty) \to \mr{Bun}(TM,\theta^* \gamma)$ starting at $\phi_M$ and equal to $\phi_N \circ \phi$ on $[\alpha,\infty)$ (i.e. a Moore path).\end{definition}

This is homotopy equivalent to the homotopy fiber over $\phi_M$ of the map
\[\phi_N \circ - \colon \mr{Bun}(TM,TN) \to \mr{Bun}(TM,\theta^* \gamma)\]

\begin{definition}We define the space of $\theta$-framed embeddings as the pull back in the following diagram (where $\mr{Emb}$ has the $C^\infty$-topology):
\[\xymatrix{\mr{Emb}^\theta(M,N) \ar@{.>}[d] \ar@{.>}[r] & \mr{Bun}^\theta (TM,TN) \ar[d]\\
	\mr{Emb}(M,N) \ar[r] & \mr{Bun}(TM,TN)}\]
\end{definition}

\begin{lemma}\label{lemthetaconn} For any two $\theta$-framings $\phi_1$ and $\phi_2$ on $\bR^n$, the space $\mr{Emb}^\theta((\bR^n)_{\phi_1},(\bR^n)_{\phi_2})$ is path-connected.\end{lemma}

\begin{proof}For any two $\theta$-framed manifolds the map $\mr{Emb}^\theta(M,N) \to \mr{Emb}(M,N)$ is a Serre fibration, because it is the pullback of the Serre fibration $\mr{Bun}^\theta (TM,TN) \to \mr{Bun}(TM,TN)$. In the case that $M = N = \bR^n$, the map $\mr{Emb}(M,N) \to \mr{Bun}(TM,TN)$ is a weak equivalence. Thus it suffices to show that $\mr{Bun}^{\theta}((T\bR^n)_{\phi_1},(T\bR^n)_{\phi_2})$ is path-connected. It is weakly equivalent to the homotopy fiber of the map $\phi_2 \circ - \colon O(n) \to \mr{Fr}(\theta^* \gamma)$ at the $\theta$-framing $\phi_1$, the latter being the frame bundle of $\theta^* \gamma$. This is in turn weakly equivalent to the homotopy fiber of the map $* \to W$ (the base point does not matter since $W$ is $0$-connected), and hence to the based loop space $\Omega W$, which is path-connected since $W$ was assumed to be $1$-connected.\end{proof}

One can compose two $\theta$-framed embeddings $(f,\Phi)\colon M \to N$ and $(g,\Psi)\colon N \to P$ by composing the embeddings to $g \circ f$, and concatenating the triples $\Phi = (\phi,\alpha,\varphi)$ and $\Psi = (\psi,\beta,\upsilon)$ as follows. One composes the bundle maps to obtain a bundle $\psi \circ \phi \colon TM \to TP$, takes the locally constant function $\beta \circ f + \alpha \colon M \to [0,\infty)$ and takes the path of bundle maps $TM \to \theta^*\gamma$ to be
\[t \mapsto \begin{cases} \varphi(t) & \text{if $t \in [0,\alpha]$} \\
\upsilon(t-\alpha) \circ \phi & \text{if $t \in [\alpha,\alpha+\beta]$} \\
\phi_P \circ \psi \circ \phi & \text{if $t \in [\alpha+\beta,\infty)$}
\end{cases}\]
This composition is associative. Composition with $\mr{id} \colon M \to M$ with the length $0$ Moore path at $\phi_M$ is the identity. 

We can then define a category of $\theta$-framed manifolds and embeddings.

\begin{definition} Let $\cat{Emb}^{\theta}$ denote the topological category whose objects are $\theta$-framed manifolds and whose morphism spaces are spaces of $\theta$-framed embeddings. Disjoint union gives $\cat{Emb}^{\theta}$ the structure of a symmetric monoidal category. On morphisms $(f,\Phi) \colon M \to N$ and $(g,\Psi) \colon P \to Q$, it is given by the map $(f \sqcup g,\Phi \sqcup \Psi)$, with $\Phi \sqcup \Psi$ defined by requiring its restriction to $M$ and $P$ to be $\Phi$ and $\Psi$ respectively. 
	
	
\end{definition}

The reason for requiring the length $\alpha$ of our Moore path to be locally constant on $M$ instead of constant, is to make $(-\sqcup-)$ a functor.


\begin{example}The examples of tangential structures relevant to this paper are:
	
	\begin{itemize}
		\item  the \emph{unoriented} tangential structure $\pi_O\colon BO(n) \m BO(n)$ (note here $W=BO(n)$ is not $1$-connected),
		\item the \emph{oriented} tangential structure $\pi_{SO}\colon EO(n)/SO(n) \to BO(n)$, and
		\item the \emph{framed} tangential structure $\pi_\mr{pt}\colon EO(n) \to BO(n)$.
	\end{itemize} For $\pi_O$, we have $\mr{Bun}^{O}(TM,TN) \simeq \mr{Bun}(TM,TN)$ and $\mr{Emb}^{O}(M,N) \simeq \mr{Emb}(M,N)$. 	One can always forget a $\theta$-framing to a $\pi_O$-framing, using the canonical map $W \to BO(n)$ over $BO(n)$ to induce natural maps $\mr{Bun}^{\theta}(TM,TN) \m \mr{Bun}^{O}(TM,TN)$ and $\mr{Emb}^{\theta}(M,N) \m \mr{Emb}^{O}(M,N)$.\end{example}

\subsubsection{$\cdt$-algebras} The operad of $\theta$-framed little $n$-disks will be defined in terms of $\theta$-framed embeddings of copies of $\bR^n$ into $\bR^n$.

\begin{definition} The \emph{$\theta$-framed little $n$-disks operad} $\cdt$ is the operad in the category $(\cat{Top},\times,*)$ given by $\cdt(k) = \mr{Emb}^\theta(\sqcup_{k} \bR^n,\bR^n)$. The symmetric group acts by permuting the Euclidean spaces in the domain, the operad composition map is given by composition of $\theta$-framed embeddings as described in the previous subsection, and the unit in $\cdt(1)$ is $\mr{id}\colon \bR^n \to \bR^n$ with Moore path at $\phi_{\bR^n}$ of length 0. \end{definition}

\begin{definition}Let $\bdt$ denote the monad in $\cat{C}$ associated to the operad $\cdt$.\end{definition}

The operad associated to the tangential structure $\pi_{SO}\colon EO(n)/SO(n) \to BO(n)$, $\cdtf$, is called the \emph{framed little $n$-disks operad}, and the operad associated to the tangential structure $\pi_\mr{pt}\colon EO(n) \to BO(n)$, $\cdtt$, is called the \emph{little $n$-disks operad}. 

\begin{remark}The nomenclature surrounding little disks operad can be confusing. We give some clarifications:
	
\begin{itemize}
	\item It is standard but confusing that the framed little $n$-disks operad consists of orientation-preserving embeddings while the little $n$-disks operad consists of the framed embeddings.
	\item Some authors use a version of the framed little $n$-disks operad which we denote $(\cdtt)^\mr{rect}$, given by taking the subspace of $\mr{Emb}(\sqcup_{k} D^n,D^n)$ consisting of the orientation-preserving embeddings that are a composition of translation, dilation and rotations. Identifying $\mr{int}(D^n)$ with $\bR^n$ induce a weak equivalence of operads $(\cdtt)^\mr{rect} \to \cdtt$, which can used to obtain from every $(\cdtt)^\mr{rect}$-algebra a weakly equivalent $\cdtt$-algebra by double bar construction $B(\bdtt,(\bdtt)^\mr{rect},-)$. To this algebra one can then apply the results of this paper.
	\item Some authors define the framed little $n$-disks operad using embeddings that do not necessarily preserve the orientation, i.e. they use $BO(n)$ instead of $BSO(n)$. Our proof does not work for this definition, cf.\ part (ii) of Remark \ref{remgener}.
	\item Finally, there is a distinction between the unital and non-unital $\cdt$-operads, which differ in whether $\cdt(0)$ is $\ast$ or $\varnothing$. Ours is the former.\end{itemize}\end{remark}

\subsubsection{Charged algebras}  \label{subsubcharged} We discuss algebra with a generalization of charge as discussed in the introduction.

\begin{definition}\label{defcharge} A \emph{partial monoid of charges} is an abelian cancellative partial monoid $C$ with unit.\end{definition}

We fix such a partial monoid $C$ of charges. The reader may want to keep in mind the example $C = \bN_0$, the non-negative integers under addition. Partial monoids that are not monoids will not be needed until Section \ref{subseccells}. We next define $C$-charged algebras to remedy the fact connected components do not make sense for chain complexes.  



\begin{definition}  \begin{enumerate}[(i)]
\item A \emph{$C$-charged space} is a space $X$ with a decomposition $X = \bigsqcup_{\bc \in C}X(\bc)$. A morphism of $C$-charged spaces is a continuous map preserving the decomposition.
\item A \emph{$C$-charged simplicial set} is a simplicial $X$ with a decomposition $X = \bigsqcup_{\bc \in C}X(\bc)$. A morphism of $C$-charged spaces is a continuous map preserving the decomposition.
\item A \emph{$C$-charged chain complex} is a non-negatively graded chain complex $A$ with a decomposition $A = \bigoplus_{\bc \in C} A(\bc)$ as chain complexes. A morphism of $C$-charged chain complexes is a chain map preserving the decomposition.
\end{enumerate}
\end{definition}

If $x \in X$ is an element of $X(\bc)$ or $a \in A$ is an element of $A(\bc)$, then we say it has \emph{charge} $\bc$. There is a $C$-charged space $C$ with a single point in each charge, and a $C$-charged chain complex $\bZ[C]$ with a single $\bZ$-summand in each charge. We will also view $C$ as a simplicial set by taking $C$ at every level and taking all maps to be isomorphisms. This is naturally a $C$-charged simplicial set. 

\begin{definition}\begin{enumerate}[(i)]
\item An \emph{augmented} $C$-charged space $X$ is a $C$-charged space $X$ with a map $\epsilon\colon X \to C$ of $C$-charged spaces, which we call the \emph{augmentation map}. We denote the category of augmented $C$-charged spaces by $\cat{Top}_C$.
\item An \emph{augmented} $C$-charged simplicial set $X$ is a $C$-charged simplicial set $X$ with a map $\epsilon\colon X \to C$ of $C$-charged simplicial sets, which we call the \emph{augmentation map}. We denote the category of augmented $C$-charged simplicial sets by $\cat{sSet}_C$.

\item An \emph{augmented} $C$-charged chain complex $A$ is a $C$-charged chain complex $A$ with a map $\epsilon\colon A \to \bZ[C]$ of $C$-charged chain complexes, which we call the \emph{augmentation map}. We denote the category of augmented $C$-charged spaces by $\cat{Ch}_C$.
\end{enumerate}\end{definition}

We remark that any $C$-charged space or simplicial set has a unique augmentation, but this is not the case for $C$-charged chain complexes. Conversely, from a map $\epsilon \colon X \to C$ one can recover the structure of a $C$-charged space or simplicial set on $X$, but this is not the case for $C$-charged chain complexes.

\begin{definition}\begin{enumerate}[(i)]
\item An augmented $C$-charged space $X$ is said to be \emph{connected} if $\epsilon$ induces a $\pi_0$-isomorphism.
\item An augmented $C$-charged simplicial set $X$ is said to be \emph{connected} if $\epsilon$ induces a $\pi_0$-isomorphism.
\item An augmented $C$-charged chain complex $A$ is said to be \emph{connected} if each chain complex $A(\bc)$ is $0$ in negative homological degrees and $\epsilon$ induces a $H_0$-isomorphism.
\end{enumerate}\end{definition}

The chains on a (augmented) $C$-charged space have the structure of a (augmented) $C$-charged chain complex. Augmented $C$-charged spaces, simplicial sets or chain complexes are symmetric monoidal categories using Day convolution applied to the Cartesian and tensor products respectively: $(X \times Y)(\bc) = \bigsqcup_{\bc'+\bc''=\bc} X(\bc') \times X(\bc'')$ and $(A \otimes B)(\bc) = \bigoplus_{\bc'+\bc''=\bc} A(\bc') \otimes B(\bc'')$. The new augmentations are obtained as compositions $X \times Y \to C \times C \to C$ and $A \otimes B \to \bZ[C] \otimes \bZ[C] \to \bZ[C]$, with both right-hand maps coming from the partial monoid structure of $C$.

The operad $\cdt$ can be considered as an operad in $C$-charged spaces, by declaring all elements to have charge $0$. Note that the forgetful functor from $C$-charged spaces, simplicial sets or chain complexes to spaces, simplicial sets or chain complexes is only symmetric monoidal if $C$ is a monoid. Thus, $\cdt$-algebras in $C$-charged spaces, simplicial sets or chain complexes are not necessarily $\cdt$-algebras in the underlying category. 

\begin{definition}\begin{enumerate}[(i)]
\item A \emph{$C$-charged $\cdt$-algebra in spaces} is a connected augmented $C$-charged space $X$ with the structure of an algebra over $\cdt$ in $\cat{Top}_C$.
\item A \emph{$C$-charged $\cdt$-algebra in simplicial sets} is a connected augmented $C$-charged simplicial set $X$ with the structure of an algebra over $\cdt$ in $\cat{sSet}_C$.
\item  A \emph{$C$-charged $\cdt$-algebra in chain complexes} is a connected augmented $C$-charged chain complex $A$ with the structure of an algebra over $C_*(\cdt)$ in $\cat{Ch}_C$.
\end{enumerate}\end{definition}

To make constructions in this paper homotopy invariant, we will need to impose some cofibrancy conditions on our $C$-charged $\cdt$-algebra in chain complexes.

\begin{definition}A $C$-charged $\cdt$-algebra in chain complexes $A$ is \emph{cofibrant} if the underlying chain complex $A$ is cofibrant.
\end{definition}
	
Note that if $X$ is a $C$-charged $\cdt$-algebra in spaces or simplicial sets, then $C_*(X)$ is cofibrant. Unless mentioned otherwise, charged algebras will be assumed to be cofibrant.
	
\subsubsection{Homological stability for charged algebras}

Let $A$ be a $C$-charged $\cdt$-algebra in chain complexes. The augmentation gives a preferred choice of a generator of $H_0(A(\bc_0))$. Assume  $\bc+\bc_0$ is defined in $C$ and pick a representative $a \in A(\bc_0)$ of this generator. We can define a stabilization map 
\[t_{\bc_0}\colon H_*(A(\bc)) \to H_*(A(\bc+\bc_0))\] 
by multiplying with the homology class of $a$ using the ring structure on $H_*(A)$ induced by the $\cdt$-algebra structure. If $n \geq 2$, the ring structure is commutative, and the map $t_{\bc_0}$ is uniquely determined by the choice of augmentation. The singular chains on a $C$-charged $\cdt$-algebra $X$ in spaces are a $C$-charged $\cdt$-algebra in chain complexes, so this definition also gives us a stabilization map on the homology of $X$. 

We now define homological stability when $C = \bN_0^d$, whose elements we denote by $\bk$. We let $\be_i$ denote the basis vector $(0,\ldots,0,1,0,\ldots,0)$ (1 is in the $i$th position) and call $t_i \coloneqq t_{\be_i}$ a \emph{basic stabilization map}.

\begin{definition}
A map $f\colon X \to Y$ of chain complexes, simplicial sets, or spaces is a called an $N$-equivalence if the relative homology group $H_*(Y,X)$ vanishes for $* \leq N$. Equivalently $f$ induces an isomorphism on homology groups $H_*$ for $* \leq N-1$ and a surjection for $* \leq N$.
\end{definition}

\begin{definition}Let $C = \bN_0^d$. \label{defhomstab}
\begin{enumerate}[(i)]
\item Let $A$ be a $C$-charged $\cdt$-algebra in chain complexes, then we say that $A$ \emph{has homological stability} if there is a function $\rho\colon \bN_0 \to \N_0$ with $\lim_{j \to \infty} \rho(j) =\infty$ such that all basic stabilization maps $t_i\colon H_*\left(A(\bk)\right) \to H_*\left(A(\bk+\be_i)\right)$ are $\rho(k_i)$-equivalences.
\item Let $X$ be a $C$-charged $\cdt$-algebra in spaces or simplicial sets, then we say that $X$ \emph{has homological stability} if $C_*(X)$ has.\end{enumerate}\end{definition}

\begin{remark}We make some remarks regarding this definition: \begin{itemize}

\item In the case $C = \bN_0$, we can remove reference to a function $\rho$ and instead say that all basic stabilization maps induce isomorphisms on $H_*$ for $k$ large with respect to $*$.
\item This is not the most general definition of homological stability that one could give. It is not enough to demand that iterated stabilization maps are eventually isomorphisms, because we will need uniformity in the $d$ different directions in $\bN_0^d$. However, one can use different $\rho$ for each direction and make these depend on all coordinates of $\bk$ instead of just a single one. 
\item To give a quantitative version of the main theorem, we will impose extra conditions on the function $\rho$ (see Theorem \ref{thmmainrange}). These extra conditions are easier to state if one extends the domain and codomain of $\rho$ from $\bN_0$ to $\R$.


\end{itemize}\end{remark}

\subsection{Topological chiral homology} In this subsection we will define topological chiral homology, first for $\cdt$-algebras on $\theta$-framed manifolds but eventually for partial $\cdt$-algebras on $\theta$-framed manifold bundles. We will also discuss homological stability and a useful spectral sequence. As before $\cat{C} = \cat{sSet}$ or $\cat{Ch}$, and $C$-charged $\cdt$-algebras in $\cat{Ch}$ are assumed to be cofibrant.

\subsubsection{Topological chiral homology}\label{subsubtch}
To define topological chiral homology, we define the following right $\bdt$-functor.

\begin{definition}\label{defmtheta}Suppose $M$ is a $\theta$-framed manifold. Let $\mathbf{M}^\theta$ be the right functor over $\bdt$ given by
\[C \mapsto \bigsqcup_{k \geq 0} \mr{Emb}^\theta(\sqcup_k \bR^n,M) \odot_{\fS_k} C^{\otimes k}.\] The right functor structure is induced by composition of $\theta$-framed embeddings. 
\end{definition}


We will now define topological chiral homology of an $\cdt$-algebra $A$ in $\cat{C}$ over a $\theta$-framed manifold $M$ using the two-sided bar construction of Definition \ref{defbarconstruction}. This is one of many equivalent models for topological chiral homology and is a concrete instance of the homotopy coend in Definition 3.2 of \cite{Fr2}. One could also check that this model satisfies the axioms in Theorem 1.1 of \cite{Fr2}.

\begin{definition}\label{deftch}We define the \emph{topological chiral homology} $\int_M A$ of $A$ over $M$ to be the geometric realization
	\[\int_M A \coloneqq  B(\mathbf{M} ^\theta,\bdt,A).\]
\end{definition}

We remark that $\mathbf{M}^\theta$ is an enriched functor from the category $\cat{Emb}^\theta$ of $\theta$-framed manifolds with spaces of $\theta$-framed embeddings as spaces of morphisms, to the category of right $\bdt$-functors in $\cat{C}$. This makes sense as $\cat{C}$ has a copowering over $\cat{Top}$. The following is one of the axioms given in \cite{Fr2}.

\begin{lemma}\label{lemtchrn} We have that $\int_{\bR^n} A \simeq A$.\end{lemma}

\begin{proof} For $M=\R^n$, we have that $\mathbf{M} ^\theta=\bdt$. The unit of the monad endows the simplicial object $B_\bullet(\bdt,\bdt,A)$ with an extra degeneracy.\end{proof}

Another basic property is topological chiral homology preserves $N$-equivalences.

\begin{lemma}\label{lemtchcomp}Let $A \to B$ is a map of $\cdt$-algebras that induces an isomorphism on homology. Then the induced map $\int_M A \to \int_M B$ is also an isomorphism on $H_*$. More generally, if $f\colon A \to B$ is an $N$-equivalence then so is $\int_M A \to \int_M B$.\end{lemma}

\begin{proof}By Lemma \ref{lemchainssimplicialobjects} we may assume $\cat{C} = \cat{Ch}$. Since we constructed topological chiral homology by geometrically realizing a Reedy cofibrant simplicial object, Lemma \ref{lemgeomrealss} says there is a geometric realization spectral sequence converging to the relative homology $H_*\left(\int_M B,\int_M A\right)$ with $E^1$-page given by the relative homology
\[E^1_{p,q} = H_q\left(\mathbf{M}^\theta(\bdt)^p(B),\mathbf{M}^\theta(\bdt)^p(A)\right)\]
The lemma follows if the $E^1$ page vanishes for $q \leq N$, and to prove this we prove a version of Lemma \ref{lemweakequivfunctor} to see that if $f\colon X \m Y$ is an $N$-equivalence, then so are the maps $\mathbf{M}^\theta(X) \m \mathbf{M}^\theta(Y)$ and $\bdt(X) \m \bdt(Y)$. The functor $\bdt$ is a special case $M = \bR^n$ of the functor $\mathbf{M}^\theta$. Recall the definition
\[\mathbf{M}^\theta(X) = \bigoplus_{k \geq 0}  C_*(\mr{Emb}^\theta(\sqcup_k \bR^n,M)) \otimes_{\bZ[\mathfrak{S}_k]} X^{\otimes k}\]
Since the action of $\mathfrak{S}_k$ on $\mr{Emb}^\theta(\sqcup_k \bR^n,M)$ is free, $C_*(\mr{Emb}^\theta(\sqcup_k \bR^n,M))$ is a levelwise free chain complex of $\bZ[\mathfrak{S}_k]$-modules. If $C_\ast$ is a levelwise free chain complexes of $\bZ[\mathfrak{S}_k]$-modules, then we have that $C_\ast \otimes_{\bZ[\mathfrak{S}_k]} (-) \colon \cat{Ch}_{\bZ[\mathfrak{S}_k]} \to \cat{Ch}$ preserves weak equivalences and $N$-equivalences. This can be proven using the K\"unneth spectral sequence. By the K\"unneth theorem $X^{\otimes k} \to Y^{\otimes k}$ is an $N$-equivalence if $X \to Y$ is, which proves the desired result.\end{proof}

\subsubsection{Topological chiral homology of partial $\cdt$-algebras and completions} \label{subsecpartialalg}

In later sections, we will construct $\cdt$-algebras by adding generators and relations to given $\cdt$-algebras. This is done by creating an intermediate object called a partial $\cdt$-algebra, and then freely completing the result to obtain an $\cdt$-algebra, a procedure described in this subsection.

\begin{definition}\label{defpartialbarconstruction} A \emph{structure of a partial $\cdt$-algebra} on an object $A$ of $\cat{C}$ is a simplicial object $A_\bullet$ with the following properties:
	\begin{enumerate}[(i)]
		\item There exists a subobject $\mr{Comp}_p \subset (\bdt)^p(A)$ so that $A_p = \bdt(\mr{Comp}_p)$.
		\item We have that $\mr{Comp}_0 = A$.
		\item If $p \geq 1$ and $0 \leq i < p$, then the face map $d_i \colon \bdt(\mr{Comp}_p) \to \bdt(\mr{Comp}_{p-1})$ is the restriction of the map \[(\bdt)^i(c_{(\bdt)^{p-i-1}(A)}) \colon (\bdt)^{p+1}(A) \to (\bdt)^{p}(A)\] with $c$ the composition natural transformation of the monad $\bdt$.
		\item If $p \geq 0$ and $0 \leq j \leq p$, the degeneracy map $s_j \colon \bdt(\mr{Comp}_p) \to \bdt(\mr{Comp}_{p+1})$ is the restriction of the map \[(\bdt)^{j+1}(1_{(\bdt)^{p-j}(A)}) \colon (\bdt)^{p+1}(A) \to (\bdt)^{p+2}(A)\] with $1$ the unit natural transformation of the monad $\bdt$.
	\end{enumerate}
\end{definition}

We denote the geometric realization $|A_\bullet|$ in $\cat{C}$ by $\bar{A}$. When $\cat{C} = \cat{Ch}$, we will again assume that $A_\bullet$ is levelwise cofibrant. In all relevant examples appearing in this paper, this is the case if $A$ is. Since $A_\bullet$ is an $\cdt$-algebra in simplicial objects in $\cat{C}$, $\bar{A}$ is also an $\cdt$-algebra. Furthermore, there is a canonical map $A \to \bar{A}$ given by including $A$ as $\mr{id} \odot A$ viewed as $0$-simplices of $\bar{A}$.

\begin{definition}We call the $\cdt$-algebra $\bar{A}$ the \emph{completion} of $A_\bullet$ (or the completion of $A$, when the partial $\cdt$-algebra structure on $A$ is implicit).\end{definition}

\begin{example}\label{examptrivialpartial} If $A$ is $\cdt$-algebra, then it in particular can be considered as a partial $\cdt$-algebra $A_\bullet$ by taking $A_\bullet = B_\bullet(\bdt,\bdt,A)$ with face maps induced by the operad composition and the $\cdt$-algebra structure on $A$. In this case we have that $\mr{Comp}_p = (\bdt)^p(A)$. We claim its completion is canonically weakly equivalent to $A$ as a $\cdt$-algebra. The algebra structure map $\bdt(A) \to A$ induces an augmentation map $A_\bullet \to A$. This induces a weak equivalence upon geometric realization, because the outermost unit maps $(\bdt)^{p+1}(A) \to (\bdt)^p(A)$ give $A_\bullet \to A$ an extra degeneracy, and hence $\bar{A} = |A_\bullet| \to A$ is a weak equivalence.

\end{example}

\begin{example}\label{exampcomp}In all cases that we will consider, the partial $\cdt$-algebra on $A$ will be obtained by constructing $\mr{Comp}_p$ out of two pieces of data: (a) a subobject $\mr{Comp}_1 \subset \bdt(A)$, and (b) a map $c_1 \colon \mr{Comp}_1 \to \mr{Comp}_0 = A$ such that $c_1$ is associative and unital in the following sense:
\begin{enumerate}[(i)]
\item  if $\mr{Comp}_2, \mr{Comp}_2' \subset (\bdt)^2(A)$ are defined as the pull backs \[\xymatrix{\mr{Comp}_2 \ar@{.>}[r] \ar@{.>}[d]_{c_2} & \bdt(\mr{Comp}_{1}) \ar[d]^{\bdt(c_1)}\\
\mr{Comp}_1 \ar[r]_{i} & \bdt(A)} \qquad \xymatrix{\mr{Comp}_2' \ar@{.>}[r] \ar@{.>}[d]_{\tilde{c}_2'} & (\bdt)^2(A) \ar[d]^{c}\\
\mr{Comp}_1 \ar[r]_{i} & \bdt(A)}\] 
where $i$ is the inclusion and $c \colon (\bdt)^2 \to \bdt$ the monad composition natural transformation, then we have $\mr{Comp}_2 \subset \mr{Comp}_2'$ and $c_1 \circ c_2 = c_1 \circ \tilde{c}_2'$ on $\mr{Comp}_2$. 
\item $\mr{id} \odot A \subset \mr{Comp}_1$ and following diagram commutes
\[\xymatrix{A \ar[r]^-{\mr{id} \odot -}  \ar@{=}[rd] & \mr{Comp}_1 \ar[d]^{c_1} \\
 & A}\]

 \end{enumerate} 




We explain how to obtain $\mr{Comp}_p$ from this. Suppose we have defined $\mr{Comp}_{p-1} \subset \bdt(\mr{Comp}_{p-2})$ and a map $c_{p-1}\colon \mr{Comp}_{p-1} \to \mr{Comp}_{p-2}$, then $\mr{Comp}_p \subset \bdt(\mr{Comp}_{p-1})$ and $c_p \colon \mr{Comp}_p \to \mr{Comp}_{p-1}$ are obtained via the pull back
\[\xymatrix{\mr{Comp}_p \ar@{.>}[r] \ar@{.>}[d]_{c_p} & \bdt(\mr{Comp}_{p-1}) \ar[d]^{\bdt(c_{p-1})}\\
	\mr{Comp}_{p-1} \ar[r]_-{i} & \bdt(\mr{Comp}_{p-2})}\]
with $i$ the inclusion. Since $\bE_n^\theta$ preserves pullbacks, as it is built out of coproducts, quotients by group actions and limits, we see that we can also describe $\mr{Comp}_p$ as the pullback
\[\xymatrix{\mr{Comp}_p \ar@{.>}[r] \ar@{.>}[d] & (\bdt)^{p-1}(\mr{Comp}_1) \ar[d]^{(\bdt)^{p-1}(c_1)} \\
\mr{Comp}_{p-1} \ar[r]_-i & (\bdt)^{p-2}(\mr{Comp}_1)}\]
Using this we can describe $\mr{Comp}_p$ as those elements of $(\bdt)^{p-1}(\mr{Comp}_1)$ such that applying $c_1$ in the innermost most position $(p-2)$ times we obtain an element of $\mr{Comp}_1$, and $c_p$ as the restriction of $(\bdt)^{p-1}(c_1)$.

The last description makes clear that Example \ref{examptrivialpartial} is a special case of this construction. One declares all elements to be composable by taking $\mr{Comp}_1 = \bdt(A)$ and $c \colon \mr{Comp}_1 = \bdt(A) \to A$ to be given by the $\cdt$-algebra structure on $A$.

The next lemma describes how to obtain a partial $\cdt$-algebra structure from this construction.
\end{example}

\begin{lemma}For any right $\bdt$-functor $\mathbf{F}$, we have that $\mathbf{F}(\mr{Comp}_p) \subset \mathbf{F}(\bdt)^{p}(A)$ is a simplicial object, if all face maps and degeneracy maps are the restrictions of those of the bar construction, with the exception of $d_p \colon \mathbf{F}(\mr{Comp}_p) \to \mathbf{F}(\mr{Comp}_{p-1})$ for $p \geq 1$, which is defined to be $\mathbf{F}(c_p)$.\end{lemma} 

\begin{proof}It suffices to check that: (a) the face and degeneracy maps have images in the required codomain, (b) the last face map $d_p = \mathbf{F}(c_p)$ satisfies the simplicial identities.

We start with (a) in the case of degeneracy maps. Recall that $\mathbf{F}(\mr{Comp}_p)$ consists of all elements of $\mathbf{F}(\bdt)^{p-1}(\mr{Comp}_1)$ satisfying the following condition 
\begin{quote}
	($\ast$) applying the map $d_p$ in the innermost position $(p-2)$ times we obtain an element of $\mathbf{F}(\mr{Comp}_1)$.
\end{quote}
Adding an additional identity element preserves this property, with the possible exception for the right-most degeneracy map $s_p \colon \mathbf{F}(\mr{Comp}_p) \to \mathbf{F}(\mr{Comp}_{p+1})$ where it is a consequence of property (ii).

For (a) in the case of face maps $d_i \colon \mathbf{F}(\mr{Comp}_p) \to \mathbf{F}(\mr{Comp}_{p-1})$, we distinguish $i=0$, $0 < i < p$ and $i=p$. For $i=0$, one uses that $\mr{Comp}_p \subset \bdt(\mr{Comp}_{p-1})$. For $0<i<p$, one uses that associativity of the monad composition implies that ($\ast$) is preserved by $d_i$. The case $i=p$ follows directly from the definition of ($\ast$).

For checking (b), i.e. that $d_p$ satisfies the simplicial identities, one uses that $d_p$ is obtained by restriction $\mathbf{F}((\bdt)^{p-2}(c_1))$ and then applying $\mathbf{F}(\bdt)^{p-2}$ to properties (i) and (ii) gives $d_p \circ d_p = d_p \circ d_{p-1}$ and $d_p \circ s_p = \mr{id}$ respectively.
\end{proof}

Using the completion of partial $\cdt$-algebras, we define topological chiral homology of partial $\cdt$-algebras as follows.

\begin{definition}For $A$ a partial $\cdt$-algebra, with partial $\cdt$-algebra structure given by $A_\bullet$ and completion $\bar{A}$, we define $\int_M A$ to be $\int_M \bar{A}$.\end{definition}

We now give a smaller equivalent construction in the setting of Example \ref{exampcomp}. To do so we consider $\mathbf{M}^\theta (\mr{Comp}_\bullet)$ with $\mathbf{M}^\theta$ as in Definition \ref{defmtheta}, which was shown to be a simplicial object in Example \ref{exampcomp}.

\begin{lemma}\label{lemcompsimp} Suppose $A$ has a partial $\cdt$-algebra structure $A_\bullet$ as in Example \ref{exampcomp}. Then we have that $\int_M \bar{A}$ is weakly equivalent to realization of the simplicial object $\mathbf{M}^\theta \left(\mr{Comp}_\bullet\right)$.\end{lemma}

\begin{proof}
To see that these two definitions are equivalent, note that $\int_M \bar{A}$ is the realization of the bisimplicial object 
\[[p,p'] \mapsto \mathbf{M}^\theta(\bdt)^{p+1}(\mr{Comp}_{p'})\]
which admits an augmentation in the $p$-direction to $\mathbf{M}^\theta \mr{Comp}_{\bullet}$ by using the right $\bdt$-functor structure of $\mathbf{M}^\theta$. An extra degeneracy argument in the $p$-direction implies that for each $p'$ the augmentation induces a weak equivalence 
\[\left\vert[p] \mapsto  \mathbf{M}^\theta(\bdt)^{p+1}\left(\mr{Comp}_{p'}\right)\right\vert \to \mathbf{M}^\theta\left(\mr{Comp}_{p'}\right)\]
Realizing this weak equivalence in the $p'$-direction and using Lemma \ref{lemgeomrel} gives the desired weak equivalence.
\end{proof}

\subsubsection{Topological chiral homology for manifold bundles} One may generalize topological chiral homology to manifold bundles, using embeddings of a collection of disks with image in a fiber.

Let $B$ be a paracompact space and  $\pi \colon E \to B$  a manifold bundle of $B$ with fibers diffeomorphic to $M$. Let $\tilde{\pi} \colon \tilde{E} \to B$ be the associated principal $\mr{Diff}(M)$-bundle and  vertical tangent bundle $T_v E$ be given by
\[T_v E = \tilde{E} \times_{\mr{Diff}(M)} TM\]
A $\theta$-framing of $\pi \colon E \to M$ is a $\theta$-framing of its vertical tangent bundle of fibers, i.e. a bundle map $\phi_E \colon T_v E \to \theta^*\gamma$.


We shall define $\smash{\int_{E \downarrow B} A}$ for an $\cdt$-algebra $A$ by following the construction in Subsection \ref{subsubtch} but modifying the right $\bdt$-functor $\mathbf{M}^\theta$. To do so, we follow the constructions in Subsection \ref{subsubsectangent} and define a space $\mr{Emb}_B(\sqcup_k \bR^n,E)$ of embeddings $\sqcup_k \bR^n \hookrightarrow E$ with image in some fiber $\pi^{-1}(b)$ of $\pi$:
\[\mr{Emb}_B(\sqcup_k \bR^n,E) = \tilde{E} \times_{\mr{Diff}(M)} \mr{Emb}(\sqcup_k \bR^n,M)\]

Similarly define $\mr{Bun}_{B}(\sqcup_k T\bR^n,T_v E)$ as the space of vector bundle maps $\phi \colon \sqcup_k T\bR^n \to T_v E$ with image in a single fiber $(T_v E)_b$, and $\mr{Bun}^\theta_{B}(\sqcup_k T\bR^n,T_v E)$ as the space of triples $(\phi,\alpha,\varphi)$ of a bundle map $\phi \colon  \sqcup_k T\bR^n \to T_v E$, a locally constant function $\alpha \colon \sqcup_k \bR^n \to [0,\infty)$ and a path $\varphi_t \colon [0,\infty) \to \mr{Bun}_B(\sqcup_k T\bR^n,\theta^* \gamma)$ starting at $\prod_k \phi_{\bR^n}$ and equal to $\phi_E \circ \phi$ on $[\alpha,\infty)$. Then $\mr{Emb}^\theta_B(\sqcup_k \bR^n,E)$ is defined to be the pull back
\[\xymatrix{\mr{Emb}_B^\theta(\sqcup_k \bR^n,E) \ar@{.>}[r] \ar@{.>}[d] & \mr{Bun}_{B}^\theta(\sqcup_k T\bR^n,T_v E) \ar[d] \\
\mr{Emb}_B(\sqcup_k \bR^n,E) \ar[r] & \mr{Bun}_B(\sqcup_k T\bR^n,T_v E)}\]

\begin{definition}Let $\mathbf{E{\downarrow}B}^\theta$ be the following right functor over $\bdt$
\[C \mapsto \bigsqcup_{k \geq 0} \mr{Emb}^\theta_B(\sqcup_k\bR^n,E) \odot_{\fS_k} C^{\otimes k}\]
where the right functor structure comes from composition of $\theta$-framed embeddings.
\end{definition}

\begin{definition} The \emph{fiberwise topological chiral homology} $\int_{E \downarrow B} A$ of the manifold bundle $E \to B$ is the geometric realization \[\int_{E \downarrow B} A \coloneqq |B_\bullet(\mathbf{E{\downarrow}B}^\theta,\bdt,A)|\]
\end{definition}

There is a map $\pi_k \colon \mr{Emb}_B(\sqcup_k \bR^n,E) \to B$, mapping an $\theta$-framed embedding to the fiber containing its image. This commutes with the $\mathfrak{S}_k$-action and is preserved under precomposition by $\theta$-framed embeddings, if $\cat{C} = \cat{sSet}$ it induces a well-defined map $\int_{E \downarrow B} A \to B$. 

This allows us to take the the homology of fiberwise topological chiral homology with certain local coefficient system coming from the base $B$ when $\cat{C} = \cat{Top}$ or $\cat{sSet}$. For $\cat{C} = \cat{Ch}$, we will describe an analogous construction.

Suppose we are given a discrete group $G$, a $\Z[G]$-module $\cL$, and a principal $G$-bundle $p \colon \tilde{B} \to B$. Pull back $E$ along $p$ to get a new $\theta$-framed manifold bundle $\tilde{\pi} \colon \tilde{E} \to \tilde{B}$ with a $G$-action covering the $G$-action of $\tilde{B}$. The fiberwise $\theta$-framed embeddings $\mr{Emb}^\theta_{\tilde{B}}(\sqcup_k\bR^n,\tilde{E})$ inherit a $G$-action commuting with the $\mathfrak{S}_k$-action. We can define a right $\bdt$-functor $\mathbf{E \downarrow B}^{\theta,\cL}$ by
\[C \mapsto \bigsqcup_{k \geq 0} \left(C_*(\mr{Emb}^\theta_{\tilde{B}}(\sqcup_k\bR^n,\tilde{E})) \otimes_{\bZ[G]} \cL\right)  \otimes_{\fS_k} C^{\otimes k}\]

If $\cL$ is the trivial $\bZ[G]$-module $\bZ$, then this coincides with $\mathbf{E \downarrow B}^{\theta}$, since the $G$-action is free. The only other case we will need, is when $G$ is the symmetric group $\mathfrak{S}_r$ and $\cL$ is the sign representation $\bZ_{\pm 1}$.

\begin{definition}Given a principal $G$-bundle $\tilde{B} \to B$ and $\bZ[G]$-module $\cL$ as above, the \emph{fiberwise topological chiral homology with coefficients in $\cL$}, denoted $\int_{E \downarrow B,\cL} A$ of the manifold bundle $E \to B$ is the geometric realization
	\[\int_{E \downarrow B,\cL} A \coloneqq \left\vert B_\bullet\left(\mathbf{E{\downarrow}B}^{\theta,\cL},\bdt,A\right)\right\vert\]
\end{definition}

\begin{remark}For an $\cdt$-algebra $X$ in $\cat{C} = \cat{Top}$ or $\cat{sSet}$, we have that
\[H_*\left(\int_{E \downarrow B,\cL} C_*(X)\right) \cong H_*\left(\int_{E \downarrow B} X;\pi^* \cL\right)\] with $\pi \colon \int_{E \downarrow B} X \to B$ described above.\end{remark}

\subsubsection{A spectral sequence for topological chiral homology of manifold bundles} 

In this section we will construct a spectral sequence for topological chiral homology of manifold bundles, analogous to the Serre spectral sequence. We start with a lemma concerning the map $\pi_k \colon \mr{Emb}^\theta_B(\sqcup_k\bR^n,E) \to B$ induced by the map $\pi \colon E \to B$.

\begin{lemma}For a manifold bundle $\pi \colon E \to B$ with $\theta$-framing $\phi_E$, the map 
	\[\pi_k \colon \mr{Emb}^\theta_B(\sqcup_k\bR^n,E) \to B\]
is a Serre fibration.\end{lemma}

\begin{proof}The local trivializations of $\pi \colon E \to B$ induce local trivializations of the map $\pi_k \colon \mr{Emb}_B(\sqcup_k \bR^n,E) \to B$, so it is a Serre fibration. Since $\mr{Bun}_{B}^\theta(\sqcup_k T\bR^n,T_v E) \to \mr{Bun}_{B}(\sqcup_k T\bR^n,T_v E)$ is a Serre fibration, so is its pull back $\mr{Emb}_B^\theta(\sqcup_k \bR^n,E) \to \mr{Emb}_B(\sqcup_k \bR^n,E)$. Hence the composition $\mr{Emb}_B^\theta(\sqcup_k \bR^n,E) \to B$ is also a Serre fibration.
\end{proof}

Before constructing the promised spectral sequences, we recall a construction of the Serre spectral sequence due to Dress and described in Section 6.4 of \cite{mccleary}. If $f \colon E \to B$ is a Serre fibration, consider the bisimplicial set $K_{\bullet,\bullet}(f)$ with $(p,q)$-simplices given by pairs $(u,v)$ of $u \colon \Delta^p \times \Delta^q \to E$ and $v$ in $\Delta^p \to B$ such that the following diagram commutes
\[\xymatrix{\Delta^p \times \Delta^q \ar[r]^-u \ar[d] & E \ar[d]^f \\
	\Delta^p \ar[r]_-v & B}\]
This has the following properties:
\begin{enumerate}[(i)]
\item It follows from the definition that $K_{0,\bullet}(f) \cong \mr{Sing}(E)$, the singular simplicial set.
\item By Lemma 6.48 of \cite{mccleary}, precomposition with the unique map $\Delta^p \to \Delta^0$ induces a weak equivalence $K_{0,\bullet}(f) \to K_{p,\bullet}(f)$.
\item Given $v \colon \Delta^p \to B$, let $K_{p,\bullet}(f,v)$ denote the simplicial set of consisting of pairs $(u,v)$. If $e_v$ denote the initial vertex of $v$, then by the argument on page 228 of \cite{mccleary}, restricting to $e_v \times \Delta^q$ induces a map $K_{p,\bullet}(f,v) \to K_{0,\bullet}(f,e_v)$, which is a weak equivalence if $f$ is a Serre fibration. Under these identifications, if a face map does not preserve the initial vertex, one uses the two weak equivalences $ K_{0,\bullet}(f,e_v) \leftarrow  K_{1,\bullet}(f,s_v) \rightarrow  K_{0,\bullet}(f,e_v')$ where $s_v$ is the $1$-simplex connecting the old initial vertex $e_v$ and the new initial vertex $e_v'$. This zigzag will induce well-defined map on homology.
\item It follows from the definition that for each $e \colon \Delta^0 \to B$, we have that $K_{p,\bullet}(f,e) \cong \mr{Sing}(f^{-1}(e))$.
\end{enumerate}

\begin{proposition}\label{propssserre} Let $\pi \colon E \to B$ be a $\theta$-framed manifold bundle with $B$ path-connected and $n$-dimensional $\theta$-framed manifold $F$ as fiber, $p \colon \tilde{B} \to B$ be a principal $G$-bundle with $G$ discrete, and $\cL$ be a $\bZ[G]$-module that is free as an abelian group. Then there is a spectral sequence
	\[E^2_{p,q} = H_p\left(B;\cH_q\left(\int_{F} A\right) \otimes \cL\right) \Rightarrow H_{p+q}\left(\int_{E \downarrow B,\cL} A\right)\]
where $\cH_q\left(\int_{F} A\right)$ is the local system of $B$ given by $b \mapsto \cH_q\left(\int_{\pi^{-1}(b)} A\right)$. This is natural in fiberwise embeddings of manifold bundles over $B$. One can restrict to fixed charge $\bc$ to get a spectral sequence
	\[E^2_{p,q} = H_p\left(B;\cH_q\left(\int^\bc_{F} A\right) \otimes \cL\right) \Rightarrow H_{p+q}\left(\int^\bc_{E \downarrow B,\cL} A\right)\]\end{proposition}

\begin{proof}By Lemma \ref{lemchainssimplicialobjects} we may assume $\cat{C} = \cat{Ch}$. As the permutation action commutes with the map $\tilde{\pi}_k \colon \mr{Emb}^\theta_{\tilde{B}}(\sqcup_k \bR^n,\tilde{E}) \to \tilde{B}$ as above, $[p] \mapsto K_{p,\bullet}(\tilde{\pi}_k)$ is a simplicial object in simplicial sets with commuting $G$ and $\mathfrak{S}_k$-actions. Taking singular chains and applying $- \otimes_{\bZ[G]} \cL$ gives us a simplicial object with $\mathfrak{S}_k$-action in chain complexes:
\[[p] \mapsto S_{p,\ast}(\pi_k,\cL) \coloneqq C_*(\mr{Emb}^\theta_{\tilde{B}}(\sqcup_k \bR^n,\tilde{E})) \otimes_{\bZ[G]} \cL\]
The condition that $\cL$ is free as an abelian group implies this chain complex is levelwise free.
	
Thus we can define a functor $\mathbf{S}_{\bullet}(\pi,\cL) \colon \cat{Ch} \to \cat{Fun}(\Delta^\mr{op},\cat{Ch})$ by 
\[C \mapsto \left( [p] \mapsto \bigsqcup_{k \geq 0} S_{p,\bullet}(\pi_k,\cL) \otimes_{\bZ[\mathfrak{S}_k]} C^{\otimes k} \right)\]
and we remark for later use that $\mathbf{S}_{p}(\pi,\cL)$ splits as a direct sum of functors $\mathbf{S}_p(\pi,\cL,v)$ for $v \colon \Delta^p \to B$, as $K_{p,\bullet}(\tilde{\pi}_k)$ splits as a simplicial set with $G$-action into a disjoint union over $G$-orbits of $\tilde{v} \colon \Delta^p \to \tilde{B}$ of the simplicial sets $K_{p,\bullet}(\tilde{\pi}_k,\tilde{v})$. Note that the $G$-orbits of $\tilde{v} \colon \Delta^p \to \tilde{B}$ are in bijection with the maps $v \colon \Delta^p \to B$.

Similarly, since the right $\bdt$-functor structure on $\mathbf{\tilde{E} \downarrow \tilde{B}}^\theta$ is $G$-equivariant and commutes with the map to $\tilde{B}$, composition of $\theta$-framed embeddings endows the functor $\mathbf{S}_{\bullet}(\pi,\cL)$ and functors $\mathbf{S}_p(\pi,\cL,v)$ with the structure of right $\bdt$-functors. Note that $\mathbf{S}_0(\pi,\cL)$ is isomorphic to $\mathbf{E \downarrow B}^{\theta,\cL}$ as a right $\bdt$-functor. 

We next consider the bisimplicial object in $\cat{Ch}$ given by
\[[p,q] \mapsto \mathbf{S}_p(\pi,\cL)(\bdt)^q(A)\]
Properties (i) and (ii) of $K_{\bullet,\bullet}$ imply that $[p] \mapsto \mathbf{S}_p(\pi) (\bdt)^q(A)$ is weakly equivalent to the constant simplicial object $[p] \mapsto \mathbf{S}_0(\pi,\cL) (\bdt)^q(A) \cong (\mathbf{E \downarrow B}^{\theta,\cL})(\bdt)^q(A)$. Hence first realizing the $[p]$-direction and using Lemma \ref{lemgeomrel}, and then the $[q]$-direction we obtain
\[\left\vert[p,q] \mapsto \mathbf{S}_p(\pi,\cL) (\bdt)^q(A)\right\vert \simeq \left\vert[q] \mapsto (\mathbf{E \downarrow B}^{\theta,\cL})(\bdt)^q(A)\right\vert = \int_{E \downarrow B,\cL} A\]

On other hand, if we first realize the $[q]$-direction and consider the geometric realization spectral sequence of Lemma \ref{lemgeomrealss} for the remaining $[p]$-direction, we get
\[\xymatrix{E^1_{st} = \bigoplus_{v \colon \Delta^s \to B} H_t\left(\left\vert[q] \mapsto \mathbf{S}_{p}(\pi,\cL,v)(\bdt)^q(A)\right\vert\right) \ar@{=>}[d] \\ H_{s+t}\left(\left\vert[p,q] \mapsto \mathbf{S}_p(\pi,\cL) (\bdt)^q(A)\right\vert\right)}\]
and using properties (iii) and (iv) of $K_{\bullet,\bullet}$ and Lemma \ref{lemgeomrel}, we have identifications 
\[\left\vert\mathbf{S}_{p}(\pi,v,\cL)(\bdt)^q(A)\right\vert \simeq \left\vert\mathbf{S}_{0}(\pi,e_v,\cL)(\bdt)^q(A)\right\vert \simeq \left(\int_{\pi^{-1}(e_v)} A\right) \otimes \cL\]
Thus the entries on the $E^1$-page are given by 
\[E^1_{st} = \bigoplus_{v \colon \Delta^s \to B} H_t\left(\int_{\pi^{-1}(e_v)} A\right) \otimes \cL\]
and the $d^1$-differential is given by taking the alternating sum of maps induced by restricting to faces of the simplex. If a face does not have the same initial vertex as the larger simplex, then one also parallel transports along the $1$-simplex connecting these two initial vertices. This is the definition of a chain complex computing the homology of $B$ with local coefficients in the local system as described in the statement of this proposition.

For naturality, we note that the construction of $\mathbf{S}_\bullet(\pi,\cL)$ and the various identifications made throughout the proof are natural in fiberwise $\theta$-framed embeddings of $\theta$-framed manifold bundles over $B$. By working in $\cat{Ch}_C$ to keep track of the charge, one proves that the spectral sequence can be restricted to components of fixed charge $\bc$.\end{proof}

\subsubsection{Stabilization maps} If $M$ is a connected non-compact $n$-dimensional manifold, McDuff introduced a map $C_k(M) \to C_{k+1}(M)$ of configuration spaces given by ``bringing in a particle from infinity'' \cite{Mc1}. This is usually called the \emph{stabilization map}. It can be generalized to a map $\int_M^\bc A \to \int_M^{\bc+\bc_0} A$ between the components of topological chiral homology of a charged algebra, as long as $\bc+\bc_0$ is defined in the partial monoid $C$.

The construction of this map uses the fact that topological chiral homology is functorial with respect to open $\theta$-framed embeddings to make space near the boundary to add in a new labeled embedded disk. We describe the construction only for $\cat{C} = \cat{Ch}$, as the constructions for $\cat{C} = \cat{Top}$ or $\cat{sSet}$ are similar. The resulting map will depend on two choices; (i) an element $a \in A(\bc_0)$, (ii) a $\theta$-framed embedding $\psi \colon \bR^n \sqcup M \to M$.

Topological chiral homology with coefficients in a fixed $\cdt$-algebra $A$ gives a symmetric monoidal functor from $\cat{Emb}^\theta$ to  $\cat{C}$, the former defined in Subsection \ref{subsubsectangent}. This functionality is induced by composition of embeddings. Thus the map $\psi \colon \R^n \sqcup M \m M$ induces a map 
\[\psi_* \colon \int_{\R^n} A  \otimes \int_M A \m \int_M A\]

When $A$ is a $C$-charged algebra, the monoidal structure is compatible with charge and so the maps $\psi_*$ are additive with respect to charge. Using the natural equivalence $A \m \int_{\R^n} A$ and restricting to specific charges we get maps 
\[\psi_*^0 \colon A(\bc_0) \otimes \int^\bc_M A \m  \int^{\bc+\bc_0}_M A\]
as long as $\bc+\bc_0$ is defined in $C$. We can now define the stabilization maps.

\begin{definition}\label{defstabmap}
Suppose that $\bc+\bc_0$ is defined in $\cat{C}$. Let $a \in A(\bc_0)$ and $\psi \colon \bR^n \sqcup M \to M$, then $t_{a,\psi} \colon\int^\bc_M A \m   \int^{\bc+\bc_0}_M A$ is the map given by $t_{a,\psi}(b)= \psi_*^0(a,b)$.
\end{definition}

This definition also makes sense for $\theta$-framed manifold bundles $\pi \colon E \to B$ with fiberwise embedding $\psi \colon (\bR^n \times B) \sqcup E \hookrightarrow E$ over $B$, as then one can repeat the construction above. The resulting map will similarly be denoted $t_{a,\psi} \colon \int_{E \downarrow B} A \to \int_{E \downarrow B} A$. 

\subsubsection{Homological stability for topological chiral homology} We can now give a precise definition of homological stability for topological chiral homology. We will only do this in the case that $C=\bN_0^d$. To get well-defined stabilization maps on the manifold bundles involved, we will only use stabilization maps for embeddings $\bR^n \sqcup M \to M$ which can be thought of as making space for $\bR^n$ by moving part of $M$ away from infinity. This is made precise as follows.

\begin{definition}\label{defendlike} An \emph{end-like embedding} is a path of embeddings $\psi_t$ in $\mr{Emb}^\theta(\bR^n \sqcup M,M)$ such that (i) there exists an exhausting smooth function $h \colon M \to [0,1)$ such that $\psi_t$ is the identity on $h^{-1}([0,t])$, (ii) $\psi_0|_{M}$ is isotopic to the identity.\end{definition}

\begin{remark}If $M$ is the interior of a compact connected $\theta$-framed manifold with boundary, one can construct embeddings $\psi_t$ as above that up to isotopy only depends on the choice of a component of the boundary. For general non-compact connected $\theta$-framed manifold the situation is more complicated: one can always construct an embedding $\psi_t$ as above. If $\dim M \geq 5$, one can construct embeddings $\psi_t$ as above that up to isotopy only depend on the choice of a component of the space of ends.\end{remark}

Suppose that $A$ is a $C$-charged $\cdt$-algebra, and hence by definition connected. For each representative $a_i \in A(\be_i)$ of $1 \in H_0(A(\be_i)) \cong \bZ$ (the isomorphism coming from the augmentation) and \emph{end-like embedding} $\psi_t$, we get a \emph{basic stabilization map} 
\[t_{a_i,\psi_0} \colon \int^{\bk}_M A \to \int^{\bk+\be_i}_M A.\]
The fact that topological chiral homology is an enriched functor implies this map only depends on the isotopy class of $\psi_t$. From now on we will drop $a_i$ and $\psi_0$ from the notation.

\begin{definition}\label{defhomstabtch} Let $C = \bN_0^d$, $\bk=(k_1,\ldots,k_d)$, and let $M$ be a $\theta$-framed  connected non-compact manifold.
\begin{enumerate}[(i)]
\item  Let $A$ be a $C$-charged $\cdt$-algebra in chain complexes, then we say that $A$ \emph{has homological stability on $M$} if there is a function $\rho \colon \bN_0 \to \bR_{\geq 0}$ with $\lim_{j \to \infty} \rho(j) =\infty$ such that all basic stabilization maps 
\[t_i \colon H_*\left(\int^{\bk}_M A\right) \to H_*\left(\int^{\bk+\be_i}_M A\right)\] are $\rho(k_i)$-equivalences. 
\item  Let $X$ be a $C$-charged $\cdt$-algebra in spaces or simplicial sets, then we say that $X$ \emph{has homological stability on $M$} if $C_*(X)$ has.\end{enumerate}
\end{definition}

\section{Cell attachments}\label{secencell} In this section we define $\cdt$-cell attachments and study their effect on homology. For the remainder of this section, we fix the tangential structure $\theta \colon W \to BO(n)$ and the monoid of charges $C$, and hence suppress them from our notation whenever convenient. For example, if we mention charged algebras, we mean $C$-charged $\cdt$-algebras.

\subsection{$\cdt$-cell attachments} Using the completion procedure of Subsection \ref{subsecpartialalg}, we can now define $\cdt$-cell attachments in $\cat{Ch}$ or $\cat{Ch}_C$. This case is the one used in main theorem, and $\cdt$-cell attachments in $\cat{Top}$ and $\cat{sSet}$ are defined analogously. See Remark \ref{remenattach} for a discussion of the relationship between our construction and similar constructions in the literature. 

Let $A$ be a charged algebra. Fixing a cycle $b_{N-1} \in A(\bc)$ of degree $N-1$, we define the partial algebra $A \oplus e_N$ by declaring that $e_N$ is of charge $\bc$ and $d(e_N) = b_{N-1}$ and declaring no operation on $e_N$ is  defined except the identity. We now make this definition precise:

\begin{definition}\label{defpartialalgebracomp} Let $D = A \oplus e_N$ be the chain complex with underlying graded abelian group $A \oplus e_N$ with differential extending that of $A$ by setting $d(e_N) = b_{N-1}$. This can be given the structure of a partial charged algebra by setting 
\[D_p = \bdt((\bdt)^p(A) \oplus (\mr{id}^p \otimes e_{N})) \subset B_p(\bdt,\bdt,D)\] 

This is a particular case of Example \ref{exampcomp}, where $\mr{Comp}_1 = \bdt(A) \oplus (\mr{id} \otimes e_N) \subset \bdt(A \oplus e_N)$ and $c_1 \colon \mr{Comp}_1 \to A$ is given by the $\cdt$-action on $\bdt(A)$ and by $\mr{id} \otimes e_N \mapsto e_N$. It is clear that condition (i) holds, and (ii) follows from the fact that $\mr{Comp}_2 = \mr{Comp}_2' = (\bdt)^2(A) \oplus (\mr{id}^2 \otimes e_{N})$. We then have that $\mr{Comp}_p = (\bdt)^p(A) \oplus (\mr{id}^p \otimes e_{N})$.\end{definition}

We remark that $D_\bullet$ is Reedy-cofibrant if the underlying chain complex of $A$ is cofibrant. We can now define charged algebras obtained by $\cdt$-cell attachments.

\begin{definition}Let $b_{N-1}$ and $A \oplus e_N$ be as before. We let the charged algebra $A \uplus e_N$ be the completion of the partial charged algebra $A \oplus e_N$, i.e. $A \uplus e_N = |D_\bullet|$, and say that $A \uplus e_N$ is \emph{obtained from $A \oplus e_N$ by attaching an $\cdt$-cell $e_N$ to $A$}.
\end{definition}


Our next goal is to show that one may construct a map from $A \uplus e_N$ to $B$ out of a map $f \colon A \to B$ of charged algebras and an element $e \in B(k)_N$ such that $d(e) = f(b_{N-1})$.

\begin{lemma}\label{lemattachmap} Let $A \uplus e_N$ be the charged algebra obtained by attaching an $\cdt$-cell $e_N$ to $A$ along $b_{N-1}$ in charge $k$, $f$ a map $A \to B$ and $e$ an element $e \in B(k)_N$ such that $d(e) = f(b_{N-1})$. There exists a map $\hat{f} \colon A \uplus e_N \to B$ of charged algebras such that $\hat{f}(e_N) = e$ and such that the following diagram commutes:
\[\xymatrix{A \uplus e_N \ar[r]^{\hat{f}} & B \\
\bar{A} \ar[u] \ar[r]_\simeq & \ar[u]_f A}\]
where $\bar{A} \to A\uplus e_N$ is induced by the inclusion $A_\bullet \to D_\bullet$ of in the notation of Example \ref{examptrivialpartial} and Definition \ref{defpartialalgebracomp} (which will be repeated in the proof).\end{lemma}

\begin{proof}Recall that $A\uplus e_N$ is by definition given by
	\[A\uplus e_N \coloneqq \left\vert [p] \mapsto D_p = \bdt((\bdt)^p(A) \oplus \mr{id}^p \otimes e_{N})\right\vert\]
By Example \ref{examptrivialpartial}, $A$ and $B$ has partial charged algebra structures $A_\bullet = B_\bullet(\bdt,\bdt,A)$ and $B_\bullet = B_\bullet(\bdt,\bdt,B)$, and we have weak equivalences $\bar{A} = |A_\bullet| \to A$ and $\bar{B} = |B_\bullet| \to B$.

We construct a map $D_\bullet \to B_\bullet$ out of the map $f \colon A \to B$ and the element $e \in B(k)_N$ such that $d(e) = f(b_{N-1})$. Then geometric realization and the augmentation gives us the map $A \uplus e_N \to B$ defined as $\bar{D} \to \bar{B} \to B$. 

To construct the map $D_\bullet \to B_\bullet$ we only need to define compatible maps 
\[D_p = \bdt\left((\bdt)^p(A) \oplus (\mr{id}^p \otimes e_{N})\right) \to (\bdt)^{p+1}(B) = B_p\]
This is easy, as we can apply $\bdt$ to the direct sum of the maps $(\bdt)^p(f)$ and $\mr{id}^p \otimes e_N \mapsto \mr{id}^p \otimes e$. This has the desired properties.\end{proof}

\begin{remark}\label{remenattach} This construction hints at the fact that an alternative definition of an $\cdt$-cell attachment is given by the push out 
	\[A \cup_{\bdt(S^{N-1})} \bdt(D^{N})\]
in the $(\infty,1)$-category of $C$-charged $\cdt$-algebras. Equivalently one can work in a suitable model category of $\cdt$-algebras, see \cite{bergermoerdijk}, \cite{fressebook}, \cite{Sa}, or \cite{horel}. A definition of $\cdt$-cell structures in terms of push outs is used to define cofibrant replacement for algebras over operads, e.g. Section 12.3 of \cite{fressebook}. We also believe that there is an $(\infty,1)$-category (or model category) of partial algebras over an operad such that the forgetful functor from algebras to partial algebras and the completion functor from partial algebras to algebras form an $(\infty,1)$-adjunction (a Quillen adjunction). \end{remark}  

\begin{remark}\label{remmultiplecells0} The construction of this subsection can be generalized to the attachment of any collection of $\cdt$-cells $\{e_N^j\}_{j \in J}$ attached along $\{b_{N-1}^j\}_{j \in J}$: one defines the structure of a partial charged algebra on $A \oplus \bigoplus_{j \in J} e_N^j$ in the same way and defines $A \uplus \biguplus_{j \in J} e_N^j$ to be its completion. Lemma \ref{lemattachmap} generalizes in the expected way.\end{remark}


\subsection{The homology after an $\cdt$-cell attachment}\label{subsechomologyaftercellattachment} The goal of this subsection is to show how a $\cdt$-cell attachment affects homology. This will done by applying two spectral sequences, and we start by describing the first one. To do so we need to make a few definitions.

Let $F_r(M)$ denote the configuration space of $r$ ordered particles in $M$ and $C_r(M)$ the configuration space of $r$ unordered particles in $M$. Note that there is a manifold bundle $E_r(M)$ over $C_r(M)$ given by the subspace of $M \times C_r(M)$ of $(m,\{m_1,\ldots,m_r\})$ such that $m \neq m_i$ for all $1 \leq i \leq r$. The fiber of this manifold bundle over a point $\{m_1,\ldots,m_r\} \in C_r(M)$ is $M \setminus \{m_1,\ldots,m_r\}$. 

We will need to consider a variation of this manifold bundle, with $C_r(M)$ replaced by a configuration space with labels $C_r^\theta(M)$. The labels will come from a space $\theta(TM)$ over $M$; $\theta(TM) = \mr{Emb}^\theta(\bR^n,M)$ and $\theta_{TM} \colon \theta(TM) \to M$ sends a $\theta$-framed embedding $\psi$ to $\psi(0)$, the image of the origin.

\begin{lemma}The map $\theta_{TM} \colon \theta(TM) \to M$ is a Serre fibration with path-connected fibers.\end{lemma}

\begin{proof}The map $\theta_{TM}$ equals the composition
	\[ \mr{Emb}^\theta(\bR^n,M) \to  \mr{Emb}(\bR^n,M) \to M\]
	The left-hand map is a Serre fibration as it is the pull back of the fibration $\mr{Bun}^\theta(T\bR^n,TM) \to \mr{Bun}(T\bR^n,TM)$ and the right-hand map is a Serre fibration by the parametrized isotopy extension theorem.
	
	The fiber of $m \in M$ is the space of $\theta$-framed embeddings $\bR^n \to M$ sending the origin to $m$. Picking a chart around $M$, this weakly equivalent to a space of $\theta$-framed embeddings $\bR^n \to \bR^n$, the latter with a possibly non-trivial $\theta$-framing. This is path-connected by Lemma \ref{lemthetaconn}.
\end{proof}

We now define labeled configuration spaces.

\begin{definition} \label{colconfig} Fix $\bk=(k_1,\ldots,k_d) \in \N_0^d$ and $k=k_1+\ldots+k_d$. Let $\pi \colon E \to M$ be a Serre fibration. The \emph{configuration space of $k$ ordered particles with labels in the fibration $\pi$}, denoted $F_k^\pi(M)$, is given by the following subspace of $E^k$:
	\[F_k^\pi(M) \coloneqq \left\{ (e_1,\ldots,e_k) \middle|
	\pi(e_i) \neq \pi(e_j) \text{ for $i \neq j$}	\right\}\]
	
\noindent The \emph{configuration space of $\bk$-colored particles with labels in the fibration $\pi$}, denoted $C_\bk^\pi(M)$, is the quotient $F_k^\pi(M)/\mathfrak{S}_{\bk} \subset E^k/\fS_k$ where $\fS_{\bk} = \prod_i \fS_{k_i} \subset \fS_{k}$ and $\fS_k$ acts diagonally. For $d=1$, we denote this by $C_\bk^\pi(M)$.
\end{definition}  

A superscript $\theta$ will mean we take the map $\pi$ to be the map $\theta_{TM} \colon \theta(TM) \to M$. The quotient map $q \colon F_r^\theta(M) \to C_r^\theta(M)$ is a principal $\mathfrak{S}_k$-bundle. This can be seen by noting that there is a map $f \colon C_r^\theta(M) \to C_r(M)$ which forgets the labels, and that $q$ is obtained by pulling back the principal $\mathfrak{S}_k$-bundle $F_r(M) \to C_r(M)$ along $f$. We let $E_r^\theta(M)$ denote the pullback of $E_r(M)$ along $f$. The fiber of $E_r^\theta(M)$ over $\{\psi_1,\ldots,\psi_r\} \in C_r^\theta(M)$ is given by $M \setminus \{\psi_1(0),\ldots,\psi_r(0)\}$.

\begin{lemma}Every end-like embedding $\bR^n \sqcup M \hookrightarrow M$ as in Definition \ref{defendlike} induces a unique isotopy class of fiberwise embeddings $\bR^n \times C_r^\theta(M) \sqcup E_r^\theta(M)\to E_r^\theta(M)$.\end{lemma}

\begin{proof}Recall that an end-like embedding is a continuous map $\psi_t \colon [0,1) \to \mr{Emb}^\theta(\bR^n \sqcup M,M)$ that in particular has the property there exists an exhausting proper smooth function $h \colon M \to [0,1)$ such that $\psi_t$ is the identity on $h^{-1}([0,t])$. Fix such an $h$.

Pick a smooth function $\eta \colon C_r^\theta(M) \to [0,1)$ such that the configuration $x \in C_r(M)$ is contained in $h^{-1}([0,\eta(x)])$. Then the map
\[\bR^n \times C_r^\theta(M) \sqcup E_r^\theta(M) \to E_r^\theta(M)\]
defined over $x \in C_r^\theta(M)$ by $\psi_{\eta(x)}$ is a fiberwise embedding over $C_r^\theta(M)$. To show it is independent of the choice of $h$ and $\eta$ up to isotopy, note that for any pairs $h_i$, $\eta_i$, $i=0,1$ we can find $h'$ and $\eta'$ such that $h' \leq \min(h_0,h_1)$ and $\eta' \geq \max(\eta_0,\eta_1)$. Then linearly interpolating the functions $h_i$ and $h'$, and the functions $\eta'$ and $\eta$ gives isotopies from the fiberwise embedding constructed out of $h_i$ and $\eta_i$ to that constructed out of $h'$ and $\eta'$.\end{proof}

As topological chiral homology is enriched functor, the upshot of the lemma is that given a basic stabilization map for $M$ there is a map 
\[t_i \colon\int^{\bc}_{E_r^\theta(M) \downarrow C_r^\theta(M)} A \m \int^{\bc+\be_i}_{E_r^\theta(M) \downarrow C_r^\theta(M)} A\]
that is well-defined on homology. We will apply the following spectral sequence to it, which is constructed by applying Proposition \ref{propssserre} to the $\theta$-framed manifold bundle $E^\theta_r(M) \to C^\theta_r(M)$. It uses the principal $\mathfrak{S}_r$-bundle $q \colon F^\theta_r(M) \to C^\theta_r(M)$ to define local coefficients in a $\bZ[\mathfrak{S}_r]$-module $\cL$.

\begin{proposition}\label{propsscrtheta} Let $\cL$ be a $\bZ[\mathfrak{S}_r]$-module that is free as an abelian group. Then there is a spectral sequence 
\[E^2_{p,q} = H_p\left(C_{r}^\theta(M);\cH_q\left(\int^{\bc}_{M \setminus \{\text{$r$\text{ points}}\}} A\right) \otimes \cL\right) 
\Rightarrow H_{p+q}\left(\int^{\bc}_{E_r^\theta(M) \downarrow C_r^\theta(M),\cL} A\right)\]
where $\cH_q(\int^{\bc}_{M \setminus \{\text{$r$ points}\}} A)$ is the local system over $C_r^\theta(M)$ obtained by pulling back the local system over $C_r(M)$ with fiber over the point $\{m_1,\ldots,m_r\} \in C_r(M)$ given by $H_q(\int^{\bc}_{M \setminus \{m_1,\ldots,m_r\}} A)$. Stabilization maps induce a map of spectral sequences.
\end{proposition}

Now we continue to describe the second spectral sequence involved in the study of the homology of a $\cdt$-cell attachment. The identification of its $E^1$-page involves a new right $\bdt$-functor. We let 
\[\mr{Emb}^{\theta,\mr{disj}}_{F^\theta_r(M)}\left(\sqcup_k \bR^n,q^*E^\theta_r(M)\right) \subset \mr{Emb}^{\theta}_{F^\theta_r(M)}\left(\sqcup_k \bR^n,q^* E^\theta_r(M)\right)\]
be the subspace of those $\theta$-framed embedding $(\phi_1,\ldots,\phi_k)$ over those points $(\psi_1,\ldots,\psi_r) \in F^\theta_r(M)$ such that the images of the $\phi_i$ for $1\leq i \leq k$ and $\psi_j$ for $1 \leq j \leq r$ are all pairwise disjoint. Since this property is preserved by precomposition of the $\phi_i$ with $\theta$-framed embeddings, we can define a right $\bdt$-functor $\mathbf{E_r^\theta(M) \downarrow C_r^\theta(M)}^{\mr{disj},\cL} \colon \cat{Ch} \to \cat{Ch}$  by
\[C \mapsto \bigoplus_{k \geq 0} \left(C_*(\mr{Emb}^{\theta,\mr{disj}}_{F^\theta_r(M)}(\sqcup_{k}\bR^n,q^* E^\theta_r(M))) \otimes_{\bZ[\mathfrak{S}_r]} \cL\right) \otimes_{\fS_k} C^{\otimes k}\]
such that the inclusion induces a map $\mathbf{E_r^\theta(M) \downarrow C_r^\theta(M)}^{\mr{disj},\cL} \to \mathbf{E_r^\theta(M) \downarrow C_r^\theta(M)}^{\cL}$ of right $\bdt$-functors.

\begin{lemma}\label{lemdisjcomp} This map $\mathbf{E_r^\theta(M) \downarrow C_r^\theta(M)}^{\mr{disj},\cL} \to \mathbf{E_r^\theta(M) \downarrow C_r^\theta(M)}^\cL$ is a natural weak equivalence.\end{lemma}

\begin{proof}By Lemma \ref{lemweakequivfunctor} it suffices to prove that for all $r \geq 0$ the inclusion
\[\mr{Emb}^{\theta,\mr{disj}}_{F^\theta_r(M)}(\sqcup_k \bR^n,q^*E^\theta_r(M)) \hookrightarrow \mr{Emb}^{\theta}_{F^\theta_r(M)}(\sqcup_k \bR^n,q^*E^\theta_r(M))\]
is a $\mathfrak{S}_r$-equivariant weak equivalence. To see this is case, consider a commutative diagram 
	\[\xymatrix{\partial D^i \ar[r]^-f \ar[d] & \mr{Emb}^{\theta,\mr{disj}}_{F^\theta_r(M)}(\sqcup_k \bR^n,q^*E^\theta_r(M)) \ar[d] \\
	D^i \ar[r]_-F &  \mr{Emb}^{\theta}_{F^\theta_r(M)}(\sqcup_k \bR^n,q^*E^\theta_r(M))}\]

We may assume that there is an open neighborhood $U$ of $\partial D^i$ in $D^i$ such that $F|_U$ lands in $\mr{Emb}^{\theta,\mr{disj}}_{F^\theta_r(M)}(\sqcup_k \bR^n,q^* E^\theta_r(M)$. Pick a family of origin-preserving self-embeddings $\lambda_t \colon \bR^n \to \bR^n$ for $t \in [0,\infty)$ such that $\lambda_0 = \mr{id}$ and $\lambda_t$ has image in the open ball of radius $1/t$. Since $\mr{Emb}^\theta(\bR^n,\bR^n) \to \mr{Emb}(\bR^n,\bR^n)$ is a fibration, we can lift this to a family of $\theta$-framed embeddings starting at the identity.

Then there exists a real number $T \geq 0$ such that after precomposing each of the $(k+r)$ embeddings in $F(d)$ with $\lambda_{T'}$ to $T' \geq T$ their images are pairwise disjoint. Now pick a continuous function $\eta \colon D^i \to [0,1]$ with support in $U$ and $1$ on $\partial D^i$, and let $(\lambda_t)_*$ denote precomposition of all $k$ embeddings of $\bR^n$ with $\lambda_t$. Then $(\lambda_{(1-\eta)T})_* \circ F$ is $\mathfrak{S}_r$-equivariantly homotopic to $F$ rel $\partial D^i$ and has image in $\mr{Emb}^{\theta,\mr{disj}}_{F^\theta_r(M)}(\sqcup_k \bR^n,q^*F^\theta_r(M))$.
\end{proof}

The statement of the next proposition will involve the homology of components of charge $\bc-p\bk$, which by definition is the unique element of $C$ such that $(\bc-p\bk)+p\cdot \bk = \bc$, where $p \cdot \bk$ means adding $\bk$ to itself $p$ times. This may not exist for large $p$ and if it does not, the homology will be defined to be $0$. Let $\bZ_{\pm 1}$ be the $\bZ[\mathfrak{S}_r]$-module induced by the sign homomorphism $\mathfrak{S}_r \to \{\pm 1\}$. 

\begin{proposition}\label{propsscompl} Suppose $e_N$ is attached to $A$ in charge $\bk$ along $b_{N-1}$, then there is a spectral sequence
\[E^1_{pq} = H_{q-p(N-1)}\left(\int^{\bc-p\bk}_{E^\theta_p(M) \downarrow C^\theta_p(M)} A;\bZ_{\pm}^{\otimes N}\right) \Rightarrow H_{p+q}\left(\int^\bc_M A \uplus e_N\right)\]
and if $b_{N-1}=0$ the spectral sequence collapses at the $E^1$-page. If $M$ is path-connected, the map $d^1 \colon E^1_{1q} \to E^1_{0q}$ is induced by the map $H_{*-N+1}\left(\int^{\bc-\bk}_{M \setminus m} A\right) \to H_*\left(\int^\bc_M A\right)$ by adding a particle with label $b_{N-1}$ at $m \in M$.  Stabilization maps induce a map of spectral sequences. 
\end{proposition}

\begin{proof}By Lemma \ref{lemcompsimp}, $\int^\bc_M A \uplus e_N$ is equivalent to the realization of the simplicial chain complex $\mathbf{M}^\theta(\mr{Comp}_\bullet)$. Recall that $\mr{Comp}_p = (\bdt)^{p}(A) \oplus (\mr{id}^p \otimes e_N)$, so that we can define a filtration of simplicial chain complexes by setting $\mathcal F_j \mathbf{M}^\theta\left(\mr{Comp}_p\right)$ to be
\[\bigoplus_{0 \leq r \leq j} \bigoplus_{k \geq 0} C_*(\mr{Emb}^\theta(\sqcup_{k+r} \bR^n,M)) \underset{\mathfrak{S}_{k} \times \mathfrak{S}_r}{\otimes} \left((\bdt)^{p}(A)^{\otimes k} \otimes (\mr{id}^p \otimes e_N)^{\otimes r}\right)\]
That is, we allow at most $j$ copies of the new cell.

This induces a filtration of chain complexes upon realization. The spectral sequence associated to this filtration is one in the statement of this proposition, and is constructed as in Section 2.2 of \cite{mccleary}. Its convergence follows from the finiteness of the filtrations in each homological degree. The associated graded is independent of $b_{N-1}$, so it suffices to identify it in the case $b_{N-1} = 0$. In that case the filtration exactly comes from a direct sum decomposition with summands the subcomplexes of $|\mathbf{M}^\theta(\mr{Comp}_\bullet)|$ with exactly $r$ copies of $e_N$'s. More precisely, we have that
\begin{align*}\mathbf{M}^\theta\left(\mr{Comp}_p\right) &\cong \bigoplus_{r \geq 0} (\mathcal F_r \mathbf{M}^\theta \left(\mr{Comp}_p\right))/(\mathcal F_{r-1} \mathbf{M}^\theta\left( \mr{Comp}_p\right)) \\
&\cong \bigoplus_{r \geq 0} \bigoplus_{k \geq 0} C_*(\mr{Emb}^\theta(\sqcup_{k+r} \bR^n,M)) \underset{\mathfrak{S}_{k} \otimes \mathfrak{S}_r}{\otimes} \left((\bdt)^{p}(A)^{\otimes k} \otimes (\mr{id}^p \otimes e_N)^{\otimes r}\right) \\
&\cong \bigoplus_{r \geq 0} \left(\mathbf{E_r^\theta(M) \downarrow C_r^\theta(M)}^{\mr{disj},\bZ_{\pm 1}^{\otimes N}}\right)\left((\bdt)^p(A)\right)[rN] \\
&= \bigoplus_{r \geq 0} B_p\left(\mathbf{E_r^\theta(M) \downarrow C_r^\theta(M)}^{\mr{disj},\bZ_{\pm 1}^{\otimes N}},\bdt,A\right)[rN]
\end{align*}
and $|\mathbf{M}^\theta(\mr{Comp}_\bullet)|$ is isomorphic to the direct sum over $r$ of the realization of simplicial chain complexes on the last line. This direct sum decomposition causes the collapse at the $E^1$-page in the case $b_{N-1} = 0$ and allows us to identify the $E^1$-page: Lemma's \ref{lemweakequivfunctor} and \ref{lemdisjcomp} to show that the inclusion
\[B_\bullet\left(\mathbf{E_r^\theta(M) \downarrow C_r^\theta(M)}^{\mr{disj},\bZ_{\pm 1}^{\otimes N}},\bdt,A\right) \to B_\bullet\left(\mathbf{E_r^\theta(M) \downarrow C_r^\theta(M)}^{\bZ_{\pm 1}^{\otimes N}},\bdt,A\right)\]
is a weak equivalence. The $r$th column on the $E^1$-page is given by taking the homology of the associated chain complex, so that by comparing to the definition of topological chiral homology of manifold bundles with local coefficients we obtain the identification
\[E^1_{pq} = H_{*-pN+p}\left(\int^{\bc-p\bk}_{E^\theta_p(M) \downarrow C^\theta_p(M),\bZ_{\pm 1}^{\otimes N}} A\right)\]

We recall that the $d^1$-differential is defined a class $[x] \in E^1_{pq}$ on the $E^1$-page as follows: one lifts it to a cycle $x$ in the associated graded $\cF_p/\cF_{p-1}$, i.e. $d(x) \in \cF_{p-1}$. Then $d_1([x]) = [d(x)]$. In our case, an element of $E^1_{0q}$ can be represented by taking a cycle in $\int^{\bc-\bk}_{M \setminus m} A$ and adding a copy of $e_{N}$ at $m$. The differential then vanishes on the cycle and sends $e_N$ to $b_{N-1}$.
\end{proof}

A first consequence of this is that we can easily compute the effect of $\cdt$-cell attachments in homology in degrees less than or equal to the degree of the cell.

\begin{corollary}\label{corattachexactseq} Suppose that $e_N$ is attached in charge $\bk$ along $b_{N-1}$.
\begin{enumerate}[(i)] \item If $\bc<\bk$ (that is, there exists no $\bc'$ such that $\bc = \bc'+\bk$), we have that $H_*(A(\bc)) \to H_*((A \uplus e_N)(\bc))$ induces an isomorphism.
\item In general we have that $H_*(A) \to H_*(A \uplus e_N)$ induces an isomorphism for $* \leq N-2$.
\item If $\bc \geq \bk$, (that is, there exists a $\bc'$ such that $\bc = \bc'+\bk$), there is an exact sequence
\[H_N(A(\bc)) \to H_N((A \uplus e_N)(\bc)) \to \bZ \to H_{N-1}(A(\bc)) \to H_{N-1}((A \uplus e_N)(\bc)) \to 0\]
where the map $\bZ \to H_{N-1}(A(\bc))$ is induced by $1 \mapsto t_{\bc'}(b_{N-1})$. In particular, attaching an $\cdt$-cell of degree $N$ with trivial attaching map does not affect $H_{N-1}$.
\item If $b_{N-1} = 0$ the map $H_N(A(\bc)) \to H_N((A \uplus e_N)(\bc))$ is split injective and thus $H_N((A \uplus e_N)(\bc)) \cong H_N(A(\bc)) \oplus \bZ$. \end{enumerate}\end{corollary}

\begin{proof}Parts (i), (ii) and (iii) are a consequence of studying the initial parts of spectral sequences in Proposition \ref{propsscrtheta} and Proposition \ref{propsscompl} as follows. For general $M$, on the $E^1$-page of the spectral sequence of Proposition \ref{propsscompl} we have that (a) $E^1_{0,q} = H_q(\int^\bk_M A)$, (b) $E^1_{p,q} = 0$ if $p \geq 1$ and $q \leq N-2$, (c) $E^1_{p,q} = 0$ if $p \geq 1$ if $\bc<\bk$, and (d) $E^1_{1,N-1} = \bZ$ if there exists a $\bc'$ such that $\bc=\bc'+\bk$. Now specialize to $M=\R^n$ and use that $A(\bc)=\int^\bc_{\R^n} A$. 

For part (iv), one just needs to note that if $b_{N-1} = 0$ there is a map $A \uplus e_N \to A$ sending $e_N$ to $0$, which on homology splits the map $A \simeq \bar{A}_\bullet \to A \uplus e_N$.\end{proof}

\begin{remark} The results in this subsection have analogues for the attachment of any collection of $\cdt$-cells of the same dimension and charge, as in Remark \ref{remmultiplecells0}. In that case $C_r^\theta(M)$ and its relatives get replaced by configuration spaces with labels in the indexing set of the $\cdt$-cells, and further modifications are as one expects. \label{remmultiplecells1}
\end{remark}

\subsection{$\cdt$-cell attachments preserve homological stability} Computing the homology in higher degrees is harder, but we can use the spectral sequences to deduce that $\cdt$-cell attachments preserve homological stability. To do so, we restrict ourselves to monoids of charges given by $C = \bN_0^d$. We will prove a slightly stronger statement than we will eventually need.

\begin{corollary}\label{corattachstab} Let $C = \bN_0^d$, $n \geq 2$ and let $M$ be a $\theta$-framed non-compact connected manifold. Suppose that $A$ is a charged algebra such that for each $r \geq 0$ the space $\int_{M \setminus \{r \text{ points}\}} A$ has homological stability in a range that does not depend on $r$, then $\int_M (A \uplus e_N)$ has homological stability.

Moreover, fix $1 \leq i \leq d$ and let $\rho \colon C \m \R_{\geq 0}$ be a function satisfying (i) $\rho(\mathbf{a})<\rho(\mathbf{b})$ if $a_i < b_i$ and $a_j \geq b_j$ for $j \neq i$, (ii) $\rho(\mathbf{a}+\mathbf{b}) \leq \rho(\mathbf{a}) + \rho(\mathbf{b})$ for all $\mathbf{a},\mathbf{b} \in C$. Let $e_N$ be attached in charge $\bk$ with $\rho(\bk) \leq N$. Assume that for all $r \geq 0$ we have 
\[t_i \colon \int^{\bc}_{M \setminus \{r\text{ points}\}} A \m \int^{\bc+\be_i}_{M \setminus \{r \text{ points}\}} A\]
is a $\rho(\bc)$-equivalence. Then 
\[t_i \colon \int^{\bc}_M A \uplus e_N \m \int^{\bc+\be_i}_M A \uplus e_N\] is also a $\rho(\bc)$-equivalence. 

\end{corollary}

\begin{proof}First note that the second part of this corollary implies the first. The idea of the proof is to apply the two spectral sequences of Propositions \ref{propsscrtheta} and \ref{propsscompl} above. 

Firstly, we prove that  $H_*\left(\int^{\bj}_{E_r^\theta(M) \downarrow C_r^\theta(M),\cL} A\right)$ has homological stability in the same range as $H_*\left(\int^{\bj}_{M \setminus \{r\text{ points}\}} A\right)$. To do so we apply Proposition \ref{propsscrtheta} to the  stabilization map 
\[H_*\left(\int^{\bj}_{E_r^\theta(M)\downarrow C_r^\theta(M),\cL} A\right) \to H_*\left(\int^{\bj+\be_i}_{E_r^\theta(M)\downarrow C_r^\theta(M),\cL} A\right).\] 
The stabilization map induces a map of spectral sequences converging to the domain and target of the stabilization map, which on the $E^2$-pages
\[\xymatrix{E^2_{p,q} = H_p\left(C_{r}^\theta(M);\cH_q\left(\int^{\bj}_{M \setminus \{\text{$r$ points}\}} A\right) \otimes \cL\right) \ar[d] \\ (E')^2_{p,q} = H_p\left(C_{r}^\theta(M);\cH_q\left(\int^{\bj+\be_i}_{M \setminus \{\text{$r$ points}\}} A\right) \otimes \cL\right)}\]
is induced by the map $\cH_q\left(\int^{\bj}_{M \setminus \{\text{$r$ points}\}} A\right) \otimes \cL \to \cH_q\left(\int^{\bj+\be_i}_{M \setminus \{\text{$r$ points}\}} A\right) \otimes \cL$ of local system. By hypothesis, if we fix a $Q$ for $j \gg 0$ this is an isomorphism for all $q \leq Q$, and hence the map on $E^2$-pages is an isomorphism for $q \leq Q$ as well. A spectral sequence comparison argument then implies the stabilization map is an isomorphism for $q \leq Q$.

If we make this argument quantitative, we get that for all $r$ and all $\bj$, the map 
\[t \colon \int^{\bj}_{E_r^\theta(M) \downarrow C_r^\theta(M),\cL} A \m \int^{\bj+\be_i}_{E_r^\theta(M) \downarrow C_r^\theta(M),\cL} A\]
is a $\rho(\bj)$-equivalence.

Next, we prove the corollary being applying the spectral sequence of Proposition \ref{propsscompl}. The stabilization map induces a map of spectral sequences converging to the induced map $H_*(\int^{\bc}_M A \uplus e_N) \to H_*(\int^{\bc+\be_i}_M A \uplus e_N)$. On the $E^1$-pages it is given by the following.
\[\xymatrix{E^1_{p,q} = H_{q-(N-1)p}\left(\int^{\bc-p\bk}_{E_p^\theta(M) \downarrow C_p^\theta(M),\bZ_{\pm 1}^{\otimes N}} A\right) \ar[d] \\ (E')^1_{p,q} = H_{q-(N-1)p}\left(\int^{\bc+\be_i-p\bk}_{E_p^\theta(M)) \downarrow C_p^\theta(M),\bZ_{\pm 1}^{\otimes N}} A\right)}\]

We will first prove it is a surjection for $p+q \leq \rho(\bc)$. There are now three cases to consider: (i) $c_i-pk_i<-1$ or $c_j - pk_j \leq -1$ for all $j \neq i$, (ii) $c_i-pk_i=-1$ and $c_j - pk_j \geq 0$ for all $j \neq i$, and (iii) $c_j-pk_j \geq 0$ for all $j$.
\begin{enumerate}[(i)]
\item In the first case we are considering a map between zero chain complexes, which is an isomorphism.
\item In the second case we are considering a map from a zero chain complex to a non-zero one. This is obviously an isomorphism in negative homological degrees, i.e. if $q-(N-1)p < 0$, but we need it to be a surjection for all $p+q \leq \rho(\bc)$. Note that $p\rho(\bk) \geq \rho(p\bk)$ by part (ii) of our assumptions on $\rho$ and $\rho(p\bk) > \rho(\bc)$ by part (i) of our assumptions on $\rho$ since $c_i = pk_i - 1 < pk_i$ but $c_j = pk_j$. So it suffices to prove that it is a surjection for $p+q \leq p\rho(\bk)$ if $q-(N-1)p \geq 0$. This is true because the inequality $q-(N-1)p \geq 0$ is equivalent to $p+q \leq pN$, and we assumed that $\rho(\bk) \leq N$.
\item Finally, in the third case, the map on the $(p,q)$-entry of $E^1$-page is an surjection if $q-(N-1)p \leq \rho(\bc-p\bk)$, or in other words if $p+q \leq \rho(\bc-p\bk)+Np$. We want a surjection for all $p+q \leq \rho(\bc)$, which happens if $\rho(\bc) \leq \rho(\bc-p\bk)+Np$ for all $p$. By our assumption regarding where the $\cdt$-cell is attached, $N \geq \rho(\bk)$ and this inequality is implied by $\rho(\bc) \leq \rho(\bc-p\bk) + \rho(\bk)p$ for all $p$. This follows from $\rho(\mathbf{a}+\mathbf{b}) \leq \rho(\mathbf{a})+\rho(\mathbf{b})$ for all $\mathbf{a},\mathbf{b} \in C$.
\end{enumerate} 

Thus the map on the $E^1_{pq}$-pages is a surjection for $p+q \leq \rho(\bc)$. A similar argument shows that it is an isomorphism for $p+q \leq \rho(\bc)-1$. The corollary now follows from spectral sequence comparison.\end{proof}

\begin{remark} Using Remark \ref{remmultiplecells1}, the results in this subsection have analogues for the attachment of any collection of $\cdt$-cells of the same dimension and charge, as in Remark \ref{remmultiplecells0}. \label{remmultiplecells2}
\end{remark}

\subsection{Bounded generation} The definitions of the previous subsection lead us to the definition of bounded generation, which we will only give for $C = \bN_0^d$. We treat the degree zero case separately and put restriction on the $\cdt$-cells of degree one to make sure we stay in the category of charged algebras.

\begin{definition}\label{defboundedcellular} Let the monoid of charges be given by $C = \bN_0^d$. \begin{enumerate}[(i)]
\item We say that a charged algebra $A$ is \emph{obtained from $A'$ by $\cdt$-cell attachments of degree $N$ in charge $\leq \bc$} if there exists a possibly infinite collection of $\cdt$-cells $\{e^N_i\}_{i \in I}$ attached along $\{b_i\}_{i \in I}$ of homological degree $N-1$ and charge $\leq \bc$, such that there is an isomorphism
\[A' \uplus \biguplus_{i \in I} e_i^N \to A\]
If $N=1$ we require the attaching maps to be trivial, so that the underlying $C$-charged chain complexes remain connected.
\item We say that a $C$-charged algebra $A$ is \emph{bounded cellular of degree $\leq 0$} if $A = 0 \uplus \biguplus_{i=1}^d e^0_i$ with each $e^0_i$ of degree $0$.
\item We say that a charged algebra $A$ is \emph{bounded cellular of degree $\leq N$} if $A$ is obtained from an $\cdt$-algebra which is bounded cellular of degree $\leq 0$, by $\cdt$-cell attachments of $\cdt$-cells such that those of homological degrees $n \leq N$ are attached in finitely many charges $\bk$.
\end{enumerate}\end{definition}

We can now define  bounded generation, the notion appearing in the main theorem.

\begin{definition}\label{defboundedgeneration} \begin{enumerate}[(i)]
\item A charged algebra $A$ is said to be \emph{bounded generated in degrees $\leq N$} if there is an $(N+1)$-equivalence $f \colon B \m A$ with $B$ bounded cellular of degree $\leq N$. 
\item It is said to be \emph{bounded generated} if it is  bounded generated in degrees $\leq N$ for all $N \geq 0$. \end{enumerate}\end{definition}


\begin{remark} \label{cellsintop} The previous definitions and constructions can be modified for $\cdt$-algebras in $\cat{Top}$ and $\cat{sSet}$. For example, to give the translation from spaces to chain complexes, one simply thinks of $S^{N-1}$ as the chain complex with a single generator $e_{N-1}$ in degree $N-1$ and $D^N$ as the chain complex with two generators $e_N$ and $e_{N-1}$ in degrees $N$ and $N-1$ respectively, with $d(e_N) = e_{N-1}$. In the topological case the input thus consists of an $\cdt$-algebra $X$ and a map $S^{N-1} \to X$, using which one can define a partial $\cdt$-algebra $X \cup D^N$ and its completion denoted by $X \uplus D^N$. The analogues of Corollaries \ref{corattachexactseq} and \ref{corattachstab} are then true. In simplicial sets, a concern is that $X$ may not be fibrant as a simplicial set; we need to attach a subdivision of $\Delta^N$ along an iterated subdivision of $\partial \Delta^N$.\end{remark}

\section{The proof of the main theorem}\label{secproof} In this section, we prove the main theorem (Theorem \ref{thmmain}) and strengthen it to give a result for homological stability in explicit ranges. In this section we fix the tangential structure $\theta$ and set the monoid of charges $C$ to be $\bN_0^d$, and hence suppress them from our notation whenever convenient.

\begin{definition}For $\bk \in \bN_0^d$, we let $|\bk|$ denote $\max\{k_i\}$. We call this the \emph{maximal charge}.\end{definition}

\begin{theorem}\label{thmmainrange} Let $n \geq 2$, $C = \bN_0^d$, $A$ be a charged algebra and $\rho \colon \bR_{\geq 0} \m \R_{\geq 0}$ be a strictly increasing function with (i) $\rho(x+y) \leq \rho(x)+\rho(y)$ for all $x,y \in \bR_{\geq 0}$, (ii) $\rho(x) \leq x/2$ and (iii) $\rho^{-1}(\N_0) \subset \N_0$. The following are equivalent:
\begin{enumerate}[(i)]
\item For all $\bk \in C$ and each $i$, the basic stabilization maps $t_i \colon A(\bk) \m A(\bk+\be_i)$ are $\rho(k_i)$-equivalences. 
\item For any $\theta$-framed  connected non-compact $n$-dimensional manifold, all $\bk \in C$ and each $i$, the basic stabilization maps $t_i \colon \int^{\bk}_M A \m \int_M^{\bk+\be_i} A$ are $\rho(k_i)$-equivalences. 
\item There is a weak equivalence $B \m A$ where $B$ is bounded cellular such that if $B$ has an $\cdt$-cell in positive degree $N$ and with charge $\bk$, then $N \geq \rho(|\bk|)$. 

\end{enumerate}\end{theorem}

Here are some remarks about the consequences and assumptions of this theorem:
\begin{itemize}
\item Theorem \ref{thmmainrange} implies Theorem \ref{thmmain}. To see this, note that in Definition \ref{defboundedgeneration}, if all $\cdt$-cells of degree $\leq N$ are attached in finitely many charges, then there exists a function $\rho \colon \bN_0 \to \bR_{\geq 0}$ with $\lim_{j \to \infty} \rho(j) =\infty$ such that all $\cdt$-cells of degree $N$ are attached in charges $\bk$ with $\rho(|\bk|) \leq N$. Next remark that for any function $\rho \colon \bN_0 \to \bR_{\geq 0}$ with $\lim_{j \to \infty} \rho(j) =\infty$, there exists another function $\rho' \colon \bR_{\geq 0} \to \bR_{\geq 0}$ with $\lim_{j \to \infty} \rho'(j) =\infty$ which moreover satisfies the conditions in Theorem \ref{thmmainrange} and has the property that $\lfloor \rho(j) \rfloor \geq \lfloor \rho'(j) \rfloor $ for all $j \in \N_0$.
\item The homological stability range of $\int_M A$ in general is not as good as that of $A$. In the case of Example \ref{exampsym}, note that even though $\mr{Sym}(\bR^n)$ has a stability range with infinite slope, its topological chiral homology, $\mr{Sym}(M)$, only has a stability range with slope $1$.
\item Note that (i) only depends on the algebra structure on the homology of $A$. Thus if $A$ has $\cdt$-algebra and $\cE_{n'}^{\theta'}$-algebra structures inducing the same algebra structures on homology, this theorem can be used to transfer information about $A$ as an $\cdt$-algebra to $A$ as an $\cE_{n'}^{\theta'}$-algebra.
\item The assumption $n \geq 2$ is required for homological stability with labels in a bundle, the base case of an induction in the proof. No generality is lost, as the case $n=1$ is trivial per the remarks following Theorem \ref{thmmain}.
\item Closely related to  bounded generated $\cdt$-algebras are  degreewise finitely generated framed $\cdt$-algebras. These are cellular algebras with finitely many cells of a given degree. The same proof shows that having homological stability in conjunction with finite type homology, i.e. in each homological degree and charge the homology is finitely generated as an abelian group, is equivalent to degreewise finite generation.
\end{itemize}

We continue with remarks about $\rho$:
\begin{itemize}
\item The definition of $\rho$ as a function $\bR_{\geq 0} \to \bR_{\geq 0}$ instead of a function $\bN_0 \to \bN_0$ merely makes stating its required properties more convenient. 
\item One can give a version of this theorem with different $\rho_i$ for each stabilization map, which is allowed to depend on more than just the $i$th coordinate of the charge.
\item  The condition that $\rho(x) \leq x/2$ stems from the fact that this is the stability range for free $\cdt$-algebras. If one were to invert $2$ in the coefficients of homology, this can be relaxed to $\rho(x) \leq x$ when $n \geq 3$ \cite{kupersmillerimprov} \cite{federicomartin}. Rationally this is Theorem B of \cite{RW}.
\end{itemize}

Assuming Theorem \ref{thmmain}, we will give the proof of the local-to-global homological stability principle, i.e. Corollary \ref{cortopstab}. Recall that $X$ was an $\cdtf$-algebra in $\cat{Top}$ with homological stability and $\pi_0(X) = \bN_0$. We will instead prove it for $\theta$ any framing and $C = \bN_0^d$. 

\begin{proof}[Proof of Corollary \ref{cortopstab}] Since $X$ has homological stability, so does $C_*(X)$. Thus by Theorem \ref{thmmain}, $\int^\bk_M C_*(X)$ has homological stability with $M$ a $\theta$-framed connected non-compact manifold. It now suffices to note that we have that:
\[C_*\left(\int_M X\right) = C_*\left(\left\vert B_\bullet(\mathbf{M}^\theta,\bdt,X)\right\vert\right) \overset{\simeq}{\longleftarrow} \left\vert B_\bullet(\mathbf{M}^\theta,\bdt,C_*(X)) \right\vert = \int_M C_*(X)\]

\noindent The only content here is in the middle weak equivalence, which is Lemma \ref{lemchainssimplicialobjects}. \end{proof}

The proof of Theorem \ref{thmmainrange} involves an induction over the degrees in which we have homological stability. This requires us to make the following definitions. The proof will then follow from determining the interplay of the following four classes of charged algebras.

\begin{definition} Let $\rho \colon \bR_{\geq 0} \m \R_{\geq 0}$ be a strictly increasing function with (i) $\rho(x+y) \leq \rho(x)+\rho(y)$ for all $x,y \in \bR_{\geq 0}$, (ii) $\rho(x) \leq x/2$ and (iii) $\rho^{-1}(\N_0) \subset \N_0$.  We define the following classes of charged algebras:
\begin{enumerate}[(i)]
	\item $\cC_N^\rho$ is the class of bounded cellular algebras (see Definition \ref{defboundedcellular}) with $\cdt$-cells of degree $J \leq N$ and the $\cdt$-cells of degree $j$ attached in charge $\bk$ only if $j \geq \rho(|\bk|)$.
	\item $\cB^\rho_N$ is the class of charged algebras $A$ such that there is an $N$-equivalence $B \to A$ with $B \in \cC^\rho_{N}$.
	\item $\cS^\rho_N$ is the class of charged algebras $A$ that has homological stability range $\rho$ up to degree $N$. That is, the map $t_i \colon A(\bk) \to A(\bk+\be_i)$ is a $\min(N,\rho(k_i))$-equivalence. 
	\item $\cT^\rho_N$ is the class of charged algebras $A$ which on $\theta$-framed  connected non-compact $n$-dimensional manifolds $M$ have homological stability range $\rho$ up to degree $N$. That is, for each $i$ the map $t_i \colon \int^\bk_M A \to \int^{\bk+\be_i}_M A$ is a $\min(N,\rho(k_i))$-equivalence. 
\end{enumerate}

Also let $\cB_{\infty}^\rho$, $\cS_{\infty}^\rho$ or $\cT_{\infty}^\rho$ denote the intersection of respectively $\cB_{N}^\rho$, $\cS_{N}^\rho$ or $\cT_{N}^\rho$ over all $N \in \N_0$. 
\end{definition}

We deduce Theorem \ref{thmmainrange} from the propositions which follow later in this section.

\begin{proof}[Proof of Theorem \ref{thmmainrange}]
Combining Propositions \ref{propinit}, \ref{proptnimpliessn}, \ref{propcntinft} and \ref{propstabimpliesbound}, we see that $\cC_N^\rho=\cS_N^\rho=\cT_N^\rho$ for all $N \geq 0$. Thus $\cC_\infty^\rho=\cS_\infty^\rho=\cT_\infty^\rho$ which is equivalent to Theorem \ref{thmmainrange}.\end{proof}




\subsection{Closure properties}

We start with an easy proposition.

\begin{lemma}\label{lemclasscomp} Let $B \to A$ be an $N$-equivalence. If $B \in \cB^\rho_N$, $\cS^\rho_N$ or $\cT^\rho_N$, then respectively so is $A$.  \end{lemma}


\begin{proof}The statement about $\cB^\rho_N$ is true because the composition of maps which are $N$-equivalences is an $N$-equivalence. 

Let $B \m A$ be an $N$-equivalence and assume $B \in \cS^\rho_N$. Consider the diagram \[\xymatrix{B(\bk) \ar[r] \ar[d] & B(\bk+\be_i) \ar[d]  \\
A(\bk)  \ar[r]& A(\bk+\be_i)}\] Since the top horizontal map is a $\min(N,\rho(k_i))$-equivalence and the vertical maps are $N$-equivalences, the bottom horizontal map is a $\min(N,\rho(k_i))$-equivalence and so we may conclude that $A \in \cS^\rho_N$. 

Now let $B \m A$ be an $N$-equivalence but instead assume $B \in \cT^\rho_N$. Consider the diagram \[\xymatrix{\int^{\bk}_M B \ar[r] \ar[d] & \int_M^{\bk+\be_i} B \ar[d]  \\
\int_M^{\bk} A  \ar[r]& \int_M^{\bk+\be_i} A}\] The top horizontal map is a $\min(N,\rho(k_i))$-equivalence by assumption. Lemma \ref{lemtchcomp} implies that the vertical maps are $N$-equivalences. Thus the bottom horizontal map is a $\min(N,\rho(k_i))$-equivalence and so we may conclude that $A \in \cT^\rho_N$. 



\end{proof}


\subsection{The case $N = 0$} Since our proof of Theorem \ref{thmmainrange} is an induction over $N$, we start with $N=0$.

\begin{proposition}\label{propinit} Any charged algebra is in $\cB^\rho_0 = \cT^\rho_0 = \cS^\rho_0$.\end{proposition}

\begin{proof} By definition, charged algebras are augmented and connected. Thus we have a chosen isomorphism $H_0(A) \cong \bZ[\bN_0^d]$, which is in $\cS^\rho_0$. We claim it is in $\cB^\rho_0$ as well. To see this, take any representative $a_i \in A(\be_i)$ of the chosen generator of $H_0(A(\be_i)) \cong \bZ$. Let $B$ be the free $\cdt$-algebra on $d$ generators in degree $0$. This is in $\cC^\rho_0$ by definition. The choice of $a_i \in A(\be_i)$ gives a map $B \m A$ which induces an isomorphism on $H_0$ (and hence is a $0$-equivalence). Thus, $A \in \cB^\rho_0$.

By the proof of Proposition \ref{propcntinft}, $H_0(\int^{\bk}_M B) \cong \bZ$ for $M$ connected. This shows that $B \in \cT^\rho_0$. Since $B \m A$ is a 0-equivalence, Lemma \ref{lemclasscomp} shows that $A$ is in $ \cT^\rho_0$ as well. By Proposition \ref{proptnimpliessn}, we have $\cT^\rho_0 \subset \cS^\rho_0$.\end{proof}

\subsection{Stability for topological chiral homology implies stability}
Another easy proposition says that homological stability of topological chiral homology on any open manifold implies it for $\bR^n$.

\begin{proposition}\label{proptnimpliessn}For any $N \geq 0$ we have that $\cT^\rho_N \subset \cS^\rho_N$.\end{proposition}

\begin{proof} We have $\int_{\R^n} A \simeq A$ by Lemma \ref{lemtchrn}, and $\R^n$ is a $\theta$-framed connected non-compact manifold.\end{proof}

\subsection{Level wise bounded generation implies stability for topological chiral homology} Next we prove that bounded generation implies homological stability for topological chiral homology of $A$ on any $\theta$-framed  connected non-compact $n$-dimensional manifold $M$. See Definition \ref{colconfig} for the definition of configuration spaces labeled in a fibration.

\begin{proposition}\label{propcntinft} For any $N$ we have that $\cC^\rho_N \subset \cT^\rho_\infty$ and thus $\cC^\rho_N \subset \cS^\rho_\infty$.\end{proposition}

\begin{proof}We prove this by induction on $N$. For $N=0$, we can assume that the  $\cdt$-algebra $A$ is of the form \[\bigoplus_{\bk = (k_1,\ldots,k_d)} C_*\left(\cdt(k_1+\ldots+k_d)/\fS_{\bk}\right)\] We need to establish homological stability for $\int^{\bk}_M A$. 
	
Recall from Subsection \ref{subsechomologyaftercellattachment} that $C_{\bk}^\theta(M)$ is quotient by $\fS_\bk$ of the space of $\theta$-framed embeddings $\sqcup_k \bR^n \to M$ whose centers are disjoint, which is equals the configuration space of $\bk$-colored particles with labels in the fibration $\theta(TM) \to M$.
	
By an extra degeneracy argument $\int_M A$ is weakly equivalent to chains on the space 
\[\bigsqcup_{\bk = (k_1,\ldots,k_d)} \mr{Emb}^\theta(\sqcup_{k_1+\ldots+k_d} \bR^n,M)/\fS_\bk\]
The inclusion of this into $C_{\bk}^\theta(M)$ is a weak equivalence, by a similar shrinking argument as in Lemma \ref{lemdisjcomp}. It thus suffices to establish homological stability for the spaces $\bigoplus_{\bk} C_*(C_{\bk}^\theta(M))$. This is the case with function $\rho(x) = x/2$ by a generalization of Proposition B.4 of \cite{federicomartin}: this generalization says that if $E \to M$ is a fiber bundle over $M$ with path-connected fibers, then the configuration space of unordered of points in $M$ with labels in $E$ exhibits homological stability. The necessary generalizations involve replacing a single color by multiple colors and replacing fiber bundle with fibration. They prove a stability range of $\frac{k}{2}-1$ but this can be improved to a range of $\frac{k}{2}$ using the arguments appearing in the last paragraph of page 9 of \cite{kupersmillerimprov}.\footnote{See also the appendix of the arXiv-version of this paper.} We conclude that $A \in \cT^\rho_\infty$ and so $\cC^\rho_0 \subset \cT^\rho_\infty$. 

Assume the statement is true for $N-1$ and choose $A \in \cC^\rho_{N}$. Then there are cellular $\cdt$-algebras $A_0$, $A_1$, $\ldots$ such that $A_0 \in \cC^\rho_{N-1}$, $A_{i+1}$ is constructed from $A_i$ by attaching cells in degree $N$ and charge $i$, and $A_i=A$ for $i > \rho(N)$. By our induction hypothesis,  $A_0 \in \cT^\rho_\infty$. Corollary \ref{corattachstab} and Remark \ref{remmultiplecells2} imply that if $A_i \in \cT^\rho_\infty$ and $i \leq \rho(N)$, then $A_{i+1} \in \cT^\rho_\infty$. Thus, $A \in \cT^\rho_\infty$.




\end{proof}

\begin{proposition}\label{propboundimpliesstab} For any $N$ we have that $\cB^\rho_N \subset \cT^\rho_N$ and thus $\cB^\rho_N \subset \cS^\rho_N$.\end{proposition}

\begin{proof}By definition, $A \in \cB^\rho_N$ implies the existence of a $N$-equivalence $B \to A$ such that $B \in \cC^\rho_N$. But we just saw that $\cC^\rho_N \subset \cT^\rho_\infty$ and hence $B \in \cT^\rho_N$. The conclusion $A \in \cT^\rho_N$ now follows from Lemma \ref{lemclasscomp}.\end{proof}

\subsection{Stability implies  bounded generation} Finally we prove the hardest step in the induction, that homological stability in the range $* \leq N+1$ implies  bounded generation in the range $* \leq N+1$ given that result for the range $* \leq N$. The strategy is that of relative CW approximation of maps.

\begin{proposition}\label{propstabimpliesbound} If $\cS^\rho_N = \cB^\rho_N$, then $\cS^\rho_{N+1} \subset \cB^\rho_{N+1}$.\end{proposition}

\begin{proof}Let $A \in \cS^\rho_{N+1}$, then our hypothesis implies $A \in \cB^\rho_N$ and there exists an $N$-equivalence $f\colon B \to A$ with $B \in \cC^\rho_{N}$. Let $R_{N+1}= \rho^{-1}(N+1)$, which is a single non-negative integer since $\rho\colon \bR_{\geq 0} \to \bR_{\geq 0}$ is strictly increasing and satisfies $\rho^{-1}(\bN_0) \subset \bN_0$. We will extend the map $f\colon B \to A$ by attaching to $B$ a collection of $\cdt$-cells of degree $N+1$ and charge $\bk$ satisfying $|\bk| \leq R_{N+1}$. The end result will be an $\cdt$-algebra $B'$ and an $(N+1)$-equivalence $f'\colon B' \to A$ of $\cdt$-algebras extending $f\colon B \to A$. 

We will build $f'\colon B' \to A$ by induction on maximal charge $c = |\bk|$. Let $B_{-1}^\mr{surj}\coloneqq B$ and $f_{-1}^\mr{surj} = f$. For $0 \leq c \leq R_{N+1}$, we will define  $\cdt$-algebras with maps of $\cdt$-algebras extending $f\colon B \to A$:
\[f^\mr{surj}_c\colon B_c^\mr{surj} \to A\qquad \text{and} \qquad  f^\mr{iso}_c\colon B_c^\mr{iso} \to A\]
with the following properties
\begin{itemize}
\item Both $B_c^\mr{surj}$ and $B_c^\mr{iso}$ will be in $\cC^\rho_{N+1}$.
\item  The algebra and map $f_c^\mr{surj}\colon B_c^\mr{surj} \to A$ will extend $f_{c}^\mr{iso}\colon B_{c}^\mr{iso} \to A$ and have the property that $B_c^\mr{surj}(\bk) \m A(\bk)$ is an $(N+1)$-equivalence for $|\bk| \leq c$ and an $N$-equivalence for $|\bk| > c$. 
\item The algebra and map $f_c^\mr{iso}\colon B_c^\mr{iso} \to A$ will extend $f_{c-1}^\mr{surj}\colon B_{c-1}^\mr{surj} \to A$ and have the property that $B_c^\mr{iso}(\bk) \m A(\bk)$ is an $N$-equivalence for all $\bk$ and $H_N(B_c^\mr{iso}(\bk)) \m H_N(A(\bk))$ is an isomorphism for $|\bk| \leq c$.
\end{itemize} 

The idea is that the algebra and map $f_c^\mr{surj}\colon B_c^\mr{surj} \to A$ correct the problem that $f_{c}^\mr{iso}\colon B_{c}^\mr{iso} \m A$ is not a surjection on $H_{N+1}$ and maximal charge $|\bk| = c$, while the algebra and map $f_c^\mr{iso}\colon B_c^\mr{iso} \to A$ correct the problem that the map $f_{c-1}^\mr{surj}\colon B^\mr{surj}_{c-1} \m A$ is not an isomorphism on $H_N$ and maximal charge $|\bk| = c$. Eventually $f'\colon B' \to A$ will be $f^\mr{surj}_{R_{N+1}}\colon B_{R_{N+1}}^\mr{surj} \to A$.

It suffices to assume we have constructed $f^\mr{surj}_{c-1}\colon B_{c-1}^\mr{surj} \to A$ for $c \leq R_{N+1}$, and construct first $f_c^\mr{iso}\colon B_c^\mr{iso} \to A$ and then $f_c^\mr{surj}\colon B_c^\mr{surj} \to A$.

\begin{description}
\item[Step 1 -- isomorphism]

Assume we have defined $f_{c-1}^{surj}\colon B^\mr{surj}_{c-1} \m A$  with the desired properties listed above. Consider the kernels of $H_{N}(B^\mr{surj}_{c-1}(\bk)) \to H_{N}(A(\bk))$ for each $\bk$ satisfying $|\bk| = c$. Take cycles $\{b_{p}\}_{p \in P_\bk}$ with $b_{p} \in B^\mr{surj}_{c-1}(\bk)$ representing a generating set of these kernels. Because the homology class of $b_{p}$ is in the kernel of the map to $H_*(A(\bk))$, $f^\mr{surj}_{c-1}(b_p)$ is a boundary $d(x_p)$.

For each $p \in P_k$ we attach an $\cdt$-cell $(e^{N+1})_{p}$ to $B^\mr{surj}_{c-1}$ in charge $\bk$ along $b_p$. We call the resulting algebra $B_c^\mr{iso}$ and using Lemma \ref{lemattachmap} extend the map $B^\mr{surj}_{c-1} \m A$ to a map 
\[f^\mr{iso}_c\colon B^\mr{iso}_c \coloneq B^\mr{surj}_{c-1} \uplus \biguplus_{p \in P} (e^{N+1})_{p} \to A\]
by sending $(e^{N+1})_{p}$ to $x_p$. 

Note that $B^\mr{iso}_c \in \cC^\rho_{N+1}$ and that $B^\mr{iso}_c \to A$ induces an isomorphism on $H_*$ for $* \leq N-1$ by Property (ii) of Corollary \ref{corattachexactseq} and is an $(N+1)$-equivalence for charges $\bk$ with $|\bk| \leq c-1$ by Property (i) of Corollary \ref{corattachexactseq}. 

We also claim it induces an isomorphism on $H_{N}$ for charges $\bk$ with $|\bk| = c$. This will be a consequence of Property (iii) of Corollary \ref{corattachexactseq}, which says there is an exact sequence 
\begin{align*}H_{N+1}(B^\mr{surj}_{c-1}(\bk)) &\to H_{N+1}(B^\mr{iso}_c(\bk)) \to \bigoplus_{p \in P_{\bk'}} \bZ \to H_{N}(B^\mr{surj}_{c-1}(\bk)) \\
&\to H_{N}(B^\mr{iso}_c(\bk)) \to 0\end{align*}
where $\bk'$ ranges over all $\bk'$ such that $\bk' \leq \bk$ and the generator of the $\bZ$-summand corresponding to $p \in P_\bk$ goes to $t_{\bk-\bk'}(b_p)$. We conclude that $H_{N}(B^\mr{iso}_c(\bk)) = H_{N}(B^\mr{surj}_{c-1}(c))/\mr{im}(\bigoplus_{p \in P_{\bk'}}\bZ)$ and note that this image contains the kernel $\ker(H_{N}(B^\mr{surj}_{c-1}(c)) \to H_{N}(A(c)))$. Furthermore, by construction the surjective map $H_{N}(B^\mr{surj}_{c-1}(\bk)) \to H_{N}(A(\bk))$ factors over $H_{N}(B^\mr{iso}_c(\bk))$. Hence the map $H_{N}(B^\mr{iso}_c(\bk)) \to H_{N}(A(\bk))$ is a surjection with trivial kernel and hence an isomorphism.


\item[Step 2 -- surjection] Assume we have defined $f^\mr{iso}_c \colon B^\mr{iso}_{c} \to A$ with the desired properties. Take cycles $\{a_q\}_{q \in Q_\bk}$ with $a_q \in A(\bk)$ representing a generating set of $H_{N+1}(A(\bk))$ for each $\bk$ satisfying $|\bk| = c$. 

For each $q \in Q_k$ we attach an $\cdt$-cell $(e^{N+1})_q$  to $B^\mr{iso}_{c}$ in charge $\bk$ with trivial attaching map (that is, setting $d((e^{N+1})_q) = 0$). We call the resulting algebra $B^\mr{surj}_{c}$ and using Lemma \ref{lemattachmap} extend the map $B^\mr{iso}_{c} \to A$ to a map 
\[f^\mr{surj}_c\colon B^\mr{surj}_{c} \coloneqq B^\mr{iso}_{c} \uplus \biguplus_{q \in Q} (e^{N+1})_q \to A\]
by sending $(e^{N+1})_q$ to $a_q$.  This is possible since $d(a_q) = 0$ because $a_q$ is a cycle. 

Note that $B^\mr{surj}_{c} \in \cC^\rho_{N+1}$ and that $B^\mr{surj}_{c} \to A$ induces an isomorphism on $H_*$ for $* \leq N$ by properties (ii) and (iii) of Corollary \ref{corattachexactseq} and a surjection on $H_{N+1}$ for charges $\bk$ with $|\bk| \leq c-1$ by Property (i) of Corollary \ref{corattachexactseq}. 

We also claim it induces a surjection on $H_{N+1}$ for charges $\bk$ with $|\bk| = c$. This is a consequence of Property (iv) of Corollary \ref{corattachexactseq}, since in this case we get an exact sequence
\[H_{N+1}(B^\mr{iso}_{c}(\bk)) \to H_{N+1}(B^\mr{surj}_{c}(\bk)) \to \bigoplus_{q \in Q_{\bk'}} \bZ \to 0\]
where $\bk'$ ranges over all $\bk'$ such that $\bk' \leq \bk$. The first map is canonically split by sending the $(e^{N+1})_q$ to $0$, which implies we have an isomorphism $H_{N+1}(B^\mr{surj}_{c}(\bk)) \cong H_{N+1}(B^\mr{iso}_{c}(\bk)) \oplus \bigoplus_{q \in Q_{\bk'}} \bZ$. By construction we have that the map $H_{N+1}(B^\mr{surj}_{c}(\bk)) \to H_{N+1}(A(\bk))$ extends the map $H_{N+1}(B^\mr{iso}_{c}(\bk)) \to H_{N+1}(A(\bk))$ by sending the generator of the $\bZ$-summand corresponding to $q \in Q_\bk$ to $a_q \in H_{N+1}(A(\bk))$. This makes the map surjective since we hit all generators of the cokernel.
\end{description}

Let $f'\colon B' \to A$ be $f^\mr{surj}_{R_{N+1}}\colon B_{R_{N+1}}^\mr{surj} \to A$, thus obtaining an $\cdt$-algebra $B' \in \cC^\rho_{N+1}$ with a map $B' \m A$ such that the map is an $(N+1)$-equivalence for charges $\br$ satisfying $|\br| \leq R_{N+1}$. We will prove it is an $(N+1)$-equivalence for all charges. 

By definition, $A$ has homological stability range $\rho$ and by Proposition \ref{propcntinft}, so does $B'$. Let $\mathbf{r}$ satisfy $|\br| = R_{N+1}$, $\mathbf{j} \in \bN_0^d$ and consider the following commutative diagram:
\[\xymatrix{B'(\br) \ar[r]^{t_\bj} \ar[d] & B'(\br+\bj) \ar[d] \\
A(\br) \ar[r]^{t_\bj}& A(\br+\bj)}\] The two stabilization maps (the horizontal maps) are $(N+1)$-equivalences by homological stability. The leftmost vertical map is an $(N+1)$-equivalence by the argument above. Therefore, the rightmost vertical map is an  $(N+1)$-equivalence as well. This shows that $B' \m A$ is an $(N+1)$-equivalence and so $A \in \cB^\rho_{N+1}$. 
\end{proof}

\begin{remark}\label{remcelldecompexist} The technique used in the proof of Proposition \ref{propstabimpliesbound} can be used to show that any connected augmented charged algebra in chain complexes can be obtained by iterated $\cdt$-cell attachments up to weak equivalence.\end{remark}

\section{Extending the main theorem to compact manifolds}\label{secscanning} In this section we prove that a restatement of the main theorem extends to connected compact manifolds. In this section we restrict to $\cdtt$-algebras in $\cat{Top}$ or $\cat{sSet}$, although several results should have versions over $\cat{Ch}$ and other tangential structures. This restriction allows us to quote the results of \cite{Mi2}.

For us, the $n$-fold delooping $B^n X$ will be defined to be $B(\Sigma^n_+,\bdtt,X)$, with $\Sigma^n_+$ as defined below. There is a natural zig-zag of weak equivalences between this model of the $n$-fold delooping and that of May \cite{M}.

\begin{definition}\label{defsuspension} Let $\RS$ be the right functor over $\bdt$ in $\cat{Top}$ given by
	\[X \mapsto \Sigma^n X_+ \] where we consider $S^n$ as the one-point compactification of $\bR^n$. The right module structure $a\colon \Sigma^n_+ \bdt X \m \Sigma^n_+ X$ is defined as follows. A point of $\Sigma^n_+ \bdt X$ is either the base point or determined by a triple $(r,(e,\psi),x)$ of $r \in \R^n$, $(e,\psi) \in \mr{Emb}^\theta(\sqcup_k \bR^n,\bR^n)$ and $x \in X$. Here $e$ is an embedding and $\phi$ is a path of bundle maps. If $r$ is not in the image of $e$, we define $a(r,(e,\psi),x) \in \Sigma^n_+ X$ to be the base point. Otherwise, suppose that $r$ is contained in the image under $e$ of the $i$th component of $\sqcup_k \bR^n$, let $e_i\colon\R^n \m \R^n$ be the restriction of $e$ to the component whose image contains $r$ and set
	\[a(r,(e,\psi),x) \coloneqq (e_i^{-1}(r),x) \in \R^n \times X.\]
\end{definition}






Since $\Sigma^n_+$ is a right $\bdtt$-functor, so is $F \circ \Sigma^n_+$ for any functor $F$. Let $\mr{Map}^c(M,\Sigma^n_+ -)$ denote the right $\bdtt$-functor whose value on a space $Y$ is $\mr{Map}^c(M,\Sigma^n_+ Y)$, the space of compactly supported maps. See \cite{Mi2} for a description of a natural transformation $\mathbf{M}^\mr{pt} \m \mr{Map}^c(M,\Sigma^n_+ -)$. We now recall the definition of the scanning map in this context.

\begin{definition}
For a framed manifold $M$ and an $\cdtt$-algebra $X$, the \emph{scanning map} 
\[s \colon \int_M X \m \mr{Map}^c(M,B^n A)\]
is the composition of the map $B(\mathbf{M}^\mr{pt},\bdtt,X) \m B(\mr{Map}^c(M,\Sigma^n_+ -),\bdtt,X)$ induced by the map of right modules, with the natural map $B(\mr{Map}^c(M,\Sigma^n_+ -),\bdtt,X) \m \mr{Map}^c(M,B(\Sigma^n_+-,\bdtt,X))$.
\end{definition}

When $X$ is grouplike, the scanning map is a weak equivalence. This is known as \emph{non-abelian Poincar\'e duality} \cite{lurieha} \cite{Sa} \cite{ayalafranciskoszul}. In particular, there is a weak-equivalence $\int_M \Omega^n B^n X \m \mr{Map}^c(M,B^n X)$ for any $\cdtt$-algebra $X$. For connected $M$, we have that $\pi_0(\mr{Map}^c(M,B^n X))$ is the Grothendieck group of $\pi_0(X)$. Given $\bk$ in the Grothendieck group of $\pi_0(X)$, we let $\mr{Map}^c_{\bk}(M,B^n X)$ denote the corresponding connected component.

For $M=\R^n$, the scanning map gives a map $g \colon X \m \Omega^n B^n X$ which is the group-completion map. We now assume $C=\N^d_0$ for simplicity. When $X$ has homological stability, the map $g \colon X \m \Omega^n B^n X$ is a homology equivalence in the same range by the group-completion theorem \cite{MSe}. The analogous result for non-compact path-connected framed manifolds was proven in \cite{Mi2}.

\begin{theorem}[Miller] \label{propstablehomology} Let $M$ be a non-compact path-connected framed manifold of dimension $n \geq 2$. Let $X$ be a $\cdtt$-algebra $X$ with homological stability range $\rho$. The scanning map $s\colon \int^\bk_M X \m \mr{Map}_{\mathbf{k}}^c(M,B^n X)$ is a $\rho(\min\{k_i\})$-equivalence. \end{theorem}

We will extend this result to compact manifolds. In the compact case, the spaces $\mr{Map}_{\mathbf{k}}^c(M,B^n X)$ need not have the same homotopy type as $\bk$ varies and thus the spaces $\int^\bk_M X$ need not exhibit homological stability even if $X$ has homological stability. Naturality of topological chiral homology gives us a map $\smallint g$, which fits into the following diagram:
\[\xymatrix{\int^\bk_M X \ar[rr]^-s \ar[rd]_-{\smallint g} & & \mr{Map}_{\mathbf{k}}^c(M,B^n X)\\
	& \int^\bk_M \Omega^n B^n X \ar[ru]_-\simeq & }\]
For technical reasons, it is easier to study the map $\smallint g$ than the scanning map $s$. 

\begin{theorem}\label{thmcompact}Let $n \geq 2$. For any framed path-connected manifold $M$ and an $\cdtt$-algebra $X$ that has homological stability with range $\rho$, the map $s\colon \int^\bk_M X \to \mr{Map}_{\mathbf{k}}^c(M,B^n X)$ is a $\rho(\min\{k_i\})$-equivalence.\end{theorem} 

\begin{proof}The proof uses a resolution by punctures argument introduced in \cite{RW}.

First of all we replace $s$ by $\smallint g$. Let $\mr{Emb}^\mr{pt}_\infty(\sqcup_k \bR^n,M)$ be the subspace of $\mr{Emb}^\mr{pt}(\sqcup_k \bR^n,M)$ consisting of embeddings and paths such that the closure of the image is an embedded closed disk and its complement has infinite cardinality. We use it to modify  Definition \ref{defmtheta} and Definition \ref{deftch} by using the right $\bdtt$-functor $\mathbf M^\mr{pt}_\infty$ given by $X \mapsto \bigsqcup_{k \geq 0} \mr{Emb}^\mr{pt}_{\infty}(\bR^n,M) \odot_{\fS_k} X^k$ and defining
“\[\int_{M,\infty} X = B (\mathbf{M}^\mr{pt}_\infty,\bdtt,X) \] The inclusion $\mr{Emb}^\mr{pt}_\infty(\sqcup_k \bR^n,M) \hookrightarrow \mr{Emb}^\mr{pt}(\sqcup_k \bR^n,M)$ is a weak equivalence by a similar argument to \ref{lemdisjcomp}, so the inclusion $\int_{M,\infty} X \to \int_M X$ is a weak equivalence as well by Lemmas \ref{lemweakequivfunctor} and \ref{lemgeomrel}.

We define a semisimplicial space with $q$-simplices given by pairs $(\phi,(m_0,\ldots,m_q))$ of an ordered $(q+1)$-tuple  of points $m_i$ in $M$ and a framed embedding $\phi\colon \sqcup_k \bR^n \to M \setminus \{m_0,\ldots,m_q\}$. We put the usual topology on the embeddings and paths of bundle maps, but put the \emph{discrete} topology on the points in $M$. The face map $d_i$ forgets the point $m_i$. This semisimplicial space has an augmentation to $\mr{Emb}^\mr{pt}_\infty(\sqcup_k \bR^n,M)$ by forgetting all the $m_i$. It defines semisimplicial right $\bdtt$-functor $\mathbf{M}^\mr{pt}_{\bullet,\infty}$ and thus an augmented semisimplicial simplicial space $I_{\bullet,\bullet}(M,X)$ given by 
\[[p,q] \mapsto B_p(\mathbf{M}^\mr{pt},\bdtt,X) = \mathbf{M}^\mr{pt}_{q,\infty}(\bdtt)^p X\]
This satisfies
\[\left\vert [p] \mapsto I_{p,q}(M,X)\right\vert \simeq \begin{cases} \int_{M,\infty} X & \text{if $q = -1$} \\
\bigsqcup_{(m_0,\ldots,m_q)} \int_{M \setminus \{m_0,\ldots,m_q\},\infty} X & \text{if $q\geq 0$} \end{cases}\]

Realizing in the other direction, we claim that the augmentation 
\[\epsilon_p\colon \left\vert [q] \mapsto \mathbf{M}^\mr{pt}_{q,\infty}(\bdtt)^p X \right\vert \to \mathbf{M}^\mr{pt}_\infty (\bdtt)^p X\]
is a weak equivalenc. To prove this, we start by nothing that because the complement of the closure of the images of the outermost embeddings in $\int_{M,\infty} X$ has infinitely many points (this was the point of using $\mr{Emb}_\infty$), that the point inverses of $\epsilon_p$ are contractible by the argument on Page 17 of \cite{RW} or Lemma 5.7 of \cite{kupersmillertran}. Since every point has a neighborhood not intersecting the image of a disk (this was the point of using closures in our definition of $\mr{Emb}^\mr{pt}_\infty$), a standard argument like Proposition 5.8 of \cite{kupersmillertran} proves that $\epsilon_p$ is a microfibration as in Section 2 of \cite{weissclassify}. By Lemma 2.2 of \cite{weissclassify}, the map $\epsilon_p$ is a Serre fibration and hence a weak equivalence. Using Lemma \ref{lemgeomrel} conclude that $|I_{\bullet,\bullet}(M,X)| \to \int_{M,\infty} X$ is a weak equivalence. 

Our construction gives a map 
\[(\smallint g)_{\bullet,\bullet}\colon I_{\bullet,\bullet}(M,X) \to I_{\bullet,\bullet}(M,\Omega^n B^n X)\]
After realizing in the $p$-direction, this is given by $\smallint g$ on the augmentation and each of the simplicial levels. Restricting to charge $\bk$ and using that the augmentation maps are weak equivalences, we get a relative geometric realization spectral sequence as in Lemma \ref{lemgeomrealss} for the $p$-direction:
\begin{align*}E^1_{pq} = \bigoplus_{\{m_0,\ldots,m_p\}}H_q&\left(\int^\bk_{M \setminus \{m_0,\ldots,m_p\}} \Omega^n B^n X,\int^\bk_{M \setminus \{m_0,\ldots,m_p\}} X\right) \\
&\Rightarrow H_{p+q}\left(\int^\bk_M \Omega^n B^n X,\int^\bk_M X\right)\end{align*}

\noindent The result of \cite{Mi2} says that for each of the summands, the map $(\smallint g)_p$ is a $\rho(\min\{k_i\})$-equivalence on the $\bk$-component (this uses $n \geq 2$ to guarantee that $M$ remains connected after removing points). This means that the $E^1$-page vanishes for $q \leq \rho(\min\{k_i\})$ and hence so does the $E^\infty$-page. This proves the result.
\end{proof}

Thus the local-to-global principle can be rephrased as saying that the group-completion map induces a homology equivalence in a range tending to infinity if and only if the scanning map for topological chiral homology does.

\section{Applications of the main theorem} \label{secapplications}

In this section we discuss examples of applications of Theorem \ref{thmmain}. In the first subsection, we discuss applications of the implication (i) $\Rightarrow$ (ii) and in the second subsection we discuss applications of the implication (iii) $\Rightarrow$ (ii). 

We will consider many examples of labeled configuration spaces in this section. All of these examples are homeomorphic to the configuration spaces $C(M;A)$ considered in \cite{Sa} for some choice of partial algebra $A$. The techniques of Section 3.3 of \cite{Mi3} show that $\int_M C(\R^n;A) \simeq C(M;A)$ and so we can use topological chiral homology to study these spaces.

\subsection{Applications of the local-to-global homological stability principle} \label{subseclocalglobal}

In this subsection, we apply the local-to-global homological stability principle to give new proofs of homological stability for bounded symmetric powers and the divisor spaces which appear in Segal's work on spaces of rational functions \cite{Se}. Recall that we already discussed symmetric powers in Example \ref{exampsym}.

\subsubsection{Bounded symmetric powers and divisor spaces associated to projective spaces}\label{boundedsymmetricpowerssubsectionlocaltoglobal} 
We start by defining bounded symmetric powers. 

\begin{definition}
Let $\mr{Sym}_k^{\leq d}(M)$ denote the subspace of $\mr{Sym}_k(M)$ where at most $d$ particles occupy the same point in $M$ and define
\[\mr{Sym}^{\leq d}(M)  \coloneqq \bigsqcup_k \mr{Sym}_k^{\leq d}(M).\] 
\end{definition}

The spaces $\mr{Sym}^{\leq d}(\R^2)$ can be identified with the space of monic polynomials with roots of multiplicity of order at most $d$. They have applications to the study of $J$-holomorphic curves \cite{Mi3} and were also studied in \cite{VW} in relation to a motivic analogue of homological stability. Since there is a canonical map of operads $\cdt \to \cE^\mr{O}_n$, the functoriality of $\mr{Sym}^{\leq d}$ with respect to embeddings makes $\mr{Sym}^{\leq d}(\bR^n)$ into an $\cdt$-algebra for all $\theta$, but we will only use the $\cdtf$-algebra structure. Note that $\pi_0(\mr{Sym}^{\leq d}(\bR^n)) \cong \N_0$ and let $t$ denote the stabilization map that increases charge by one.

Before describing corollaries of the local-to-global homological stability principle for bounded symmetric powers, we recall the definition of certain divisor spaces.  These divisor spaces appear in the study of holomorphic maps from a curve to complex projective space $\bC P^d$. Using the local-to-global homological stability principle and the results of \cite{CCMM1}, we will give a new proof of homological stability for divisor spaces. One can consider $(\mr{Sym}(M))^{d+1}$  as a configuration space of particles labeled by particles with $d+1$ different colors.

\begin{definition}Let $\mr{Div}^{ d}(M)$ denote the subspace of $(\mr{Sym}(M))^{d+1}$ where any given point in $M$ is occupied by particles of at most $d$ different colors.
\end{definition}

The functoriality of $\mr{Div}^d$ with respect to embeddings makes $\mr{Div}^d(\R^n)$ into an $\cdt$-algebra. Note that $\pi_0(\mr{Div}^d(\R^n))$ is $\N_0^{d+1}$. For $\bk = (k_1,\ldots,k_{d+1})$, we define $\mr{Div}^{d}_{\bk}(M)$ to be the subspace of $\mr{Div}^{d}(M)$ where there are exactly $k_i$ particles of color $i$. 

The space $\mr{Div}^d(\R^2)$ can be thought of as the space of $d+1$ monic polynomials with no common root. Let $\delta(k)$ denote $(k,k,\ldots,k) \in \bN_0^{d+1}$. In \cite{Se}, Segal noted that $\mr{Div}_{\delta(k)}^{d}(\bR^2)$ is homeomorphic to $\mr{Hol}^*_k(\C P^1, \C P^d)$, the space of based holomorphic degree $k$ maps from $\C P^1$ to $\C P^d$. More generally, there is a map \[\pi\colon\mr{Div}_{\delta(k)}^{d}(\Sigma_g) \m (J_g)^d\]
with $\Sigma_g$ a curve of genus $g$ and $J_g$ its Jacobian. The fiber of $\pi$ at $0$ is homeomorphic to $\mr{Hol}^*_k(\Sigma_g, \C P^d)$ and the map $\pi$ is a homology fibration in a range \cite{Se}. There are $d$ different basic stabilization maps, with \[t_i\colon\mr{Div}_{\bk}^{d}(\R^2) \m \mr{Div}_{\bk+\be_i}^{d}(\R^2)\]
the basic stabilization map that adds a point of the $i$th color. Note that on Page 46 of \cite{Se}, it is shown that if $k_i \geq k_j$ for all $j$ then $t_i$ is a homotopy equivalence. Thus, the relevant map to consider is $t=t_1 \circ t_2 \circ \ldots \circ t_d\colon \mr{Div}_{\delta(k)}^{d}(\R^2) \m \mr{Div}_{\delta(k+1)}^{d}(\R^2)$; that is, we can restrict to the diagonal. 

These divisor spaces and bounded symmetric powers are closely related when $M=\R^n$.  In \cite{V1}, Vassiliev proved the following.

\begin{theorem}[Vassiliev]
There is a homology equivalence compatible with the stabilization map between $\mr{Div}_{\delta(k)}^{d}(\R^2)$ and $\mr{Sym}^{\leq d}_{k(d+1)}(\R^2)$. Additionally, the stabilization map 
\[t\colon\mr{Sym}^{\leq d}_{k}(\R^2) \m  \mr{Sym}^{\leq d}_{k+1}(\R^2)\]
is a homotopy equivalence unless $k+1$ is divisible by $d+1$. 
\end{theorem}

These results say that homological stability for bounded symmetric powers of $\R^2$ is equivalent to stability for divisors in $\R^2$. In \cite{CCMM1}, Cohen, Cohen, Mann and Milgram computed the cohomology of the spaces $\mr{Div}_{\delta(k)}^{d}(\R^2)$ with field coefficients and also described the effect of the stabilization map. See also \cite{KSstring} for this calculation. From these explicit calculations, one sees that $t\colon\mr{Div}_{\delta(k)}^{d}(\R^2) \m \mr{Div}_{\delta(k+1)}^{d}(\R^2)$ induces a homology isomorphism in the range $* \leq k(2d-1)$. This implies that $t\colon\mr{Sym}^{\leq d}_{k}(\R^2) \m \mr{Sym}^{\leq d}_{k+1}(\R^2)$ induces a homology isomorphism in the range $* \leq \frac{k(2d-1)}{d+1}$. Applying the local-to-global homological stability principle we get the following corollaries.

\begin{corollary}\label{corboundedsymm}
For $\Sigma$ an orientable connected non-compact surface, \[t\colon\mr{Sym}^{\leq d}_{k}(\Sigma) \m  \mr{Sym}^{\leq d}_{k+1}(\Sigma)\]
induces an isomorphism on $H_*$  for $* \leq k/2$.\end{corollary}

This greatly improves on the range of $* \leq \frac{k}{d+1}-d+3$ established in \cite{Y} and agrees with the range recently proved in \cite{kupersmillertran}.

\begin{corollary}
For $\Sigma$ an orientable connected non-compact surface, \[t\colon\mr{Div}^{d}_{\delta(k)}(\Sigma) \m  \mr{Div}^{d}_{\delta(k+1)}(\Sigma)\]
induces an isomorphism on $H_*$  for $* \leq k/2$.

\end{corollary}

Homological stability for the spaces $\mr{Div}^{d}_{\delta(k)}(\Sigma)$ was originally proven in \cite{Se}. Note that although Segal's range has a higher slope, our range has a higher constant term. Although $\mr{Div}_{\delta(k)}^{d}(\R^2)$ and $\mr{Sym}^{\leq d}_{k(d+1)}(\R^2)$ are homology equivalent, $\mr{Div}_{\delta(k)}^{d}(M)$ and $\mr{Sym}^{\leq d}_{k(d+1)}(M)$ are not homology equivalent for $M$ a surface not homeomorphic to $\R^2$. Thus the previous two corollaries are not equivalent.


\subsubsection{Divisor spaces associated to toric varieties}

Let $T$ be a toric variety with $H_2(T)$ torsion free. In \cite{Gu}, Guest introduced a configuration space $\Div^T(M)$ such that $\mr{Hol}^*(\C P^1,T)$ is a collection of components of $\Div^T(\R^2)$. Guest proved homological stability for the components of the spaces $\Div^T(\R^2)$. Using the local-to-global homological stability theorem, we can conclude that the components of $\Div^T(\Sigma)$ have homological stability for $\Sigma$ a non-compact connected orientable surface. As far as we know, this is a new result. This is a first step towards generalizing Guest's theorem on the topology of holomorphic maps to a toric variety to the case of maps out of a higher genus surface.

\subsubsection{A mysterious example}
In the introduction we mentioned other examples of $\cdt$-algebras where homological stability is known. Often there is no geometric interpretation of their topological chiral homology, as is the case in the following example. If $\Sigma_{g,1}$ denotes a genus $g$ surface with one boundary component and $\mr{Diff}^\partial(-)$ the topological group of diffeomorphisms fixing the boundary topologized with the $C^\infty$-topology, then
\[\bigsqcup_{g \geq 0} B\mr{Diff}^\partial(\Sigma_{g,1})\]
is an $\cE^\mr{SO}_2$-algebra. Since not all Browder operations vanish (Theorem 2.5 of \cite{fiedsong}), this action does not extend to an action of the $\cE^\mr{SO}_3$-operad. Homological stability for this $\cE^\mr{SO}_2$-algebra was proven in \cite{harerstab} (see also \cite{wahlmcg}). Thus, the spaces 
\[\int_{\Sigma}^k \left(\bigsqcup_{g \geq 0} B\mr{Diff}^\partial(\Sigma_{g,1})\right)\]  have homological stability for $\Sigma$ any orientable connected non-compact surface. We know of no geometric interpretation of this result. Using \cite{sorenoscarstability}, one can construct similar higher-dimensional examples.

\subsection{Homological stability via $\cdt$-cell decompositions} \label{subseccells} Sometimes it is possible to identify $\cdt$-cell decompositions (or at least bound where $\cdt$-cells are attached) and then use Theorem \ref{thmmain} to deduce homological stability. In this subsection we will apply this idea to $\cdt$-algebras built out of completions of partial $\cdt$-algebras under certain assumptions about the partial monoid of connected components.


\subsubsection{Completing partial monoids} Any partial abelian monoid $C$ can be completed to an abelian monoid $c(C)$. Abstractly this is the left adjoint to the inclusion of abelian monoids into partial abelian monoids, and concretely it is given by setting $c(C)$ to be given by formal finite sums $\bc_1 \boxplus \cdots \boxplus \bc_k$ of elements of $C$ under the equivalence relation generated by the fact that $\boxplus$ is commutative and $\bc_1 \boxplus \bc_2 \sim \bc_1 + \bc_2$ if $\bc_1 + \bc_2$ is defined. In all cases we consider, $c(C) = \bN_0^d$. To distinguish between the monad $\bdt$ in the category of $C$-charged chain complexes and $c(C)$-charged chain complexes, we denote the former by $\bdtc$ and the latter by $\bdt$.

If $C \m D$ is a map of partial monoids, we get an induced functor from the category of $C$-charged algebras to the category of partial $D$-charged algebras. Thus, any $C$-charged algebra $A$ can be viewed as partial $c(C)$-charged algebra and then completed to form an actual $c(C)$-charged algebra. To describe this completion procedure concretely, we define a partial $c(C)$-charged algebra structure on  a $C$-charged algebra $A$ by setting 
\[A_p = B_p(\bdt \circ \iota,\bdtc,A)\] where $\iota$ is the inclusion of $\cat{Ch}_C$ into $\cat{Ch}_{c(C)}$.

\begin{definition}Let $A$ be a $C$-charged algebra, then the $c(C)$-completion $\bar{A}$ is defined to be the realization $|A_\bullet|$. This is a $c(C)$-charged algebra.\end{definition} 

Suppose $A$ is a $C$-charged algebra and we have an element $b_{N-1} \in A(k)_{N-1}$, then there are two things we can do. First, we could define $A \uplus e_N$ in $C$-charged algebras by attaching an $\cdt$-cell to $b_{N-1}$, view it as a partial $c(C)$-charged algebra and complete it to a $c(C)$-charged algebra $\overline{A \uplus e_N}$. Alternatively, we could view $A$ as a partial $c(C)$-charged algebra and complete it to a $c(C)$-charged algebra $\bar{A}$ and then attach an $\cdt$-cell to the image of $\mr{id} \otimes b_{N-1}$ in simplicial degree $0$ to get a $c(C)$-charged algebra $\bar{A} \uplus e_N$. A formal argument should tell us that $\overline{A \uplus e_N}$ and $\bar{A} \uplus e_N$ are the equivalent (they are the values of two functors that are derived left adjoint to the same functor). Since we have not set up the categorical framework needed to make this precise, we instead give a concrete simplicial argument.

\begin{proposition}Let $A$ be a $C$-charged algebra. The $c(C)$-charged algebras $\bar{A} \uplus e_N$ and $\overline{A \uplus e_N}$ are weakly equivalent. 

\label{completecommuteswithcells}

\end{proposition}


\begin{proof} As both $\bar{A} \uplus e_N$ and $\overline{A \uplus e_N}$ are completions of completions, they are naturally realizations of bisimplicial chain complexes. We will not compare them directly but instead compare both of them to a simplicial chain complex $\overline{A \oplus e_N}$ obtained as the completion of the partial $c(C)$-charged algebra $A \oplus e_N$. To define a partial $c(C)$-charged algebra structure on $A \oplus e_N$, we take $\mr{Comp}_1$ to be the subspace of $\bdt(A \oplus e_N)$ spanned by $\bdtc(A)$ and $\mr{id} \otimes e_N$, i.e. the only compositions that are defined involve elements of $A$ where the total charge is in $C$ or the identity acting on $e_N$. More precisely, $\overline{A \oplus e_N}$ is the geometric realization of the simplicial object $P_\bullet$ defined by 
\[[p] \mapsto P_p \coloneqq \bdt\left[(\bdtc)^p(A) \oplus (\mr{id}^p \otimes e_N)\right] \subset B_p(\bdt \circ \iota,\bdtc,A \oplus e_N)\]





We will now prove using an extra degeneracy argument that there is a weak equivalence
\[\bar{A} \uplus e_n \to \overline{A \oplus e_N}\]

As discussed above, the $c(C)$-charged algebra $\bar{A}$ is the realization of the simplicial object $B_\bullet(\bdt \circ \iota,\bdtc,A)$ and by Definition \ref{defpartialalgebracomp}, $c(C)$-charged algebra $\bar{A} \uplus e_N$ is the realization of the simplicial object $[[p] \mapsto \bdt \left((\bdt)^p(\bar A) \oplus (\mr{id}^p \otimes e_{N}) \right)$. Thus, $\bar{A} \uplus e_N$ is the realization of the bisimplicial object $P_{\bullet\bullet}^1$ given by 
\[[p,q] \mapsto P^1_{pq} \coloneqq \bdt \left[ (\bdt)^p \left(  B_q(\bdt \circ \iota,\bdtc,A) \right) \oplus (\mr{id}^p \otimes e_{N}) \right]  \] We will now define an augmentation map for its $p$-direction
\[\epsilon \colon P_{0q}^1= \bdt \left[\left(\bdt(\bdtc)^q(A)\right) \oplus e_{N} \right] \to P_q = \bdt \left[(\bdtc)^q ( A) \oplus (\mr{id}^q \otimes e_{N}) \right] \] There are maps
\[ \bdt     \left [\left(\bdt (\bdtc)^q (A)\right) \oplus e_{N} \right] \m  \bdt     \left [\bdt (\bdtc)^q (A \oplus e_{N}) \right] \to \bdt (\bdtc)^q  (A \oplus e_N) \]
the first of which is induced by the monad unit maps $1 \to \bdt$ and $1 \to \bdtc$, and the second of which is induced by the monad multiplication map $\bdt \bdt \m \bdt$. The image of the composition of these two maps is contained in $P_q$ and so we get a map  $P_{0q}^1 \m P_q$. The two maps $\epsilon \circ d_0, \epsilon \circ d_1 \colon P_{1q}^1 \m P_q$ agree as they only involve monad composition in the copies of $\bdt$, so this is an augmentation map. 

We thus get an augmented simplicial object $P^1_{\bullet q}$ by setting $P^1_{-1,q} = P_q$ and this has an extra degeneracy coming from the unit map $1 \m \bdt$. More specifically, to define the extra degeneracy $P^1_{pq} \m P^1_{p+1,q}$ for $p \geq 0$, we insert this identity to the left of the $\bdt$ factor coming from $B_q(\bdt \circ \iota,\bdtc,A)$ in \[P^1_{pq} = \bdt\left[(\bdt)^p\left( B_q(\bdt \circ \iota,\bdtc,A) \right) \oplus (\mr{id}^p \otimes e_{N})\right]  \] Similarly, the unit gives us a map $P^1_{-1,q} = P_q \m P^1_{0,q}$.  Thus the augmentation induces a weak equivalence $|P^1_{\bullet q}| \to P_q$. Since level wise weak equivalences realize to weak equivalence, we conclude that $\bar{A} \uplus e_N \simeq |P^1_{\bullet \bullet}| \to |P_\bullet| \simeq \overline{A \oplus e_N}$ is a weak equivalence.


We next prove that $\overline{A \uplus e_N} \simeq \overline{A \oplus e_N}$. One can use an argument as above, but we shall quote a result from the literature. The algebra $\overline{A \uplus e_N}$ is the realization of the bisimplicial object \[[p,q] \mapsto B_p\left[\bdt \circ \iota,\bdtc, \bdtc\left(   (\bdtc)^q (  A) \oplus (\mr{id}^q \otimes e_{N}) \right) \right] \] Theorem 9.10.3 of \cite{M} implies that this is equivalent to the realization of the simplicial object with $q$ simplices  given by $[q] \mapsto \bdt \left((\bdtc)^q (  A) \oplus (\mr{id}^q \otimes e_{N}) \right)$, which is exactly $P_\bullet$. 

Thus $\bar{A} \uplus e_N$ and $\overline{A \uplus e_N}$ are both weakly equivalent to $\overline{A \oplus e_N} = |P_\bullet|$ and hence to each other.


\end{proof}

Using this proposition, we prove homological stability for $\cdt$-algebras formed by completing certain partial algebras. Let $[d]=\{0,1,\ldots,d\} \subset \bN_0$ viewed as a partial monoid with addition. Completions of partial algebras with monoid of components $[d]$ were studied before in \cite{kupersmillercompletions}. 

\begin{proposition}

\label{completionwithfunction}
 Let $\rho \colon \R_{\geq 0} \m \R_{\geq 0}$ be a strictly increasing function satisfying $\rho(a+b) \leq \rho(a)+\rho(b)$ for $a,b \in \bR_{\geq 0}$, $\rho(s) \leq s/2$ for all $s \in \R_{\geq 0}$ and $\rho^{-1}(\N_0) \subset \N_0$. Let $A$ be a $[d]$-charged algebra with $t \colon A_k \m A_{k+1}$ a $\rho(k)$-equivalence for $k <d$ and let $\bar A$ denote its completion as a $\N_0$-charged algebra. For any oriented connected non-compact $n$-dimensional manifold $A$, we have that
\[t \colon \int^k_M \bar{A} \to \int^{k+1}_M \bar{A}\]
is a $\rho(k)$-equivalence.\end{proposition}

\begin{proof}
Mimicking the proof of Proposition \ref{propstabimpliesbound} and stopping after attaching $\cdt$-cells in charge $d$, we can construct a cellular $\cdt$-algebra $B$ and a map $B \m \bar A$ such that $B_k \m \bar{A}_k$ is an equivalence for all $k \leq d$ and all $\cdt$-cells of $B$ are attached above $\rho$ and in charge $\leq d$. By the implication (iii) $\Rightarrow$ (ii) of Theorem \ref{thmmainrange}, $t \colon \int^k_M B \to \int^{k+1}_M B$ is a $\rho(k)$-equivalence. Let $B^{\leq d}$ denote the partial $\cdt$-algebra formed by only considering components of charge $\leq d$. By an iterated application of Proposition \ref{completecommuteswithcells}, we have that $\overline{B^{\leq d}} \simeq B$. The map $B^{\leq d} \m A$ is an equivalence. Thus, $\overline{B^{\leq d}} \to B \m \bar A$ is an equivalence. By Theorem \ref{thmmainrange} we have that $t \colon \int^k_M \bar{A} \to \int^{k+1}_M \bar{A}$ is a $\rho(k)$-equivalence.
\end{proof}

This proposition has implications for completions (in the category of $\N_0$-charged algebras) of $[d]$-charged algebras without making any assumptions on their homology. Such a result was proven in Theorem 1.1 of \cite{kupersmillercompletions} but with a worse range. The techniques of \cite{kupersmillercompletions} are completely different than those used here and follow the traditional approach to proving homological stability introduced by Quillen. 

\begin{corollary}
\label{completionsnoassumption}
Let $A$ be a $[d]$-charged algebra, $\bar A$ the completion in the category of $\N_0$-charged algebras and $M$ a $\theta$-framed  connected non-compact $n$-dimensional manifold, then we have that
\[t \colon \int^k_M \bar{A} \to \int^{k+1}_M \bar{A}\]
is a $\rho(k)$-equivalence for $\rho(k) = \min(k/2,k/d)$.\end{corollary}

\begin{proof} We note that for any $[d]$-charged algebra $A$, $t \colon A_k \m A_{k+1}$ is a $\rho(k)$-equivalence for $k <d$.\end{proof}

In \cite{kupersmillercompletions}, it was remarked that the range established in Corollary \ref{completionsnoassumption} is optimal. 

\subsubsection{Bounded symmetric powers revisited}

\label{boundedrevisited}
We can apply Proposition \ref{completionwithfunction} to the case of bounded symmetric powers.  

\begin{corollary}
Let $M$ be a oriented path-connected non-compact manifold of dimension at least $2$. The stabilization map $t\colon \Sym^{\leq d}_k(M) \m \Sym^{\leq d}_{k+1}(M)$ induces an isomorphism on $H_*$ for $* \leq k/2$. 
\end{corollary}

\begin{proof}
For $k \leq d$, $\Sym^{\leq d}_k(\R^n)$ is contractible, so Proposition \ref{completionwithfunction} applies with $\rho(k)=k/2$.
\end{proof}

This result greatly improves on the range of $* \leq \frac{k}{2d}$ established in \cite{kupersmillercompletions} which was the only integral homological stability range known for bounded symmetric powers in high dimensions. We note that the arguments of \cite{Y} and Section \ref{boundedsymmetricpowerssubsectionlocaltoglobal} all use that $M$ is two-dimensional in a crucial way -- either using that $\Sym(\C)$ is a manifold or that $\Omega^2 S^3 \times \Z \simeq \Omega^2 S^2$ -- and thus those techniques cannot be used to achieve this range in higher dimensions. A rational range with slope $1$ was proven in \cite{kupersmillertran} using the fact that $\Sym(M)$ is an orbifold. Using the techniques of this paper, we can reprove this rational range and show that it also holds with $\Z[1/2]$-coefficients. 

Using Theorem \ref{thmcompact}, we also conclude that the scanning map \[s\colon \Sym^{\leq d}_k(M) \m \mr{Map}^c_k\left(M,B^n \Sym^{\leq d}(\R^n)\right)\] induces an isomorphism on homology in the range $* \leq k/2$ for $M$ not necessarily non-compact, but framed. This improves on the best previously known range, a range of $* \leq \frac{k-d}{2d}$ established in \cite{kupersmillercompletions}.

Recall that in Remark \ref{cellsintop}, we described cell atachments  in $\cat{Top}$. In the case $n=2$ and trivial tangential structure, it is actually possible to completely describe an $E_n$-cell structure on $\Sym^{\leq d}(\R^n)$. This result was known to S\o ren Galatius and Jacob Lurie, and we learned it from them through personal communication. It is related to the results in Section 5 of \cite{lurierotation}, which also discusses $BU$ as a $\cE^\mr{pt}_2$-algebra.

\begin{proposition}There exists an $\cE^\mr{pt}_2$-cell decomposition of $\mr{Sym}^{\leq d}(\R^2)$ with exactly one $\cE^\mr{pt}_2$-cell in dimensions $0,2,\ldots,2(d-1)$. The cell of dimension $2i$ is attached in charge $i+1$. \label{LuExample} \end{proposition}
\begin{proof}We proceed by induction on $d$. The fact that we are working with trivial framings implies that the statement is true for $d=1$. By Lemma 2.7 of \cite{GKY2}, $\mr{Sym}^{\leq d}_{d+1}(\R^2) \simeq S^{2d-1}$. Attaching an $\cE^\mr{pt}_2$-cell to this sphere makes $\mr{Sym}^{\leq d}_{d+1}(\R^2)$ contractible and hence equivalent to $\mr{Sym}^{\leq d+1}_{d+1}(\R^2)$. For $k \leq d$, $\mr{Sym}^{\leq d+1}_{k}(\R^2)$ and $\mr{Sym}^{\leq d}_{k}(\R^2)$ are both contractible and hence homotopy equivalent. Using Proposition \ref{completecommuteswithcells}, we see that $\mr{Sym}^{\leq d+1}(\R^2)$ is equivalent to $\mr{Sym}^{\leq d}_{k}(\R^2)$ with a $(2d-2)$-dimensional $\cE^\mr{pt}_2$-cell attached to the sphere representing a generator of $\pi_{2d-3}(\mr{Sym}^{\leq d}_{d+1}(\R^2))$. This proves the inductive step. \end{proof}


In the next proposition we will show that an $\cdtt$-cell decomposition of an $\cdtt$-algebra $A$ induces an ordinary cell decomposition of $B^n A$. A similar result should hold for $\cdt$-algebras and cells in spaces with an action of $\mr{Bun}^\theta(T \R^n,T \R^n) \simeq \Omega W$.

\begin{proposition}\label{cellofBn} \label{propbncell} Let $X$ be an $\cdtt$-algebra in $\cat{Top}$, then we have that 
	\[B^n(X \uplus D^N) \simeq (B^n X) \cup_{S^{n+N-1}} D^{n+N}\]\end{proposition}

\begin{proof}The $\cdtt$-algebra $X \uplus D^N$ is defined as the realization $|\bdtt(\mr{Comp}_\bullet)|$ with
\[\mr{Comp}_p = (\bdtt)^p(X) \cup_{S^{N-1}} D^N\]
where the attaching map $S^{N-1} \to (\bdtt)^p(X)$ is induced from the map $S^{N-1} \to X$ by applying the unit natural transformation of the monad $\bdtt$ $p$ times. One model for $B^n$ is given by $|B_\bullet(\Sigma^n_+ ,\bdtt,-)|$, we get that $B^n(X \uplus D^N)$ is the realization of the bisimplicial object 
\[[p,q] \mapsto B_p(\Sigma^n_+,\bdtt,\bdtt(\mr{Comp}_{q}))\]


\noindent Let us first realize in the $B_\bullet$ direction. Theorem 9.10.3 of \cite{M} implies that $B(\Sigma^n_+ ,\bdtt,\bdtt Y) \simeq \Sigma^n Y_+$ for any $\cdtt$-algebra $Y$. This implies that $B^n(X \uplus D^N)$ is weakly equivalent to the realization of $\Sigma^n (\mr{Comp}_\bullet)_+$. 

As in Definition \ref{defpartialalgebracomp}, remark that $\mr{Comp}_p$ can be obtained as the push out
\[\xymatrix{S^{N-1} \ar[r] \ar[d] & D^n \ar[d] \\
(\bdtt)^p(X) \ar[r] & \mr{Comp}_p}\]
Since $\Sigma^n_+-$ preserves push outs, $\Sigma^n_+(\mr{Comp}_p)_+$ can be obtained as a push out as well:
\[\xymatrix{S^{n+N-1} \ar[r] \ar[d] & D^{n+N} \ar[d] \\
\Sigma^n_+(\bdtt)^p(X) \ar[r] & \Sigma^n (\mr{Comp}_p)_+}\]
Since geometric realization preserves colimits, we get that $|\Sigma^n(\mr{Comp}_\bullet)_+|$ is obtained as a push out
\[\xymatrix{S^{n+N-1} \ar[r] \ar[d] & D^{n+N} \ar[d] \\
B^n X \ar[r] & |\Sigma^n(\mr{Comp}_\bullet)_+| \simeq B^n(X \uplus D^N)}\] \end{proof}

This theorem also holds in $\cat{sSet}$ and $\cat{Ch}$, and in the latter case attaching an $\cdtt$-cell of dimension $N$ to $A$ amounts to attaching an ordinary $N$-cell to the cotangent space $L(A) \simeq B^n A[-n]$ (see \cite{Fr1} for more discussion of the cotangent space). Since the number of generators of homology serves as a lower bound for the number of cells, we get the following lemma.
 
\begin{lemma}
Let $X$ be an $\cdtt$-algebra in $\cat{Top}$. Suppose $X$ is equivalent to a cellular $\cdtt$-algebra with $k$ $\cdtt$-cells of dimension $i$. Then $k$ is at least as large as the number of generators of $\tilde H_{n+i}(B^n X)$ as an abelian group. 
\label{lowerbound}
\end{lemma} 

\begin{example}It is well known that $B^2 \mr{Sym}^{\leq d}(\R^2) \simeq \C P^d$, and the lemma thus implies that any $\cE^\mr{pt}_2$-cell decomposition of $\mr{Sym}^{\leq d}(\R^2)$ must have at least one $\cE^\mr{pt}_2$-cell in dimension $0,2,\ldots, 2(d-1)$. Proposition \ref{LuExample} says that this lower bound can be achieved for $\mr{Sym}^{\leq d}(\R^2)$. \end{example}

\subsubsection{Divisor spaces revisited}

Using ideas similar to those in Section \ref{boundedrevisited}, one can prove homological stability for the divisor spaces of \cite{Se} \cite{Gu} and the spaces of coprime polynomials considered in \cite{GKYcoprime}, as well as higher dimensional versions of these spaces.  For simplicity of notation, we restrict our attention to the case relevant to holomorphic maps to $\C P^1$, the case of $\Div^1(M)$.

\begin{definition}
Let $A$ and $B$ be $\cdt$-algebras in $\cat{Top}$. Let $A \vee B$ be the partial $\cdt$-algebra with underlying space $A \vee B$ and composition only defined if all elements are either in $A$ or all in $B$.
\end{definition}

This construction is relevant as $\Div^1(\R^n) \simeq \overline{\N_0 \vee \N_0}$. Proposition \ref{completecommuteswithcells} in this case says the following.

\begin{proposition}
\label{wedge}
Let $A$ and $B$ be $\N_0$-charged algebras with choice of $\cdt$-cell structures. There is an $\cdt$-cell structure on $\overline{A \vee B}$ with the following properties: 

\begin{enumerate}[(i)]
\item There are no $\cdt$-cells attached in charge $(k,j)$ unless either $k$ or $j$ is equal to zero.
\item There is a bijection between the $\cdt$-cells of homological degree $i$ of charge $k$ of $A$ with the $\cdt$-cells of  homological degree $i$ of charge $(k,0)$ of $\overline{A \vee B}$.
\item There is a bijection between the $\cdt$-cells of homological degree $i$ of charge $k$ of $B$ with the $\cdt$-cells of  homological degree $i$ of charge $(0,k)$ of $\overline{A \vee B}$.
\end{enumerate}

\end{proposition}

We now prove homological stability for $\Div^1(M)$.

\begin{corollary}
Let $M$ be a $\theta$-framed connected  non-compact manifold of dimension at least $2$. The stabilization map $t\colon \Div^{1}_{(k,j)}(M) \m \Div^{1}_{(k+1,j+1)}(M)$ induces an isomorphism on $H_*$ for $* \leq \min(k/2,j/2)$. 
\end{corollary}

\begin{proof}
Since the components of $\N_0$ have homological stability, by the implication (i) $\Rightarrow$ (iii) of Theorem \ref{thmmainrange}, we conclude that $C_*(\N_0)$ has an $\cdt$-cell decomposition with cells of charge $k$ having homological degree at least $k/2$. Using Proposition \ref{wedge}, we conclude that $C_*(\overline{\N_0 \vee \N_0})$ has an $\cdt$-cell decomposition with cells of charge $(k,j)$ having homological degree at least $\max(k/2,j/2)$. Thus, by the implication (iii) $\Rightarrow$ (ii) of Theorem \ref{thmmainrange}, $t\colon\int_M^{k,j} \overline{\N_0 \vee \N_0} \m \int_M^{k+1,j+1} \overline{\N_0 \vee \N_0}$ induces an isomorphism in homology in the range $* \leq \min(k/2,j/2)$. The claim now follows since $\Div^1(M) \simeq \int_M  \overline{\N_0 \vee \N_0}$. 
\end{proof}

This result is new in the case $M$ has dimension greater than $2$. We note that this proposition combined with non-abelian Poincar\'e duality \cite{Mi2} gives a proof of Segal's result that $\mr{Hol}_k^*(\C P^1, \C P^1) \m \mr{Map}_k^*(\C P^1, \C P^1)$ is a homology equivalence in a range tending to infinity with $k$ whose only input consists of the basic properties of holomorphic functions and the fact that $\N_0$ has homological stability.  

Gravesen \cite{gravesen} and Boyer, Hurtubise, Mann and Milgram \cite{BMHM} gave configuration space models for spaces of based holomorphic maps from $\bC P^1$ into generalized flag varieties. We believe that our techniques can also be applied to those cases.

\input{arXivappendixJ2}

\bibliographystyle{unsrt}
\bibliography{cell}

\end{document}

%% file: arXivappendixJ2.tex
\appendix 

\section{Homological stability for configuration spaces with labels} \label{appendix} In this appendix we prove a homological stability result that is known to some experts but did not appear in the literature before the first preprint of this paper was made available. It has since appeared as Proposition B.3 of \cite{federicomartin}. It can be proven by a straightforward extension of any of the standard techniques, and for the benefit of the reader we will not only give one proof, but sketch three additional proofs.

Recall that if $\pi \colon E \to M$ is a Serre fibration with fiber $F$, the \emph{configuration space of $k$ ordered particles with labels in the fibration $\pi$}, denoted $F_k^\pi(M)$, is given by the following subspace of $E^k$:
	\[F_k^\pi(M) \coloneqq \left\{ (e_1,\ldots,e_k) \middle|
	\pi(e_i) \neq \pi(e_j) \text{ for $i \neq j$}	\right\}\]
	
\noindent The \emph{configuration space of particles with labels in the fibration $\pi$}, denoted $C_k^\pi(M)$, is defined to be the quotient $F_k^\pi(M)/\mathfrak{S}_{k}$ where $\fS_k$ acts diagonally.

\begin{theorem}\label{thmbunstab} Let $M$ be an oriented smooth connected non-compact $n$-dimensional manifold and let $\pi \colon E \to M$ be a Serre fibration with path-connected fiber $F$.
\begin{enumerate}[(i)]
\item The stabilization map $t \colon H_*(C^\pi_k(M)) \to H_*(C^\pi_{k+1}(M))$ is a $k/2$-equivalence. 
\item If $n \geq 3$, the stabilization map $t \colon H_*(C^\pi_k(M);\bZ[1/2]) \to H_*(C^\pi_{k+1}(M);\bZ[1/2])$ is a $k$-equivalence. 
\end{enumerate}
\end{theorem}

Before proving this result, we will prove a corollary about colored configurations. These are defined for $\bk = (k_1,\ldots,k_d)$ by $C^\pi_\bk(M) \coloneqq F^\pi_k(M)/\fS_\bk$ where $\fS_\bk = \prod_i \fS_{k_i} \subset \fS_{\sum k_i}$. This has $d$ different stabilization maps $t_i$, each introducing a particle of color $i$ at infinity.

\begin{corollary}\label{corbunmulh}
Under the assumptions of Theorem \ref{thmbunstab}, we have the following.
\begin{enumerate}[(i)]
\item The stabilization map $t_i: H_*(C^\pi_\bk(M)) \to H_*(C^\pi_{\bk+\be_i}(M))$ is a $k_i/2$-equivalence. 
\item If $n \geq 3$, the stablization map $t_i: H_*(C^\pi_\bk(M);\bZ[1/2]) \to H_*(C^\pi_{\bk+\be_i}(M);\bZ[1/2])$ is a $k_i$-equivalence. 
\end{enumerate}
\end{corollary}

\begin{proof}We prove case (i), case (ii) being similar. Forgetting all colors except the $i$th one gives rise to a fiber sequence
\[C_{k_i}^\pi(M\setminus \{k - k_i\text{ points}\}) \to C^\pi_\bk(M) \to C^\pi_{\bk \setminus k_i}(M)\]
The stabilization map induces a map of Serre spectral sequences which is homotopic to the identity on the base and is given by the stabilization map
\[t: H_*(C^\pi_{k_i}(M\setminus \{k-k_i \text{ points}\})) \to H_*(C^\pi_{k_i+1}(M\setminus \{k-k_i \text{ points}\}))\]
on the fiber. Combining Theorem \ref{thmbunstab} with spectral sequence comparison gives the desired result.
\end{proof}

We will give one complete proof of Theorem \ref{thmbunstab} and three sketches, with various additional assumptions and conclusions:
\begin{enumerate}[(i)]
\item Using homological stability with twisted coefficients one can get the theorem with a $(k/2-1)$-equivalence, which can be upgraded to a $k/2$-equivalence if we assume that each $H_i(F)$ is finitely generated.
\item Using a semisimplicial resolution by arcs, one can also get the theorem as stated.
\item Using transfers and finite generation of the limit, one can get the result without an explicit range if we assume that $F$ is homotopy equivalent to a finite CW-complex.
\item Using $FI\#$-modules one can prove representation stability for the rational cohomology, which implies Theorem A.1 with rational coefficients, as long as we assume that $M$ is simply-connected and each $H^i(F;\bQ)$ is finite-dimensional.
\end{enumerate}

Note that our primary application has $F \simeq SO(n)$, so that $F$ satisfies the assumptions in (i), (iii) and (iv).

\begin{proof}[Proof using homological stability with twisted coefficients] We start with a complete proof in the case that $F$ has the property that each homology group $H_i(F)$ is finitely generated. Otherwise this proof only gives a $(k/2-1)$-equivalence. We will only do the case (i), noting that the input for case (ii) is Proposition A.2 of \cite{federicomartin} or follows from applying Palmer's techniques in \cite{palmertwisted} to the main result of \cite{kupersmillerimprov}.

This proof is the shortest we are aware of and uses Palmer's result on stability for the homology of configuration spaces with coefficients in certain local systems \cite{palmertwisted}. His Corollary 1.6 in particular implies that the map 
\[t \colon H_*(C_k(M);\cH_q(F^k;\bF)) \to H_*(C_{k+1}(M);\cH_q(F^{k+1};\bF)) \]
induced by a stabilization map, which adds a new point to the configuration with some choice of label in the fiber, is an isomorphism in the range $* \leq \frac{k-q}{2}$ where $\bF$ is some field. Here $\cH_q(F^k;\bF)$ is the local system of abelian groups with fiber above a point $(m_1,\ldots,m_k) \in C_k(M)$ given by $H_q(F_{m_1} \times \cdots \times F_{m_k};\bF)$. To apply this result we consider the Serre spectral sequence for homology with $\bF$-coefficients associated to the Serre fibration
\[F^k \to C^\pi_k(M) \to C_k(M)\]

\noindent This has $E^2$-page given by $E^2_{pq} = H_p(C_k(M);\cH_q(F^k;\bF))$ and the map $t: C_k(M;E) \to C_{k+1}(M;E)$ induces a map of spectral sequences given on the $E^2$-page by the map considered by Palmer. The stability with $\bF$-coefficients for $C^\pi_k(M)$ now follows by a spectral sequence comparison.

For integral stability it is enough to prove stability with all field coefficients. The $(k/2-1)$-equivalence can now be proven by considering the long exact sequence in homology associated to the short exact sequences of abelian groups: 
 \[0 \to \bZ_{p} \to \bZ_{p^{i+1}} \to \bZ_{p^i} \to 0\]
\[0 \to \bZ \to \bQ \to \bQ/\bZ \cong \bigoplus_p \bZ_{p^{\infty}} \to 0\]
and applying the five-lemma several times. Using a transfer argument, one can show that the stabilization map is always injective in homology so in fact it induces an isomorphism for $*=k/2-1$ as well. The loss of $1$ in the range happens when applying the second long exact sequence, so we still get a $k/2$-equivalence for $\bZ/p^m \bZ$-coefficients for all primes $p$ and all $m \geq 1$.

To improve this to a $k/2$-equivalence when $F$ has finitely generated homology  groups, first make the assumption that $M$ has a finite handle decomposition. In that case $H_*(C^\pi_k(M))$ can be shown to be finitely generated using the Serre spectral sequence used before. Then note that the previous proof implies that 
\[\mr{Tor}_\bZ(H_{i-1}(C^\pi_k(M)),\bZ/p^m \bZ) \to \mr{Tor}_\bZ(H_{i-1}(C^\pi_{k+1}(M)),\bZ/p^m \bZ)\]
is an isomorphism for $i \leq k/2$, all primes $p$ and all $m \geq 1$. By the universal coefficient theorem we then conclude that $H_i(C^\pi_k(M)) \otimes \bZ/p^m \bZ \to H_i(C^\pi_{k+1}(M)) \otimes \bZ/p^m \bZ$ is an isomorphism for all $i \leq k/2$, all primes $p$ and all $m \geq 1$. But if $f: A \to B$ is a map of finitely generated abelian groups such that $f \otimes \bZ/p^m \bZ: A \otimes \bZ/p^m \bZ \to B \otimes \bZ/p^m \bZ$ is an isomorphism for all primes $p$ and $m \geq 1$, then $f$ is an isomorphism. Finally, remark that any $M$ can  be exhausted by open submanifolds $M_j$ with finite handle decompositions compatibly with the stabilization map.
\end{proof}

\begin{proof}[Sketch of proof using semisimplicial resolutions by arcs] The second proof uses a semisimplicial resolution by arcs and is a modification of the proof in \cite{RW}. In contrast to the previous result, it requires no additional assumptions on $M$ or $F$. We will only give the proof in the case (i). Case (ii) can be established by modifying the arguments of Section 3.5 of \cite{MW1} to this context.

The idea is to consider the following semisimplicial space $X_\bullet(k)$ with augmentation to $C^\pi_k(M)$. One picks an embedding $\phi$ of $(0,1)$ into $\partial M$. The space $X_p(k)$ of $p$-simplices consists of pairs $(x,((\gamma_0,\eta_0),\ldots,(\gamma_p,\eta_p)))$. Here we have that $x$ is an element of $C_k(M)$, the $\gamma_i$ are disjoint embeddings $[0,1] \to \bar{M}$ such that $\gamma_i(0) = \phi(t_i)$ with $t_0 < \ldots < t_p$ and $\gamma_i(1)$ is a point in our configuration, and the $\eta_i$ are lifts $[0,1] \to E$ of $\gamma_i$ such that $\eta_i(1)$ hits the label of $\gamma_i(1)$. The embeddings are topologized in the $C^\infty$-topology, the lifts in the compact-open topology and the face maps each forget one of the arcs and their lift.

We now make two claims, whose proofs we will sketch later: (i) the map $\epsilon: ||X_\bullet(k)|| \to C^\pi_k(M)$ induced by the augmentation is $(k-1)$-connected and (ii) $X_p(k) \simeq C^\pi_{k-p-1}(M) \times F^{p+1}$. The stabilization map $t \colon C^\pi_k(M) \to C^\pi_{k+1}(M)$ lifts to a semisimplicial map $t_\bullet \colon X_\bullet(k) \to X_\bullet(k+1)$. This gives a spectral sequence with $E^1$-page $E^1_{pq} = H_q(X_p(k+1),X_p(k)) = \bigoplus_{i+j=q} H_i(C^\pi_{k-p}(M),C^\pi_{k-p-1}(M)) \otimes H_j(F^{p+1})$ (modulo K\"unneth), converging to the relative homotopy groups $H_p(C^\pi_{k+1}(M),C^\pi_k(M))$ in a range. A spectral sequence argument as in Section 6 of \cite{RW} then finishes the proof.

Let's sketch the justification for the two claims. Claim (i) is proven by noting that for a fixed configuration $x \in C^\pi_k(M)$ the fibers of $\epsilon$ are $(k-1)$-connected and $\epsilon$ is a microfibration. Here the argument diverges for $\dim M = 2$ and $\dim M > 2$. In the case $\dim M = 2$ one leverages the connectivity of the complex $A(S;\Delta^n,\Lambda_n)$ in Section 7 of \cite{hatcherwahl} and the fact that the space of representatives of an isotopy class is contractible. Hatcher and Wahl also discuss the generalization to higher dimensions. In the case $\dim M > 2$ one leverages the connectivity of the complex of injective words on the elements of the configuration, general position and an argument that under certain conditions allows one to deduce that a semisimplicial space is highly connected if its underlying semisimplicial set is. Such an argument is given in the proof of Theorem 4.6 of \cite{sorenoscarstability}. The second claim is proven by moving the $p+1$ of the points of the configuration along the paths $\gamma_i$, while moving the label along $\eta_i$.\end{proof}

\begin{proof}[Sketch of proof using transfers and finite generation of the limit] This proof is a straightforward adaption of the proof in \cite{mcduffposneg}, which in fact claims the theorem on page 93. As mentioned before, we will need to assume that $F$ is homotopy equivalent to a finite CW-complex, and we do not get an explicit range. As in the end of the first proof, we can assume that $M$ has a finite handle decomposition.

For all $j \leq k$ there is a transfer map $\tau_{k,j}: H_*(C^\pi_k(M)) \to H_*(C^\pi_j(M))$ induced by summing over all ways of forgetting $j-k$ of the $k$ labeled particles. This satisfies the equation $\tau_{k+1,j} \circ t = \tau_{k,j} + t \circ \tau_{k,j-1}$. By Lemma 2 of \cite{Do} we have that $t: H_*(C^\pi_k(M)) \to H_*(C^\pi_{k+1}(M))$ is the inclusion of a direct summand. 

So we are done if we can prove that the limiting homology $\mr{colim}_{k \to \infty} H_*(C^\pi_k(M))$ is finitely-generated in each degree. There is a scanning map $C^\pi_k(M) \to \Gamma^c_k(M,\dot{TM} \wedge E_+)$, where $\Gamma^c_k$ are the degree $k$ compactly-supported sections and $\dot{TM} \wedge E_+$ is the fiberwise smash product of the fiberwise one-point compactification of the tangent space with the disjoint union of $E$ and a fiberwise disjoint basepoint. There are stabilization maps between the components of this section space by introducing a collection of compactly supported degree 1 or -1 sections at the boundary. These are all weak equivalences. 

Remark that limiting homology $\mr{colim}_{k \to \infty} H_*(C^\pi_k(M))$ is given by the homology of $\tilde{C}^\pi(M) = \mr{hocolim}_{k \to \infty} C^\pi_k(M)$, where we take the homotopy colimit over $t$. The scanning map then induces a map
\[\tilde{C}^\pi(M) = \mr{hocolim}_{k \to \infty} C^\pi_k(M) \to \mr{hocolim}_{k \to \infty} \Gamma^c_k(M;\dot{TM} \wedge E_+) \simeq \Gamma^c_0(M; \dot{TM} \wedge E_+)\]
where the latter homotopy equivalence follows from the fact that all stabilization maps for the section spaces are weak equivalence. That map is a homology equivalence follows by induction over handles, a result which can be found in \cite{manthorpetillmann}.

That $\Gamma^c_0(M; \dot{TM} \wedge E_+)$ has finitely-generated homology follows from the case $M = \bR^n$, in which one considers the homology of $\Omega^k \Sigma^n F_+$ for $k \leq n$ using iterated Eilenberg-Moore spectral sequences, and a handle induction using the fact that a compact manifold has a finite handle decomposition.\end{proof}

\begin{proof}[Sketch of rational proof using $FI\#$-modules] The following argument works if one is willing to use rational coefficients and assume that $M$ is simply-connected and each rational cohomology group of the fiber $F$ of $E$ is finitely-dimensional. Again, without loss of generality we can assume that $M$ has a finite handle decomposition. 
	
We will actually prove a stronger result, uniform representation stability for $H^i(F^\pi_k(M))$. With rational coefficients we have an isomorphism $H^i(C^\pi_k(M);\bQ) = H^i(F^\pi_k(M);\bQ)^{\fS_k}$. One can prove statements about the representations of $\fS_k$ that show up in $H^i(F^\pi_k(M);\bQ)$ using the theory of $FI$- and $FI\#$-modules as developed by Church, Farb and Ellenberg in \cite{churchellenbergfarb}, whose notation we will freely use. Using this theory we will prove that $\{H^i(F^\pi_k(M);\bQ)\}_{k \in \bN}$ is uniformly representable stable with stable range $\geq 2i$, because then in particular the number of trivial $\fS_k$-representations appearing in $H^i(F^\pi_k(M);\bQ)$ stabilizes for $i \leq k/2$ respectively.

If $M$ is non-compact and path-connected, then the collection $\{H^i(F^\pi_k(M);\bQ)\}_{k \in \bN}$ assembles into an FI\#-module. For an FI\#-module the stable range for uniform representation stability is twice its weight by Corollary 2.59 of \cite{churchellenbergfarb}.

In Theorem 4.7 they prove that $\{H^i(F_k(M);\bQ)\}_{k \in \bN}$ has weight $\leq i$. The weight of $\{H^i(F^k;\bQ)\}_{k \in \bN}$ is $\leq i$, which follows from Proposition 2.51 and Proposition 2.72. The Serre spectral sequence for the fibration
\[F^k \to F^\pi_k(M) \to F_k(M)\]
is then a spectral sequence of FI\#-modules with $E^{pq}_2 = H^p(F_k(M);\bQ) \otimes H^q(F^k;\bQ)$. Weight is additive under tensor products by Proposition 2.61 and by definition preserved under extensions and subquotients. This has weight $\leq p+q$ and hence so has $H^{p+q}(F^\pi_k(M);\bQ)$.
\end{proof}

\begin{remark}
	The last proof can be adapted to prove a similar stabilization result for $H^*(F^\pi_k(M);\bQ)$ when $M$ is closed, by working with FI-modules and keeping track of both weight and stability degree. We have no use for this, since for a closed manifold $M$ there are no maps relating $\int^k_M A$ and $\int^{k+1}_M A$.
\end{remark}

\begin{remark}
One can also prove an integral representation stability results without assuming any conditions on $M$ or $F$ by modifying the argument in Section 3.2 of \cite{MW1}. 
\end{remark}

%% file: SansGroupFall2016J7arXiv.bbl
\def\cprime{$'$}
\begin{thebibliography}{10}

\bibitem{Ar}
V.~I. Arnol{\cprime}d.
\newblock Certain topological invariants of algebraic functions.
\newblock {\em Trudy Moskov. Mat. Ob\v s\v c.}, 21:27--46, 1970.

\bibitem{Se}
Graeme Segal.
\newblock The topology of spaces of rational functions.
\newblock {\em Acta Math.}, 143(1-2):39--72, 1979.

\bibitem{Srod}
Norman~E. Steenrod.
\newblock Cohomology operations, and obstructions to extending continuous
  functions.
\newblock {\em Advances in Math.}, 8:371--416, 1972.

\bibitem{Y}
Kohhei Yamaguchi.
\newblock Configuration space models for spaces of maps from a {R}iemann
  surface to complex projective space.
\newblock {\em Publ. Res. Inst. Math. Sci.}, 39(3):535--543, 2003.

\bibitem{kupersmillertran}
Alexander Kupers, Jeremy Miller, and TriThang Tran.
\newblock Homological stability for symmetric complements.
\newblock {\em To appear in Transactions of the American Mathematical Society},
  preprint 2014.

\bibitem{RW}
Oscar Randal-Williams.
\newblock Homological stability for unordered configuration spaces.
\newblock {\em Quarterly Journal of Mathematics}, (64 (1)):303--326, 2013.

\bibitem{kupersmillercompletions}
Alexander Kupers and Jeremy Miller.
\newblock Homological stability for topological chiral homology of completions.
\newblock {\em preprint}, 2013.
\newblock \url{http://arxiv.org/abs/1311.5203v1}.

\bibitem{EVW}
Jordan~S. Ellenberg, Akshay Venkatesh, and Craig Westerland.
\newblock Homological stability for {H}urwitz spaces and the {C}ohen-{L}enstra
  conjecture over function fields.
\newblock {\em Ann. of Math. (2)}, 183(3):729--786, 2016.

\bibitem{harerstab}
John~L. Harer.
\newblock Stability of the homology of the mapping class groups of orientable
  surfaces.
\newblock {\em Ann. of Math. (2)}, 121(2):215--249, 1985.

\bibitem{wahlmcg}
Nathalie Wahl.
\newblock Homological stability for mapping class groups of surfaces.
\newblock In {\em Handbook of moduli. {V}ol. {III}}, volume~26 of {\em Adv.
  Lect. Math. (ALM)}, pages 547--583. Int. Press, Somerville, MA, 2013.

\bibitem{sorenoscarstability}
S{\o}ren Galatius and Oscar Randal-Williams.
\newblock Homological stability for moduli spaces of high dimensional
  manifolds.
\newblock \url{http://arxiv.org/abs/1203.6830}, 2012.

\bibitem{palmerthesis}
Martin Palmer.
\newblock Configuration spaces and homological stability, 2012.
\newblock Oxford University thesis.

\bibitem{instantonmoduli}
C.~P. Boyer, J.~C. Hurtubise, B.~M. Mann, and R.~J. Milgram.
\newblock The topology of instanton moduli spaces. {I}. {T}he {A}tiyah-{J}ones
  conjecture.
\newblock {\em Ann. of Math. (2)}, 137(3):561--609, 1993.

\bibitem{gravesen}
Jens Gravesen.
\newblock On the topology of spaces of holomorphic maps.
\newblock {\em Acta Math.}, 162(3-4):247--286, 1989.

\bibitem{BMHM}
C.~P. Boyer, J.~C. Hurtubise, B.~M. Mann, and R.~J. Milgram.
\newblock The topology of the space of rational maps into generalized flag
  manifolds.
\newblock {\em Acta Math.}, 173(1):61--101, 1994.

\bibitem{Gu}
Martin~A. Guest.
\newblock The topology of the space of rational curves on a toric variety.
\newblock {\em Acta Math.}, 174(1):119--145, 1995.

\bibitem{An}
Ricardo Andrade.
\newblock {\em From manifolds to invariants of {E}n-algebras}.
\newblock ProQuest LLC, Ann Arbor, MI, 2010.
\newblock Thesis (Ph.D.)--Massachusetts Institute of Technology.

\bibitem{Fr2}
David Ayala and John Francis.
\newblock Factorization homology of topological manifolds.
\newblock {\em J. Topol.}, 8(4):1045--1084, 2015.

\bibitem{ginottradlerzeinalian}
Gr{\'e}gory Ginot, Thomas Tradler, and Mahmoud Zeinalian.
\newblock Higher {H}ochschild homology, topological chiral homology and
  factorization algebras.
\newblock {\em Comm. Math. Phys.}, 326(3):635--686, 2014.

\bibitem{lurieha}
Jacob Lurie.
\newblock {\em Higher Algebra}.
\newblock 2011.
\newblock November 2014 version,
  \url{http://www.math.harvard.edu/~lurie/papers/higheralgebra.pdf}.

\bibitem{Sa}
Paolo Salvatore.
\newblock Configuration spaces with summable labels.
\newblock In {\em Cohomological methods in homotopy theory ({B}ellaterra,
  1998)}, volume 196 of {\em Progr. Math.}, pages 375--395. Birkh\"auser,
  Basel, 2001.

\bibitem{kupersmillerimprov}
Alexander Kupers and Jeremy Miller.
\newblock Improved homological stability for configuration spaces after
  inverting $2$.
\newblock {\em To appear in Homology, Homotopy and Applications}, preprint
  2014.
\newblock \url{http://arxiv.org/abs/1405.4441v3}.

\bibitem{Mi2}
Jeremy Miller.
\newblock Nonabelian {P}oincar\'e duality after stabilizing.
\newblock {\em Trans. Amer. Math. Soc.}, 367(3):1969--1991, 2015.

\bibitem{Ch}
Thomas Church.
\newblock Homological stability for configuration spaces of manifolds.
\newblock {\em Invent. Math.}, 188(2):465--504, 2012.

\bibitem{M}
J.~P. May.
\newblock {\em The geometry of iterated loop spaces}.
\newblock Springer-Verlag, Berlin-New York, 1972.
\newblock Lecture Notes in Mathematics, Vol. 271.

\bibitem{fressebook}
Benoit Fresse.
\newblock {\em Modules over operads and functors}, volume 1967 of {\em Lecture
  Notes in Mathematics}.
\newblock Springer-Verlag, Berlin, 2009.

\bibitem{munkres}
James~R. Munkres.
\newblock {\em Topology}.
\newblock Prentice-Hall Inc., second edition, 2000.

\bibitem{luriehtt}
Jacob Lurie.
\newblock {\em Higher topos theory}, volume 170 of {\em Annals of Mathematics
  Studies}.
\newblock Princeton University Press, Princeton, NJ, 2009.

\bibitem{segalcats}
Graeme Segal.
\newblock Categories and cohomology theories.
\newblock {\em Topology}, 13:293--312, 1974.

\bibitem{Se2}
Graeme Segal.
\newblock Classifying spaces and spectral sequences.
\newblock {\em Inst. Hautes \'Etudes Sci. Publ. Math.}, (34):105--112, 1968.

\bibitem{mccleary}
John McCleary.
\newblock {\em A user's guide to spectral sequences}, volume~58 of {\em
  Cambridge Studies in Advanced Mathematics}.
\newblock Cambridge University Press, Cambridge, second edition, 2001.

\bibitem{Mc1}
Dusa McDuff.
\newblock Configuration spaces of positive and negative particles.
\newblock {\em Topology}, 14:91--107, 1975.

\bibitem{bergermoerdijk}
Clemens Berger and Ieke Moerdijk.
\newblock Resolution of coloured operads and rectification of homotopy
  algebras.
\newblock In {\em Categories in algebra, geometry and mathematical physics},
  volume 431 of {\em Contemp. Math.}, pages 31--58. Amer. Math. Soc.,
  Providence, RI, 2007.

\bibitem{horel}
Geoffrey Horel.
\newblock Factorization homology and calculus \`a la {K}ontsevich {S}oibelman.
\newblock {\em preprint}, 2013.
\newblock \url{http://arxiv.org/abs/1307.0322v2}.

\bibitem{federicomartin}
Federico Cantero and Martin Palmer.
\newblock On homological stability for configuration spaces on closed
  background manifolds.
\newblock {\em preprint}, 2014.

\bibitem{ayalafranciskoszul}
David Ayala and John Francis.
\newblock Poincar\'e/{K}oszul duality.
\newblock {\em preprint}, 2014.
\newblock \url{http://arxiv.org/abs/1409.2478v2}.

\bibitem{MSe}
D.~McDuff and G.~Segal.
\newblock Homology fibrations and the ``group-completion'' theorem.
\newblock {\em Invent. Math.}, 31(3):279--284, 1975/76.

\bibitem{weissclassify}
Michael Weiss.
\newblock What does the classifying space of a category classify?
\newblock {\em Homology Homotopy Appl.}, 7(1):185--195, 2005.

\bibitem{Mi3}
Jeremy Miller.
\newblock The topology of the space of {J}-holomorphic maps to ${CP}^2$.
\newblock {\em To appear in Geometry \& Topology}, preprint 2012.

\bibitem{VW}
Ravi Vakil and Melanie~Matchett Wood.
\newblock Discriminants in the {G}rothendieck ring.
\newblock {\em Duke Math. J.}, 164(6):1139--1185, 2015.

\bibitem{CCMM1}
F.~R. Cohen, R.~L. Cohen, B.~M. Mann, and R.~J. Milgram.
\newblock The topology of rational functions and divisors of surfaces.
\newblock {\em Acta Math.}, 166(3-4):163--221, 1991.

\bibitem{V1}
V.~A. Vassiliev.
\newblock {\em Complements of discriminants of smooth maps: topology and
  applications}, volume~98 of {\em Translations of Mathematical Monographs}.
\newblock American Mathematical Society, Providence, RI, 1992.
\newblock Translated from the Russian by B. Goldfarb.

\bibitem{KSstring}
Sadok Kallel and Paolo Salvatore.
\newblock Rational maps and string topology.
\newblock {\em Geom. Topol.}, 10:1579--1606 (electronic), 2006.

\bibitem{fiedsong}
Zbigniew Fiedorowicz and Yongjin Song.
\newblock The braid structure of mapping class groups.
\newblock {\em Sci. Bull. Josai Univ.}, (Special issue 2):21--29, 1997.
\newblock Surgery and geometric topology (Sakado, 1996).

\bibitem{lurierotation}
Jacob Lurie.
\newblock Rotation invariance in algebraic {K}-theory.
\newblock {\em preprint}, 2014.
\newblock \url{http://www.math.harvard.edu/~lurie/papers/Waldhausen.pdf}.

\bibitem{GKY2}
M.~A. Guest, A.~Kozlowski, and K.~Yamaguchi.
\newblock Stable splitting of the space of polynomials with roots of bounded
  multiplicity.
\newblock {\em J. Math. Kyoto Univ.}, 38(2):351--366, 1998.

\bibitem{Fr1}
John Francis.
\newblock The tangent complex and {H}ochschild cohomology of {${E}_n$}-rings.
\newblock {\em Compos. Math.}, 149(3):430--480, 2013.

\bibitem{GKYcoprime}
M.~A. Guest, A.~Kozlowski, and K.~Yamaguchi.
\newblock The topology of spaces of coprime polynomials.
\newblock {\em Math. Z.}, 217(3):435--446, 1994.

\bibitem{palmertwisted}
Martin Palmer.
\newblock Twisted homology of configuration spaces.
\newblock {\em preprint}, 2013.
\newblock \url{http://arxiv.org/abs/1308.4397v1}.

\bibitem{MW1}
Jeremy Miller and Jennifer C.~H. Wilson.
\newblock Higher order representation stability and ordered configuration
  spaces of manifolds.
\newblock {\em Preprint}, 2016.
\newblock \url{https://arxiv.org/abs/1611.01920}.

\bibitem{hatcherwahl}
Allen Hatcher and Nathalie Wahl.
\newblock Stabilization for mapping class groups of 3-manifolds.
\newblock {\em Duke Math. J.}, 155(2):205--269, 2010.

\bibitem{Do}
Albrecht Dold.
\newblock Decomposition theorems for {$S(n)$}-complexes.
\newblock {\em Ann. of Math. (2)}, 75:8--16, 1962.

\bibitem{manthorpetillmann}
Richard Manthorpe and Ulrike Tillmann.
\newblock Tubular configurations: equivariant scanning and splitting.
\newblock {\em preprint}, 2013.
\newblock \url{http://arxiv.org/abs/1307.5669v1}.

\bibitem{churchellenbergfarb}
Thomas Church, Jordan~S. Ellenberg, and Benson Farb.
\newblock {FI}-modules: a new approach to stability for
  ${S}_n$-representations.
\newblock {\em preprint}, 2012.
\newblock \url{http://arxiv.org/abs/1204.4533v2}.

\end{thebibliography}
